\pdfoutput=1
\RequirePackage{ifpdf}
\ifpdf 
\documentclass[pdftex]{sigma}
\else
\documentclass{sigma}
\fi

\numberwithin{equation}{section}

\newtheorem{Theorem}{Theorem}[section]
\newtheorem*{Theorem*}{Theorem}
\newtheorem{Corollary}[Theorem]{Corollary}
\newtheorem{Lemma}[Theorem]{Lemma}
 { \theoremstyle{definition}
\newtheorem{Definition}[Theorem]{Definition}

\newtheorem{Remark}[Theorem]{Remark}
\newtheorem{Question}[Theorem]{Question}
}

\newcommand{\C}{{\mathbb C}}

\newcommand{\Z}{{\mathbb Z}}

\newcommand{\N}{{\mathbb N}}
\newcommand{\D}{{\mathbb D}}
\newcommand{\T}{{\mathbb T}}

\newcommand{\al}{\alpha}
\newcommand{\be}{\beta}

\newcommand{\ga}{\gamma}
\newcommand{\Ga}{\Gamma}

\newcommand{\la}{\lambda}
\newcommand{\ep}{\varepsilon}
\newcommand{\de}{\delta}

\newcommand{\Om}{\Omega}
\newcommand{\ze}{\zeta}

\newcommand{\di}{\displaystyle}

\newcommand{\ii}{{\rm i}}
\newcommand{\dd}{{\rm d}}

\newcommand{\qasq}{\qquad \text{as} \quad}
\newcommand{\qandq}{\qquad \text{and} \qquad}

\newcommand{\red}[1]{{\color{red} #1}}
\newcommand{\blue}[1]{{\color{blue} #1}}

\newcommand{\violet}[1]{{\color{violet} #1}}
\newcommand{\teal}[1]{{\color{teal} #1}}

\usepackage{dutchcal}
\usepackage[mathscr]{euscript}

\usepackage{tikz}
\usepackage{pict2e}
\usetikzlibrary{decorations.markings}

\tikzset{middlearrow/.style={
		decoration={markings,
			mark= at position 0.6 with {\arrow{#1}},
		},
		postaction={decorate}
	}
}\tikzset{middlearrow/.style={
		decoration={markings,
			mark= at position 0.6 with {\arrow{#1}},
		},
		postaction={decorate}
	}
}

\tikzset{->-/.style={decoration={
			markings,
			mark=at position #1 with {\arrow{latex}}},postaction={decorate}}}

\tikzset{-<-/.style={decoration={
			markings,
			mark=at position #1 with {\arrowreversed{latex}}},postaction={decorate}}}

\begin{document}

\allowdisplaybreaks

\renewcommand{\thefootnote}{}

\newcommand{\arXivNumber}{2309.14695}

\renewcommand{\PaperNumber}{062}

\FirstPageHeading

\ShortArticleName{Strong Szeg\H{o} Limit Theorems for Multi-Bordered}

\ArticleName{Strong Szeg\H{o} Limit Theorems for Multi-Bordered,\\ Framed, and Multi-Framed Toeplitz Determinants\footnote{This paper is a~contribution to the Special Issue on Evolution Equations, Exactly Solvable Models and Random Matrices in honor of Alexander Its' 70th birthday. The~full collection is available at \href{https://www.emis.de/journals/SIGMA/Its.html}{https://www.emis.de/journals/SIGMA/Its.html}}}

\Author{Roozbeh GHARAKHLOO}

\AuthorNameForHeading{R.~Gharakhloo}

\Address{Mathematics Department, University of California, Santa Cruz, CA~95064, USA}
\Email{\href{mailto:roozbeh@ucsc.edu}{roozbeh@ucsc.edu}}
\URLaddress{\url{https://roozbehgharakhloo.wordpress.com/}}

\ArticleDates{Received September 27, 2023, in final form June 20, 2024; Published online July 11, 2024}

\Abstract{This work provides the general framework for obtaining strong Szeg\H{o} limit theorems for multi-bordered, semi-framed, framed, and multi-framed Toeplitz determinants, extending the results of Basor et al.~(2022) beyond the (single) bordered Toeplitz case. For the two-bordered and also the semi-framed Toeplitz determinants, we compute the strong Szeg\H{o} limit theorems associated with certain classes of symbols, and for the $k$-bordered (${k \geq 3}$), framed, and multi-framed Toeplitz determinants we demonstrate the recursive fashion offered by the Dodgson condensation identities via which strong Szeg\H{o} limit theorems can be obtained. One instance of appearance of semi-framed Toeplitz determinants is in calculations related to the entanglement entropy for disjoint subsystems in the XX spin chain (Brightmore et al.~(2020) and Jin--Korepin (2011)). In addition, in the recent work Gharakhloo and Liechty (2024) and in an unpublished work of Professor Nicholas Witte, such determinants have found relevance respectively in the study of ensembles of nonintersecting paths and in the study of off-diagonal correlations of the anisotropic square-lattice Ising model. Besides the intrinsic mathematical interest in these structured determinants, the aforementioned applications have further motivated the study of the present work.}

\Keywords{strong Szeg\H{o} theorem; bordered Toeplitz determinants; framed Toeplitz determinants; Riemann--Hilbert problem; asymptotic analysis}

\Classification{15B05; 30E15; 30E25}

\begin{flushright}
\begin{minipage}{60mm}
\it Dedicated to Alexander Its\\ on the occasion of his 70th birthday
\end{minipage}
\end{flushright}

\tableofcontents

\renewcommand{\thefootnote}{\arabic{footnote}}
\setcounter{footnote}{0}

\section{Introduction}
For $\phi \in L^1(\T)$, denote the $n \times n$ (pure) Toeplitz matrix by $T_n[\phi]$ and its determinant by
\begin{equation}\label{ToeplitzDet}
	D_n[\phi] = \det T_n[\phi] \equiv \underset{0 \leq j,k\leq n-1}{\det} \{ \phi_{j-k} \},
\end{equation}
where \begin{equation*}
	\phi_j = \int_{\T} \phi(\ze)\ze^{-j} \frac{\dd \ze}{2 \pi \ii \ze}, \qquad j \in \Z,
\end{equation*}
is the $j$-th Fourier coefficient of $\phi$, and $\T$ denotes the unit circle oriented in the counterclockwise direction.

This work is mainly concerned with two distinct structural deformations of Toeplitz determinants, being (multi-)bordered and (multi-)framed Toeplitz determinants. A bordered Toeplitz determinant with the \textit{bulk symbol} $\phi$ and the \textit{border symbol} $\psi$ is of the form
\begin{gather}\label{btd0}
	D^{B}_{n}[\phi; \psi] := \det \begin{bmatrix}
		\phi_0& \phi_1 & \cdots & \phi_{n-2} & \psi_{n-1}\\
		\phi_{-1}& \phi_0 & \cdots & \phi_{n-3}&\psi_{n-2} \\
		\vdots & \vdots & \ddots & \vdots & \vdots\\
		\phi_{1-n} & \phi_{2-n} & \cdots & \phi_{-1}&\psi_{0}
	\end{bmatrix}, \qquad n \geq 2.
\end{gather}
While a framed-Toeplitz determinant with the bulk symbol $\phi$ and the border symbols $\xi$, $\psi$, $\eta$, and $\ga$ is of the form
\begin{gather}\label{framed intro}
	\mathscr{F}_{n}\left[\phi; \xi, \psi, \eta, \ga; \boldsymbol{a}_4\right] := \det \begin{bmatrix}
		a_1 & \xi_{0} & \xi_{1} & \cdots & \xi_{n-3} & a_2 \\
		\ga_{0} &	\phi_0& \phi_{-1} & \cdots & \phi_{-n+3} & \psi_{n-3} \\
		\ga_{1} &	\phi_{1}& \phi_0 & \cdots & \phi_{-n+4} & \psi_{n-4} \\
		\vdots &	\vdots & \vdots & \ddots & \vdots & \vdots \\
		\ga_{n-3} &	\phi_{n-3} & \phi_{n-4} & \cdots & \phi_{0} & \psi_{0} \\
		a_4 &	\eta_{n-3} & \eta_{n-4} & \cdots & \eta_{0} & a_3
	\end{bmatrix}, \qquad n \geq 3,
\end{gather}
where $\boldsymbol{a}_4$ denotes the ordered set $\{a_1, a_2, a_3, a_4\}$, and $a_k$'s are arbitrary complex numbers, ${k=1,\dots,4}$.

Throughout this work, by the boldfaced $\boldsymbol{{f}}_m$ we also denote $\{f_\ell\}^{m}_{\ell=1}$, $f_{\ell} \in L^1(\T)$, a vector of border symbols. For a bulk symbol $\phi \in L^1(\T)$, $m \in \N$ and $n \geq m+1$, define the $n\times n$ \textit{multi-bordered} Toeplitz matrix generated by $\phi$ and $\boldsymbol{\psi}_m$ as the matrix whose last $m$ columns are generated respectively by the Fourier coefficients of $\psi_1, \dots, \psi_m$ and the remaining rectangular~${n\times (n-m)}$ submatrix is of Toeplitz structure, generated by the symbol $\phi$. More precisely,
\begin{gather}\label{btd}
	D^B_n[\phi;\boldsymbol{\psi}_m] := \det \begin{bmatrix}
		\phi_0& \phi_{1} & \cdots & \phi_{n-m-1} & \psi_{1,n-1} & \cdots & \psi_{m,n-1} \\
		\phi_{-1}& \phi_0 & \cdots & \phi_{n-m-2} & \psi_{1,n-2} & \cdots & \psi_{m,n-2} \\
		\vdots & \vdots & \vdots & \vdots & \vdots & \cdots & \vdots \\
		\phi_{-n+1} & \phi_{-n+2} & \cdots & \phi_{-m} & \psi_{1,0} & \cdots & \psi_{m,0}
	\end{bmatrix},
\end{gather}
where $f_j$ is again the $j$-th Fourier coefficient of $f \in \{ \phi, \psi_1, \dots, \psi_m \}$. Instead of the notation~${D^B_n[\phi;\boldsymbol{\psi}_m]}$ in \eqref{btd}, we occasionally use $D^B_n[\phi;\psi_1, \psi_2, \dots, \psi_m]$ to avoid confusion in the order of the borders.\footnote{For example, in \eqref{E11}--\eqref{E44}.}

Similarly we can define a \textit{multi-framed} Toeplitz determinant $ \mathscr{F}^{(m)}_{n}[\phi; \boldsymbol{\xi}_m, \boldsymbol{\psi}_m, \boldsymbol{\eta}_m, \boldsymbol{\ga}_m; \boldsymbol{a}_{4m}]$, with $\boldsymbol{a}_{4m} = \{ a_1, \dots, a_{4m} \}$. For $m \in \N$ (the number of frames) and $n \geq 2m+1$, the $n\times n$ multi-framed Toeplitz matrix generated by $\phi$, $\boldsymbol{\xi}_m$, $\boldsymbol{\psi}_m$, $\boldsymbol{\eta}_m$ and $\boldsymbol{\ga}_m$ is comprised of an $(n-2m)\times(n-2m)$ Toeplitz matrix generated by $\phi$ in addition to the $m$ frames surrounding it generated by the Fourier coefficients of the symbols in the sets $\boldsymbol{\xi}_m$, $\boldsymbol{\psi}_m$, $\boldsymbol{\eta}_m$ and $\boldsymbol{\ga}_m$, where the four border symbols $\xi_{\ell}$, $\psi_{\ell}$, $\eta_{\ell}$, and $\ga_{\ell}$ generate the entries on the $\ell$-th frame, $1\leq \ell \leq m$.

For example, here we show the determinant $\mathscr{F}^{(3)}_{n}\left[\phi; \boldsymbol{\xi}_3, \boldsymbol{\psi}_3, \boldsymbol{\eta}_3, \boldsymbol{\ga}_3; \boldsymbol{a}_{12}\right]$ with colored entries for easier interpretation:
\begin{gather}\label{MM}
	\det \left(\!\begin{matrix}
		\red{a_9} & \teal{\xi_{3,n-3}}\! & \teal{\xi_{3,n-4}}\! & \teal{\xi_{3,n-5}} & \teal{\xi_{3,n-6}} & \cdots & \teal{\xi_{3,2}} & \teal{\xi_{3,1}} & \teal{\xi_{3,0}} & \red{a_{10} }\\
		\teal{\ga_{3,n-3}}\! &	\red{a_5} & \violet{\xi_{2,n-5}}\! & \violet{\xi_{2,n-6}} & \violet{\xi_{2,n-7}} & \cdots & \violet{\xi_{2,1}} & \violet{\xi_{2,0}} & \red{a_6} & \teal{\psi_{3,0}} \\
		\teal{\ga_{3,n-4}}\! & \violet{\ga_{2,n-5}}\! &	\red{a_1} & \blue{\xi_{1,n-7}} & \blue{\xi_{1,n-8}} & \cdots & \blue{\xi_{1,0}} & \red{a_2} & \violet{\psi_{2,0}} & \teal{\psi_{3,1}} \\
		\teal{ \ga_{3,n-5}}\! &	\violet{\ga_{2,n-6}}\! &	\blue{\ga_{1,n-7}}\! & \phi_0& \phi_{-1} & \cdots & \phi_{-n+7} & \blue{\psi_{1,0}} & \violet{\psi_{2,1}} & \teal{\psi_{3,2}} \\
		\teal{\ga_{3,n-6}}\! &	\violet{\ga_{2,n-7}}\! &	\blue{\ga_{1,n-8}}\! &	\phi_{1}& \phi_0 & \cdots & \phi_{-n+8} & \blue{\psi_{1,1}} & \violet{\psi_{2,2}} & \teal{\psi_{3,3}} \\
		\vdots &	\vdots &	\vdots &	\vdots & \vdots & \ddots & \vdots & \vdots & \vdots & \vdots \\
		\teal{\ga_{3,2}} &	\violet{\ga_{2,1}} &	\blue{\ga_{1,0}} &	\phi_{n-7} & \phi_{n-8} & \cdots & \phi_{0} & \blue{\psi_{1,n-7}} & \violet{\psi_{2,n-6}} & \teal{\psi_{3,n-5}} \\
		\teal{\ga_{3,1}} &	\violet{\ga_{2,0}} &	\red{a_4} &	\blue{\eta_{1,n-7}}\! & \blue{\eta_{1,n-8}} & \cdots & \blue{\eta_{1,0}} & \red{a_3} & \violet{\psi_{2,n-5}} & \teal{\psi_{3,n-4}} \\
		\teal{\ga_{3,0}} & \red{a_8} & \violet{\eta_{2,n-5}}\! & \violet{\eta_{2,n-6}} & \violet{\eta_{2,n-7}} & \cdots & \violet{\eta_{2,1}} & \violet{\eta_{2,0}} & \red{a_7} & \teal{\psi_{3,n-3}} \\		
		\red{a_{12}} & \teal{\eta_{3,n-3}}\! & \teal{\eta_{3,n-4}}\! & \teal{\eta_{3,n-5}} & \teal{\eta_{3,n-6}} & \cdots & \teal{\eta_{3,2}} & \teal{\eta_{3,1}} & \teal{\eta_{3,0}} & \red{a_{11}} \\
	\end{matrix}\!\right).\!\!\!
\end{gather}

Our approach to conducting asymptotic analysis on multi-bordered, framed, and multi-framed Toeplitz determinants involves rewriting these structured determinants in terms of others with tractable asymptotics. Such reductions to simpler structured determinants result from utilizing the \textit{Dodgson condensation identity},\footnote{Also known as the \textit{Desnanot--Jacobi} identity or the \textit{Sylvester determinant} identity.} which we occasionally abbreviate as DCI (see~\cite{Abeles,Bressoud,Fulmek-Kleber, GW} and references therein). Let $\boldsymbol{\mathscr{M}}$ be an $n \times n$ matrix. By
\[
 \mathscr{M} \left\lbrace \begin{matrix} j_1& j_2& \cdots & j_{\ell} \\ k_1& k_2& \cdots & k_{\ell} \end{matrix} \right\rbrace,
 \]
we mean the determinant of the $(n-\ell)\times(n-\ell)$ matrix obtained from $\boldsymbol{\mathscr{M}}$ by removing the rows~$j_i$ and the columns $k_i$, $1\leq i \leq \ell$. Although the order of writing the row and column indices is immaterial for this definition, in this work we prefer to respect the order of indices, for example we prefer to write \[ \mathscr{M} \left\lbrace \begin{matrix} 3& 5 \\ 1& 4 \end{matrix} \right\rbrace, \qquad \text{and not} \qquad \mathscr{M} \left\lbrace \begin{matrix} 5 & 3 \\ 1& 4 \end{matrix} \right\rbrace, \qquad \mbox{or} \qquad \mathscr{M} \left\lbrace \begin{matrix} 3& 5 \\ 4& 1 \end{matrix} \right\rbrace, \qquad \mbox{or} \qquad \mathscr{M} \left\lbrace \begin{matrix} 5 & 3 \\ 4& 1 \end{matrix} \right\rbrace, \]
although all of these are the same determinant. Let $j_1 < j_2$ and $k_1 < k_2$. The Dodgson condensation identity reads
\begin{equation}\label{DODGSON}
	\mathscr{M} \cdot \mathscr{M}\left\lbrace \begin{matrix} j_1 & j_2 \\ k_1& k_2 \end{matrix} \right\rbrace = \mathscr{M}\left\lbrace \begin{matrix} j_1 \\ k_1 \end{matrix} \right\rbrace \cdot \mathscr{M}\left\lbrace \begin{matrix} j_2 \\ k_2 \end{matrix} \right\rbrace - \mathscr{M}\left\lbrace \begin{matrix} j_1 \\ k_2 \end{matrix} \right\rbrace \cdot \mathscr{M}\left\lbrace \begin{matrix} j_2 \\ k_1 \end{matrix} \right\rbrace.
\end{equation}

Speaking of reductions to simpler structured determinants through one or multiple applications of the DCI, it turns out that the multi-bordered Toeplitz determinants can be reduced to the pure and bordered Toeplitz determinants \eqref{ToeplitzDet} and \eqref{btd0}, while the framed and multi-framed Toeplitz determinants can be expressed in terms of pure Toeplitz determinants and what we refer to as \textit{semi-framed} Toeplitz determinants. These are determinants like
\begin{gather}
	\det \begin{bmatrix}
		\phi_0& \phi_{-1} & \cdots & \phi_{-n+2} & \psi_{0} \\
		\phi_{1}& \phi_0 & \cdots & \phi_{-n+3} & \psi_{1} \\
		\vdots & \vdots & \ddots & \vdots & \vdots \\
		\phi_{n-2} & \phi_{n-3} & \cdots & \phi_{0} & \psi_{n-2} \\
		\eta_{n-2} & \eta_{n-3} & \cdots & \eta_{0} & a
	\end{bmatrix},\qquad \text{or}\nonumber\\
\det \begin{bmatrix}
		\phi_0& \phi_{-1} & \cdots & \phi_{-n+2} & \psi_{n-2} \\
		\phi_{1}& \phi_0 & \cdots & \phi_{-n+3} & \psi_{n-3} \\
		\vdots & \vdots & \ddots & \vdots & \vdots \\
		\phi_{n-2} & \phi_{n-3} & \cdots & \phi_{0} & \psi_{0} \\
		\eta_{0} & \eta_{1} & \cdots & \eta_{n-2} & a
	\end{bmatrix},\label{half-framed intro00}
\end{gather}
for $\phi,\psi,\eta \in L^1(\T)$ and a parameter $a \in \C$, where $f_j$'s are the Fourier coefficients of ${f \in \{ \phi, \psi, \eta \}}$. In the sequel, we will denote the determinants in \eqref{half-framed intro00} by $\mathscr{H}_n[\phi;\psi,\eta;a]$ and $\mathscr{L}_n[\phi;\psi,\eta;a]$, respectively.

\begin{Remark}
	Regarding the other choices for positioning the Fourier coefficients of $\psi$ and~$\eta$ in the last column and the last row, we will introduce two other (related) semi-framed Toeplitz determinants in Section~\ref{section framed} denoted by $\mathscr{G}_n[\phi;\psi,\eta;a]$ and $\mathscr{E}_n[\phi;\psi,\eta;a]$. It turns out that such different placements of Fourier coefficients does in fact affect the leading order behavior of the asymptotics, as the size of the determinant grows to infinity (see Theorems \ref{thm semi-framed rationals intro} and~\ref{thm semi-framed rationals . phi intro}).
\end{Remark}

At this juncture, we would like to highlight two primary questions:
\begin{Question}\label{q1}
 Given what we discussed above about the reductions of more complex structures to the bordered and semi-framed Toeplitz determinants, are the large-size asymptotics of these \textit{simpler} structured detereminants indeed tractable?
\end{Question}
	
\begin{Question}
\label{q2} Why is it significant to delve into the asymptotic behavior of (multi-)bordered and (mul\-ti-)framed Toeplitz determinants?
\end{Question}

Before tackling these inquiries, it is worthwhile to place them in a broader perspective. The asymptotic properties of the more classical structured determinants, such as Toeplitz \cite{BS1,BS,CIK,CK,DIK, Fahs, KV,Krasovsky1}, Hankel \cite{BGM,C,CG,DIK,ItsKrasovsky,Krasovsky}, and Toeplitz+Hankel \cite{BR,BE,BE2,Chelkak,GI}, have been extensively and successfully explored primarily via operator theoretic and Riemann--Hilbert methods. These well-established asymptotic characteristics are recognized for their connection to fundamental questions spanning diverse fields, particularly in random matrix theory and mathematical physics.

For this work, it is useful to recall the existing theory for the pure Toeplitz determinants. The asymptotic behavior of Toeplitz determinants can be described by the strong Szeg\H{o} limit theorem \cite{BS,Sz,WIDOMBlock}, which is formulated as
\begin{gather*}
	D_{n}[\phi] \sim G[\phi]^{n} E[\phi], \qquad n\to\infty,
\end{gather*}
where the terms $G[\phi]$ and $E[\phi]$ are defined by
\begin{equation}\label{G and E}
	G[\phi] = \exp ([\log \phi]_{0}) \qquad \text{and}\qquad E[\phi] = \exp \bigg( \sum_{k \geq 1} k[\log \phi ]_{k}[\log \phi]_{-k} \bigg).
\end{equation}
This theorem holds true when the function $\phi$ is suitably smooth, does not vanish on the unit circle, and possesses a winding number of zero. We refer to \cite{DIK1} for a comprehensive survey of the strong Szeg\H{o} limit theorem, including an intriguing account of its historical developments.

Now we address Question~\ref{q1} mentioned above starting with bordered Toeplitz determinants. Recently, in \cite{BEGIL} it was demonstrated that an analogous strong Szeg\H{o} limit theorem holds true for the bordered Toeplitz determinants
\begin{gather*}
	D^{B}_{n}[\phi; \psi] \sim G[\phi]^{n} E[\phi]\,F[\phi; \psi], \qquad n\to\infty,
\end{gather*}
where $F[\phi; \psi]$ is a constant described in Theorems \ref{main thm} and \ref{thm 1.2} below. Theorem \ref{main thm} below discusses the asymptotics of $D^{B}_{n}[\phi; \psi]$, where $\psi$ is of the form
\begin{equation}\label{general psi}
	\psi(z) = q_1(z) \phi(z) + q_2(z),
\end{equation}where
\begin{equation}\label{q1 q2}
	q_1(z) = a_0+a_1z+\frac{b_0}{z}+\sum^{m}_{j=1}\frac{b_j z}{z-c_j},\qquad q_2(z) = \hat{a}_0+\hat{a}_1z+\frac{\hat{b}_0}{z}+
	\sum_{j=1}^m \frac{\hat{b}_j}{z-c_j},
\end{equation}
all parameters are complex and the $c_j$ are nonzero and do not lie on the unit circle. In fact, this form of the border symbol was considered in \cite{BEGIL} as an inspiration from the two-dimensional Ising model.\footnote{For a precise description of the notations used below regarding the Ising model we refer to the introduction of \cite{BEGIL}. For further details about the two-dimensional Ising model we refer to the classical book of McCoy and Wu~\cite{McCoy-Wu} and also the survey of Deift, Its and Krasovsky \cite{DIK1}.} It was first established in 1987 by Au-Yang and Perk \cite{YP} that the next-to-diagonal two point correlation function is in fact the bordered Toeplitz determinant
\[
	\langle \sigma_{0,0}\sigma_{N-1,N} \rangle = D^B_N\big[\widehat{\phi}; \widehat{\psi}\big],
\]
with
\begin{equation}\label{hat psi}
	\widehat{\phi}(z) = \sqrt{\frac{1-k^{-1}z^{-1}}{1-k^{-1}z}},\qquad	\widehat{\psi}(z)= \frac{C_v z\widehat{\phi}(z)+C_h}{S_v(z-c_*)},
\end{equation} where $k$, $C_v$, $C_h$, $S_v$, and $c_*$ are all physical parameters of the model. In the context of the low-temperature two-dimensional Ising model, the analogue of the strong Szeg\H{o} limit theorem for bordered Toeplitz determinants (see Theorem \ref{main thm} below) was later used in \cite{BEGIL} to extract the leading and subleading terms of the \textit{long-range-order} along the next-to-diagonal direction and comparisons with the diagonal direction were made. It was concluded that although the bordered Toeplitz determinant which defines the next-to-diagonal correlation function depends on the horizontal and vertical coupling constants, its leading order asymptotics
does not. More interestingly, it was established that the sensitivity to the horizontal and vertical parameters is reflected in the second-order term of the asymptotic expansion. Before recalling the strong Szeg\H{o} limit theorems established in \cite{BEGIL}, let us define a class of symbols we are mostly concerned with in this work.

\begin{Definition}
Throughout the paper, we will occasionally refer to a symbol as \textit{Szeg\H{o}-type}, if (a)~it is $C^{\infty}$ and nonzero on the unit circle, (b)~has no winding number, and (c)~admits an analytic continuation in some neighborhood of the unit circle.
\end{Definition}

\begin{Theorem}[\cite{BEGIL}]\label{main thm}
	Let $D^{B}_{n}[\phi; \psi]$ be the bordered Toeplitz determinant with $\psi=q_1 \phi + q_2$ given by \eqref{general psi} and \eqref{q1 q2}, and $\phi$ of Szeg\H{o}-type. Then, the following asymptotic behavior of $D^{B}_{n}[\phi; \psi]$ as $n \to \infty$ takes place
	\[
			D^B_{n}[\phi ; \psi ] =
			G[\phi]^{n} E[\phi](F[\phi;\psi] + O({\rm e}^{-\mathfrak{c}n})),
		\]
		where $G[\phi]$ and $E[\phi]$ are given by \eqref{G and E},
		\begin{gather}
			F[\phi;\psi] =	
			a_0+ b_0 [\log \phi]_{1} +
		\sum^m_{j=1 \atop 0<|c_j|<1} b_j \frac{\al(c_j)}{\alpha(0)}\nonumber\\
			\phantom{F[\phi;\psi] =}{}+	
			\frac{1}{\al(0)}\Bigg(
			\hat{a}_0-\hat{a}_1[\log \phi]_{-1} -
			\sum^{m}_{j=1 \atop |c_j|>1 }\frac{\hat{b}_j}{c_j} \al(c_j) \Bigg),\label{ConstantF2}\\
			\al(z):= \exp \left[ \frac{1}{2 \pi {\rm i} } \int_{\T} \frac{\ln(\phi(\tau))}{\tau-z}{\rm d}\tau \right],\label{al}
		\end{gather}
		and $\mathfrak{c}$ is some positive constant.
	\end{Theorem}

	Transitioning beyond the category of symbols linked to the Ising model, for a broader range of border symbols, a different version of the strong Szeg\H{o} limit theorem was proven in \cite{BEGIL}. In this instance, the only requirement was that $\psi$ has an analytic continuation in some neighborhood of the unit circle.
	\begin{Theorem}[\cite{BEGIL}]\label{thm 1.2}
		Let $\psi(z)$ be a function which admits an analytic continuation in a neighborhood of the unit circle, and let
		$\phi$ be of Szeg\H{o}-type. Denote by $\phi_{\pm}(z)$ the factors of a canonical Wiener--Hopf factorization of the symbol
		$\phi(z)$, i.e., $\phi=\phi_-\phi_+$. Then
		\begin{gather*}
			D^B_{n}[\phi ; \psi ] =
			G[\phi]^{n} E[\phi]( F[\phi;\psi] + O({\rm e}^{-\mathfrak{c}n})),
		\end{gather*}
		where $G[\phi]$ and $E[\phi]$ are given by \eqref{G and E},
		\begin{gather*}
			F[\phi;\psi]=\frac{\big[\phi_-^{-1}\psi\big]_0}{[\phi_+]_0},
		\end{gather*}
		and $\mathfrak{c}$ is some positive constant.
	\end{Theorem}	
	It is worth mentioning that	these asymptotic results for the bordered Toeplitz determinants were obtained in parallel and independent to each other using the Riemann--Hilbert and operator-theoretic methods.

	While as discussed above the asymptotics of bordered determinants for a general class of symbols were established in \cite{BEGIL}, the asymptotics of semi-framed Toeplitz determinants remained uncharted territory. In this work, we undertake the task of filling this gap. A pivotal enabling factor for accessing these asymptotics lies in the connection of these objects to the system of bi-orthogonal polynomials on the unit circle (BOPUC): Just as BOPUC characterize (single) bordered determinants \cite{BEGIL}, we demonstrate that the reproducing kernel of BOPUC serves as the characterizing object for the semi-framed Toeplitz determinants. Consequently, with these characterizations in terms of bi-orthogonal polynomials and their reproducing kernel, we can employ the Riemann--Hilbert approach to BOPUC \cite{BDJ} to attain the sought-after asymptotics for multi-bordered, framed, and multi-framed Toeplitz determinants.
	
	Now, we make an attempt to address Question~\ref{q2} mentioned above. The semi-framed Toeplitz determinants have already appeared in the calculations of entanglement entropy for disjoint subsystems in the XX spin chain \cite{BGIKMMV,JK}, which we briefly recall below. To ensure consistency in notations, we closely follow the paper \cite{BGIKMMV}. For a more comprehensive description of the model, we refer to \cite{BGIKMMV,JK} and the references therein. Consider the chain of free fermions%
\begin{equation}\label{FFchain}
		H_F = -\sum_{j=1}^{N}b_j^\dagger b_{j+1} + b_jb_{j+1}^\dagger,
\end{equation}
	where the Fermi operators $b_j$ are defined by the anticommutation relations $\{b_j,b_k\} = 0$,\footnote{The bracket notation for two operators in this context is the anti-commutator: $\{a,b\}:=ab+ba$.} and $\big\{b_j,b_k^\dagger\big\}= \delta_{jk}$. Define the quantity
	\begin{equation}\label{eq:entr-int}
		S(\rho_P) = \lim_{\varepsilon\searrow 0} \frac{1}{2\pi {\rm i}} \oint_{\Gamma_\varepsilon} e(1+\varepsilon,\lambda)\frac{\rm d}{{\rm d}\lambda}\ln D(\lambda) {\rm d}\lambda,
	\end{equation}
	where
	\begin{equation*}
		 e(x,v) := -\frac{x+v}{2}\ln\frac{x+v}{2}-\frac{x-v}{2}\ln\frac{x-v}{2},
	\end{equation*}
	the contour $\Gamma_\varepsilon$ goes around the $[-1,1]$ interval once in the positive direction avoiding the cuts~${(-\infty,-1-\varepsilon]\cup[1+\varepsilon,\infty)}$ of $ e(1+\varepsilon,\cdot)$, and the function $D(\la)$ is defined further below (in terms of semi-framed Toeplitz determinants).
	
	Let $k,m,n\in\N$. The quantity in \eqref{eq:entr-int} is considered as a measure of entanglement between the subsystem
	\begin{gather*}
		P = \{1,2,\dots, m\}\cup\{m+k+1,m+k+2,\dots,m+k+n\},
	\end{gather*}
	and the rest of the chain of free fermions \eqref{FFchain} in the thermodynamic limit $N \to \infty$. The connection to the semi-framed Toeplitz determinants is through the function $D(\la)$, which we define here. Let $g\colon \T \to \C$ be defined as
	\begin{equation}\label{entanglement symbol}
		g(z) = \left\{ \begin{matrix}
			\hphantom{-}1, & \operatorname{Re} z \geq 0, \\
			-1, & \operatorname{Re} z < 0.
		\end{matrix} \right.
	\end{equation}
	Consider
	\begin{gather*}
		A = \left(\begin{matrix}
			A_{11} & A_{12} \\
			A_{21} & A_{22}
		\end{matrix}\right)\in\C^{(m+n)\times(m+n)},
	\end{gather*}
 and $D(\lambda) := \det(\lambda I - A)$, $\lambda\in\C$, where, recalling the notation $T_n[\phi]$ from \eqref{ToeplitzDet}, the matrix $A$ is defined as
	\begin{gather*}
		A_{11} = -T_m[g]\in\C^{m\times m}, \qquad A_{22} = -T_n[g]\in\C^{n\times n}, \\ A_{12} = A_{21}^{\mathsf T} = (\mathscr{A}_{ij}(k))_{i=1,\dots, m; j=1,\dots, n}\in\C^{m\times n},
	\end{gather*}
	and
	\begin{gather*}
		\mathscr{A}_{ij}(k) \equiv \mathscr{A}_{ij} = -\det \begin{bmatrix}
			g_{i-j-m-k} & g_{i-m-1} & g_{i-m-2} & \dots & g_{i-m-k} \\
			g_{1-j-k} & g_0 & g_{-1} & \dots & g_{1-k} \\
			g_{2-j-k} & g_1 & g_0 & \dots & g_{2-k} \\
			\vdots & \vdots & \vdots & \ddots & \vdots \\
			g_{-j} & g_{k-1} & g_{k-2} & \dots & g_0
		\end{bmatrix}.
	\end{gather*}
	Notice that $ \mathscr{A}_{ij}(k)$ is a $(k+1)\times(k+1)$ semi-framed Toeplitz matrix, which can be written in terms of the semi-framed Toeplitz determinants $\mathscr{H}_n[\phi;\psi,\eta;a]$ and $\mathscr{L}_n[\phi;\psi,\eta;a]$ introduced above. Indeed, by multiple adjacent row and column swaps, and recalling \eqref{half-framed intro00}, we can write
	\begin{align*}
			\mathscr{A}_{ij} ={}& -\begin{bmatrix}
				g_0 & g_{-1} & \dots & g_{1-k} & g_{1-j-k} \\
				g_1 & g_0 & \dots & g_{2-k} & g_{2-j-k} \\
				\vdots & \vdots & \vdots & \ddots & \vdots \\
				g_{k-1} & g_{k-2} & \dots & g_0 & g_{-j} \\
				g_{i-m-1} & g_{i-m-2} & \dots & g_{i-m-k} & g_{i-j-m-k} \\
			\end{bmatrix} \\ ={}& - \mathscr{H}_{k+1} \big[g(z); g(z) z^{j+k-1}, g(z) z^{m+k-i}; g_{i-j-m-k} \big] \\ ={}&
			- \mathscr{L}_{k+1} \big[g(z); \tilde{g}(z) z^{-j}, \tilde{g}(z) z^{i-m-1}; g_{i-j-m-k} \big],
	\end{align*}
	where $\tilde{f}(z)\! =\! f\big(z^{-1}\big)$. Even though the computation of the entanglement between the chain~\eqref{FFchain} and the rest of the system is of interest in the regime where all three parameters $m$, $n$, $k$ tend towards infinity, the authors in \cite{BGIKMMV} specifically focused on the scenario where $k=1$ and $m,n \to \infty$. We quote:\footnote{Also see \cite[Remark 1 in Section 3]{BGIKMMV}.}
\begin{quote}
		\textit{Our ultimate interest is to analyse $S(\rho_P)$ as $k,m,n\to\infty$, however, at this point the general problem seems to be far too complicated to attack directly. Therefore, we decided to start with the easier case when the gap between the two intervals is fixed to be $k=1$. $\dots$ As we shall see, this simplest case already leads to a mathematically very challenging problem.}
	\end{quote}
	
	Now let us demonstrate how the findings of this work could be relevant to the goal of \cite{BGIKMMV} in extending the analysis to the asymptotic regime $k \to \infty$. Indeed, in Section~\ref{section framed}, among other results, we prove that for general symbols the semi-framed Toeplitz determinants have a~representation in terms of the solution of the BOPUC Riemann--Hilbert problem. For example, for~${\mathscr{L}_k[\phi;\psi,\eta;a]}$ we show
\begin{gather}
		\frac{\mathscr{L}_{n+2}[\phi;\psi, \eta ;a]}{D_{n+1}[\phi]} \nonumber\\
\qquad = a - \int_{\T} \int_{\T} \frac{ \Tilde{\eta}(z_2) \tilde{\psi}(z_1)}{z_1-z_2} \det \begin{bmatrix}
			X_{11}(z_2;n+1) & X_{21}(z_2;n+2) \\
			X_{11}(z_1;n+1) & X_{21}(z_1;n+2)
		\end{bmatrix} \frac{\dd z_2}{2 \pi \ii z_2} \frac{\dd z_1}{2 \pi \ii z_1},\label{L X intro}
	\end{gather}
	where $X_{11}$ and $X_{21}$ are the entries in the first column of the solution $X$ to the Riemann--Hilbert problem for BOPUC associated with the orthogonality weight $\phi$. More precisely, $X$ solves the following Riemann--Hilbert problem (RHP) \cite{BDJ}:
	\begin{itemize}\itemsep=0pt
		\item RH-X1: 
$X(\cdot;n)\colon \C\setminus \T \to \C^{2\times2}$ is analytic.
		\item RH-X2: 
The limits of $X(\xi;n)$ as $\xi$ tends to $z \in \T $ from the inside and outside of the unit circle exist, which are denoted by $X_{\pm}(z;n)$ respectively. Moreover, the functions~${z \mapsto X_{\pm}(z)}$ are continuous on $\T$ and are related by the \textit{jump condition}
		\begin{equation*}
			X_+(z;n)=X_-(z;n)\begin{bmatrix}
				1 & z^{-n}\phi(z) \\
				0 & 1
			\end{bmatrix}, \qquad z \in \T.
		\end{equation*}
		\item RH-X3: 
 As $z \to \infty$
		\begin{equation}\label{X asymp1}
			X(z;n)=\left(\di I+\frac{ \overset{\infty}{X}_1(n)}{z}+\frac{\overset{\infty}{X}_2(n)}{z^2} + O\big(z^{-3}\big)\right) z^{n \sigma_3},
		\end{equation}
		where $\sigma_3 = \left[\begin{smallmatrix} 1 & \hphantom{-}0 \\ 0 & -1 \end{smallmatrix}\right]$ is the third Pauli matrix.
	\end{itemize}
	Since the symbol $\phi=g$ given by \eqref{entanglement symbol} is a Fisher--Hartwig symbol, one may refer to \cite{DIK} for the asymptotics of $D_{n}[g]$ and the polynomials $X_{11}(z;n)$ and $X_{21}(z;n)$ as $n \to \infty$, and then perform the integrations in \eqref{L X intro}. This approach is expected to yield the asymptotic behavior of $\mathscr{A}_{ij}(k)$ as $k$ approaches infinity, which could contribute to our comprehension of the entanglement in the limit where $k$, $m$, and $n$ all tend to infinity.

	We would like to emphasize that this work is expected to provide the \textit{general framework} of translating the objects of interest in terms of the solution to the $X$-RHP. We do occasionally use the RHP characterizations, such as \eqref{L X intro}, to demonstrate how the asymptotics could be obtained, but we do not intend to exhaust all the cases. For example, as a way of explanation, we show in Section~\ref{section framed} how this scheme works in the case of Szeg\H{o}-type (non Fisher--Hartwig) symbols and when $\psi$ and $\eta$ are either rational functions, or the product of a rational function with the bulk symbol $\phi$. We have made such choices since these choices are simple and yet nontrivial enough to illustrate the procedure, but by no means we do not want to convey the message that those are the only cases which can make the asymptotic analysis feasible. We plan to undertake the task of using identities like \eqref{L X intro} for Fisher--Hartwig symbols in a forthcoming work, especially in connection to the entanglement problem discussed above.
	
	In addition to their relevance in the context of XX quantum spin chains, the author has recently received information on the appearance and relevance of multi-bordered, semi-framed, framed and multi-framed Toeplitz determinants in other contexts. Professor Karl Liechty has communicated to the author that these structured determinants arise in the analysis of ensembles of nonintersecting paths, via the Lindstr\"om--Gessel--Viennot (LVG) formula.\footnote{For more details and references, please see the recent work~\cite{GL24} and the references therein.} In a separate communication, Professor Nicholas Witte has highlighted to the author the relevance of these structures in his ongoing research on the Ising model \cite{W}. These recent discoveries have served as additional inspiration for the current study.	
	
\subsection{An outline of main results}
	
\subsubsection[Strong Szeg\H{o} limit theorem for two-bordered Toeplitz determinants]{Strong Szeg\H{o} limit theorem for two-bordered Toeplitz determinants}
	
	In Section~\ref{Sec 2.2}, we prove the following.
	
	\begin{Theorem}\label{main thm 2-bordered}
		For $\ell=1,2$,	let \smash{$\psi_\ell(z) = q^{(\ell)}_1(z) \phi(z) + q^{(\ell)}_2(z)$}, where
		\begin{gather*}
			q^{(\ell)}_1(z) = a^{(\ell)}_0+a^{(\ell)}_1z+\frac{b^{(\ell)}_0}{z}+\sum^{m_{\ell}}_{j=1}\frac{b^{(\ell)}_j z}{z-c^{(\ell)}_j}, \\
 q^{(\ell)}_2(z) = \hat{a}^{(\ell)}_0+\hat{a}^{(\ell)}_1z+\frac{\hat{b}^{(\ell)}_0}{z}+
			\sum_{j=1}^m \frac{\hat{b}^{(\ell)}_j}{z-c^{(\ell)}_j},
		\end{gather*}
		and suppose that $\phi$ is of Szeg\H{o}-type. Then, the associated two-bordered Toeplitz determinant has the following asymptotic behavior as $n \to \infty$:
		\begin{equation}\label{135}
			D^B_n[\phi;\boldsymbol{\psi}_2] \equiv D^B_n[\phi;\psi_1,\psi_2] = G^n[\phi] E[\phi] \{ J_1[\phi,\psi_1,\psi_2] + O(\rho^{-n}) \},
		\end{equation}
		where $G[\phi]$ and $E[\phi]$ are given by \eqref{G and E},
		\begin{equation}\label{136}
			J_1[\phi,\psi_1,\psi_2] = \begin{vmatrix}
				F[\phi,\psi_2] & F[\phi,\psi_1] \\
				H[\phi,\psi_2] & H[\phi,\psi_1]
			\end{vmatrix},
		\end{equation}
		in which $F[\phi,\psi]$ is given by \eqref{ConstantF2}, and
		\begin{align*}
				H[\phi;\psi] =	{}&
				a_1 - \sum_{j=1}^{m} \frac{b_j}{c_j}+ a_0 [\log \phi]_{1} + b_0 [\log \phi]_{2} + \frac{b_0}{2} [\log \phi]_{1}^2
				\\ & +	
				\frac{1}{G[\phi]}\Bigg(
				\hat{a}_1-
				\sum^{m}_{j=1 \atop |c_j|>1 }\frac{\hat{b}_j}{c^2_j} \al(c_j) + 		\sum^m_{j=1 \atop 0<|c_j|<1} \frac{b_j}{c_j} \al(c_j) \Bigg).
		\end{align*}
		In the above formula,
		$\al$ is given by \eqref{al}, and the number $\rho$ is such that
\[
1<\rho < \underset{1 \leq j \leq m \atop |c_j|>1}{\min} \{|c_j|\}, \qquad \underset{1 \leq j \leq m \atop 0<|c_j|<1}{\max} \{|c_j|\}<\rho^{-1}<1,
\]
 and $\phi$ is analytic in the annulus $\big\{z\colon \rho^{-1}<|z|<\rho\big\}$.
	\end{Theorem}
	
	At the end of Section \ref{sec bordered}, in Remark \ref{remark multi-bordered}, we explain how the techniques used for the two bordered case can be recursively used to obtain the asymptotics of $k$-bordered Toeplitz determinants, $k>2$.
	
	\subsubsection[The Riemann--Hilbert problem for BOPUC when the weight has\\ a nonzero winding number]{The Riemann--Hilbert problem for BOPUC\\ when the weight has a nonzero winding number}
	In order to arrive at the above asymptotic results for multi-bordered Toeplitz determinants, one is invited to asymptotically analyze (single) bordered Toeplitz determinants of the form $D_n[z\phi;q_1 \phi + q_2]$, as a result of employing a Dodgson condensation identity. Notice that if~$\phi$ is of Szeg\H{o}-type, then $z \phi$ is not, as it does not have a zero winding number. Therefore, the asymptotics of $D_n[z\phi;q_1 \phi + q_2]$ can not be obtained from Theorem \ref{main thm} and thus must be treated differently. Such bordered determinants are characterized in terms of the solution to the following Riemann--Hilbert problem:
	\begin{itemize}\itemsep=0pt
		\item RH-Z1: 
$Z(\cdot;n)\colon\C\setminus \T \to \C^{2\times2}$ is analytic.
		\item RH-Z2: 
The limits of $Z(\ze;n)$ as $\ze$ tends to $z \in \T $ from the inside and outside of the unit circle exist, which are denoted by $Z_{\pm}(z;n)$ respectively. Moreover, the functions~${z \mapsto Z_{\pm}(z)}$ are continuous on $\T$ and are related by
		\begin{equation*}
			Z_+(z;n)=Z_-(z;n)\begin{bmatrix}
				1 & z^{-n+1}\phi(z) \\
				0 & 1
			\end{bmatrix}, \qquad z \in \T.
		\end{equation*}
		
		\item RH-Z3: 
$Z(z;n)=\big( I + O\big(z^{-1}\big) \big) z^{n \sigma_3}$ as $z \to \infty$.
	\end{itemize}
This is the same as RH-X1--RH-X3, the only difference being that $\phi$ is now replaced by $z\phi$. However, the usual steps of the Deift--Zhou nonlinear steepest descent analysis \cite{DZ} do not work for a symbol with nonzero winding number.\footnote{ This is because at the stage of finding the solution to the global parametrix RHP one would need to find the solution of the scalar RHP: $\be_+(z)-\be_-(z) = \log (z \phi(z))$ for $z \in \T$ and $\be(z)=1+O(1/z)$ as $z \to \infty$, see Appendix~\ref{Appendices}. However, the Plemelj--Sokhotskii formula can not be applied \cite{Gakhov} as the function $\log (z \phi(z))$ has a~jump discontinuity on the unit circle, for a~Szeg\H{o}-type $\phi$.} Instead, we find the explicit formulae relating the solution of the $Z$-RHP to the solution of the $X$-RHP, which is amenable to the Deift--Zhou nonlinear steepest descent analysis (see Appendix~\ref{Appendices}). In our work, such relations are essential in proving Theorem \ref{main thm 2-bordered}. 	
	
	\begin{Theorem}\label{thm Z-RHP intro}
Assume that the solution $X(z;n)$ of the Riemann--Hilbert problem {\rm RH-X1} through {\rm RH-X3} and the solution $Z(z;n)$ of the Riemann--Hilbert problem {\rm RH-Z1} through {\rm RH-Z3} exist. The solution $Z(z;n)$ can be expressed in terms of the data extracted from the solution~${X(z;n)}$ as
		\begin{equation}\label{ZX1}
			Z(z;n) = \left( \begin{bmatrix}
				 \tfrac{\overset{\infty}{X}_{1,12}(n)X_{21}(0;n)}{X_{11}(0;n)} & -\overset{\infty}{X}_{1,12}(n) \\[6pt]
					-\tfrac{X_{21}(0;n)}{X_{11}(0;n)} & 1
			\end{bmatrix} z^{-1} + \begin{bmatrix}
				1 & 0 \\
				0 & 0
			\end{bmatrix} \right) X(z;n) \begin{bmatrix}
				1 & 0 \\
				0 & z
			\end{bmatrix},
		\end{equation}
		where $\overset{\infty}{X}_{1,12}(n)$ is the $12$-entry of the matrix $\overset{\infty}{X}_{1}(n)$ in the asymptotic expansion \eqref{X asymp1}.
	\end{Theorem}
	
	\begin{Remark}
		Assuming the existence of the solutions $X(z;n)$ and $Z(z;n)$ of the X-RHP and the Z-RHP, respectively, we can show that $X_{11}(0;n) \neq 0$. This will be done in
		Section~\ref{sec proof of thm 1.6} where we prove the Theorem \ref{thm Z-RHP intro}.
	\end{Remark}
	
	We alternatively prove another way to connect the solution of the $Z$-RHP to the solution of the $X$-RHP, described in the following theorem.
	\begin{Theorem}\label{thmXZ intro} Suppose that the solution $X(z;n-1)$ of the Riemann--Hilbert problem \mbox{{\rm RH-X1}} through {\rm RH-X3} $($with parameter $n-1)$ and the solution $Z(z;n)$ of the Riemann--Hilbert problem {\rm RH-Z1} through {\rm RH-Z3} exist. The solution $Z(z;n)$ can be expressed in terms of the data extracted from the solution $X(z;n)$ of the Riemann--Hilbert problem {\rm RH-X1} through {\rm RH-X3}~as%
		\begin{equation}\label{Z in terms of X1}
			Z(z;n) = \begin{bmatrix}
				z + \overset{\infty}{X}_{1,22}(n-1) - \tfrac{\overset{\infty}{X}_{2,12}(n-1)}{\overset{\infty}{X}_{1,12}(n-1)} & -\overset{\infty}{X}_{1,12}(n-1) \vspace{6pt}\\
				\tfrac{1}{\overset{\infty}{X}_{1,12}(n-1)} & 0
			\end{bmatrix} X(z;n-1),
		\end{equation} where \smash{$\overset{\infty}{X}_{1,jk}(n)$} and \smash{$\overset{\infty}{X}_{2,jk}(n)$} are the $jk$-entries of the matrices \smash{$\overset{\infty}{X}_{1}(n)$} and \smash{$\overset{\infty}{X}_{2}(n)$} in the asymptotic expansion \eqref{X asymp1}.
	\end{Theorem}
	
\begin{Remark}
		Assuming the existence of the solutions $X(z;n-1)$ and $Z(z;n)$ of the X-RHP (with parameter $n-1$) and the Z-RHP, respectively, we can show that \smash{$\overset{\infty}{X}_{1,12}(n-1) \neq 0$}. This will be done in
		Section~\ref{sec proof of thm 1.8}, where we prove the Theorem \ref{thmXZ intro}.
	\end{Remark}
	
	\begin{Remark}
		The compatibility of these two theorems offers a new proof for the recurrence relations governing the system of bi-orthogonal polynomials on the unit circle, as detailed in Lemma \ref{lemma 3term}.	
	\end{Remark}
	
	\subsubsection{Strong Szeg\H{o} limit theorems for semi-framed Toeplitz determinants}	
	For $\phi,\psi,\eta \in L^1(\T)$ and a parameter $a \in \C$ define the $n\times n$ semi-framed Toeplitz determinants~${\mathscr{E}_n[\phi;\psi,\eta;a]}$, $\mathscr{G}_n[\phi;\psi,\eta;a]$, $\mathscr{H}_n[\phi;\psi,\eta;a]$ and $\mathscr{L}_n[\phi;\psi,\eta;a]$ as
	\begin{gather*}
		\mathscr{E}_n[\phi;\psi,\eta;a] := \det \begin{bmatrix}
			\phi_0& \phi_{-1} & \cdots & \phi_{-n+2} & \psi_{n-2} \\
			\phi_{1}& \phi_0 & \cdots & \phi_{-n+3} & \psi_{n-3} \\
			\vdots & \vdots & \ddots & \vdots & \vdots \\
			\phi_{n-2} & \phi_{n-3} & \cdots & \phi_{0} & \psi_{0} \\
			\eta_{n-2} & \eta_{n-3} & \cdots & \eta_{0} & a
		\end{bmatrix},
\\
		\mathscr{G}_n[\phi;\psi,\eta;a] := \det \begin{bmatrix}
			\phi_0& \phi_{-1} & \cdots & \phi_{-n+2} & \psi_{0} \\
			\phi_{1}& \phi_0 & \cdots & \phi_{-n+3} & \psi_{1} \\
			\vdots & \vdots & \ddots & \vdots & \vdots \\
			\phi_{n-2} & \phi_{n-3} & \cdots & \phi_{0} & \psi_{n-2} \\
			\eta_{0} & \eta_{1} & \cdots & \eta_{n-2} & a
		\end{bmatrix},
\\
		\mathscr{H}_n[\phi;\psi,\eta;a] := \det \begin{bmatrix}
			\phi_0& \phi_{-1} & \cdots & \phi_{-n+2} & \psi_{0} \\
			\phi_{1}& \phi_0 & \cdots & \phi_{-n+3} & \psi_{1} \\
			\vdots & \vdots & \ddots & \vdots & \vdots \\
			\phi_{n-2} & \phi_{n-3} & \cdots & \phi_{0} & \psi_{n-2} \\
			\eta_{n-2} & \eta_{n-3} & \cdots & \eta_{0} & a
		\end{bmatrix},
	\end{gather*}
	and
	\begin{gather*}
		\mathscr{L}_n[\phi;\psi,\eta;a] := \det \begin{bmatrix}
			\phi_0& \phi_{-1} & \cdots & \phi_{-n+2} & \psi_{n-2} \\
			\phi_{1}& \phi_0 & \cdots & \phi_{-n+3} & \psi_{n-3} \\
			\vdots & \vdots & \ddots & \vdots & \vdots \\
			\phi_{n-2} & \phi_{n-3} & \cdots & \phi_{0} & \psi_{0} \\
			\eta_{0} & \eta_{1} & \cdots & \eta_{n-2} & a
		\end{bmatrix},
	\end{gather*}
	where $f_j$'s are the Fourier coefficients of $f \in \{ \phi, \psi, \eta \}$. Consider the reproducing kernel
	\begin{equation}\label{RepKer intro}
		K_{n}(z,\mathcal{z}) := \sum_{j=0}^{n} Q_{j}(\mathcal{z})\widehat{Q}_{j}(z),
	\end{equation} of the system of bi-orthogonal polynomials on the unit circle associated with the symbol $\phi$, satisfying the \textit{bi-orthogonality} relation
	\begin{equation}\label{bi-orthogonality intro}
		\int_{\T} Q_n(\ze)\widehat{Q}_m\big(\ze^{-1}\big)\phi(\ze)\frac{\dd \ze}{2\pi \ii \ze}= \de_{nm}, \qquad n,m \in \N \cup \{0\}.
	\end{equation}
	In Section \ref{section framed}, we prove the following representation of the above semi-framed Toeplitz determinants in terms of the reproducing kernel \eqref{RepKer intro}.
	
	\begin{Theorem}\label{semis in terms of RepKer intro}
		Let $D_k[\phi] \neq 0$ for $k=0,1,\dots, n+1$.	The semi-framed Toeplitz determinants~${\mathscr{E}_n[\phi;\psi,\eta;a]}$, $\mathscr{G}_n[\phi;\psi,\eta;a]$, $\mathscr{H}_n[\phi;\psi,\eta;a]$ and $\mathscr{L}_n[\phi;\psi,\eta;a]$ can be represented in terms of the reproducing kernel of the system of bi-orthogonal polynomials on the unit circle associated with $\phi$ given by \eqref{RepKer intro} and \eqref{bi-orthogonality intro} as
		\begin{align}
			&\frac{\mathscr{E}_{n+2}[\phi; \psi, \eta ;a]}{D_{n+1}[\phi]} = a - \int_{\T} \left[ \int_{\T} K_n(z_1,z_2) z_2^{-n} \eta(z_2) \frac{\dd z_2}{2 \pi \ii z_2} \right] z_1^{-n} \psi(z_1) \frac{\dd z_1}{2 \pi \ii z_1}, \nonumber \\
& \frac{\mathscr{G}_{n+2}[\phi; \psi, \eta ;a]}{D_{n+1}[\phi]} = a - \int_{\T} \left[ \int_{\T} K_n\big(z^{-1}_1,z^{-1}_2\big) \eta(z_2) \frac{\dd z_2}{2 \pi \ii z_2} \right] \psi(z_1) \frac{\dd z_1}{2 \pi \ii z_1}, \nonumber \\
& \frac{\mathscr{H}_{n+2}[\phi; \psi, \eta ;a]}{D_{n+1}[\phi]} = a - \int_{\T} \left[ \int_{\T} K_n\big(z^{-1}_1,z_2\big) z_2^{-n} \eta(z_2) \frac{\dd z_2}{2 \pi \ii z_2} \right] \psi(z_1) \frac{\dd z_1}{2 \pi \ii z_1}, \label{H} \\
& \frac{\mathscr{L}_{n+2}[\phi; \psi, \eta ;a]}{D_{n+1}[\phi]} = a - \int_{\T} \left[ \int_{\T} K_n\big(z_1,z^{-1}_2\big) \eta(z_2) \frac{\dd z_2}{2 \pi \ii z_2} \right] z^{-n}_1 \psi(z_1) \frac{\dd z_1}{2 \pi \ii z_1}, \nonumber
		\end{align}
		where $D_{n}[\phi]$ is given by \eqref{ToeplitzDet}.
	\end{Theorem}
	Using the Christoffel--Darboux identity for the bi-orthogonal polynomials on the unit circle, we obtain the following characterizations in terms of the solution $X$ of {\rm RH-X1} through {\rm RH-X3} in the following corollary.
	
	\begin{Corollary}\label{EGL and RHP intro}
		Let $D_k[\phi] \neq 0$ for $k=0,1,\dots, n+1$. The semi-framed Toeplitz determinants~${\mathscr{H}_{n+2}[\phi; \psi, \eta ;a]}$, $\mathscr{E}_{n+2}[\phi; \psi, \eta ;a]$, $\mathscr{G}_{n+2}[\phi; \psi, \eta ;a]$, and $\mathscr{L}_{n+2}[\phi; \psi, \eta ;a]$ are encoded into the $X$-RHP data as
		\begin{align}
& \frac{\mathscr{E}_{n+2}[\phi;\psi, \eta ;a]}{D_{n+1}[\phi]}\nonumber \\
&\qquad= a - \int_{\T} \int_{\T} \frac{z^{-n}_2\eta(z_2) \Tilde{\psi}(z_1)}{z_1-z_2} \det \begin{bmatrix}
				X_{11}(z_2;n+1) & X_{21}(z_2;n+2) \\
				X_{11}(z_1;n+1) & X_{21}(z_1;n+2)
			\end{bmatrix} \frac{\dd z_2}{2 \pi \ii z_2} \frac{\dd z_1}{2 \pi \ii z_1}, 				\label{E and X-RHP} \\
& 		\frac{\mathscr{G}_{n+2}[\phi;\psi, \eta ;a]}{D_{n+1}[\phi]}\nonumber\\
&\qquad = a - \int_{\T} \int_{\T} \frac{z^{-n}_1 \tilde{\eta}(z_2) \psi(z_1)}{z_1-z_2} \det \begin{bmatrix}
				X_{11}(z_2;n+1) & X_{21}(z_2;n+2) \\
				X_{11}(z_1;n+1) & X_{21}(z_1;n+2)
			\end{bmatrix} \frac{\dd z_2}{2 \pi \ii z_2} \frac{\dd z_1}{2 \pi \ii z_1}, 			 \nonumber \\
&	 	\frac{\mathscr{H}_{n+2}[\phi;\psi, \eta ;a]}{D_{n+1}[\phi]}= a - \int_{\T} \int_{\T} \frac{z^{-n}_1z^{-n}_2\eta(z_2) \psi(z_1)}{z_1-z_2}\nonumber\\
&\phantom{\frac{\mathscr{H}_{n+2}[\phi;\psi, \eta ;a]}{D_{n+1}[\phi]}=}{} \times\det \begin{bmatrix}
				X_{11}(z_2;n+1) & X_{21}(z_2;n+2) \\
				X_{11}(z_1;n+1) & X_{21}(z_1;n+2)
			\end{bmatrix} \frac{\dd z_2}{2 \pi \ii z_2} \frac{\dd z_1}{2 \pi \ii z_1}, \label{H and X-RHP} \\
&	 	\frac{\mathscr{L}_{n+2}[\phi;\psi, \eta ;a]}{D_{n+1}[\phi]}= a - \int_{\T} \int_{\T} \frac{ \Tilde{\eta}(z_2) \tilde{\psi}(z_1)}{z_1-z_2} \det \begin{bmatrix}
				X_{11}(z_2;n+1) & X_{21}(z_2;n+2) \\
				X_{11}(z_1;n+1) & X_{21}(z_1;n+2)
			\end{bmatrix} \frac{\dd z_2}{2 \pi \ii z_2} \frac{\dd z_1}{2 \pi \ii z_1}, 				 \nonumber
		\end{align}
		where $\tilde{f}(z) = f\big(z^{-1}\big)$, $D_{n}[\phi]$ is given by \eqref{ToeplitzDet}, and $X_{11}$ and $X_{21}$ are respectively the $11$ and $21$ entries of the solution to {\rm RH-X1} through {\rm RH-X3}.
	\end{Corollary}
	
In Section \ref{sec semi-framed}, we prove the following strong Szeg\H{o} theorems for semi-framed Toeplitz determinants for a class of \textit{frame symbols} $\psi$ and $\eta$. In Theorem \ref{thm semi-framed rationals intro}, we consider the case where the frame-symbols are rational functions with arbitrary number of simple poles.
	
	\begin{Theorem}\label{thm semi-framed rationals intro}
		Let $\phi$ be of Szeg\H{o}-type, and $c$ and $d$ be complex numbers that do not lie on the unit circle. Then, the following strong Szeg\H{o} asymptotics hold for $\mathscr{H}$, $\mathscr{L}$, $\mathscr{E}$ and $\mathscr{G}$:
		\begin{gather}
			\mathscr{H}_{n+1}\Bigg[\phi; \di \sum_{j=1}^{m_1} \di \frac{A_j}{z-d_j}, \sum_{k=1}^{m_2} \di \frac{B_k}{z-c_k};a\Bigg] = G^n[\phi] E[\phi] (a + O(\rho^{-n})),\label{asymp H rationals lin comb}	
\\
			\mathscr{L}_{n+1}\Bigg[\phi; \di \sum_{j=1}^{m_1} \di \frac{A_j}{z-d_j}, \sum_{k=1}^{m_2} \di \frac{B_k}{z-c_k};a\Bigg] = G^n[\phi] E[\phi] (a + O(\rho^{-n})),\nonumber	
\\
			\mathscr{E}_{n+1}\Bigg[\phi; \di \sum_{j=1}^{m_1} \di \frac{A_j}{z-d_j}, \sum_{k=1}^{m_2} \di \frac{B_k}{z-c_k};a\Bigg] \nonumber \\
\qquad= G^n[\phi] E[\phi] \Bigg( a + \sum^{m_1}_{j=1 \atop |d_j|>1 } \sum^{m_2}_{k=1 \atop |c_k|>1 } A_jB_k \frac{\al(c_k)}{\al(d_j^{-1})}\cdot \frac{1}{1-c_kd_j} +O(\rho^{-n}) \Bigg),\nonumber	
\\
			\mathscr{G}_{n+1}\Bigg[\phi; \di \sum_{j=1}^{m_1} \di \frac{A_j}{z-d_j}, \sum_{k=1}^{m_2} \di \frac{B_k}{z-c_k};a\Bigg] \nonumber\\
\qquad= G^n[\phi] E[\phi] \Bigg( a + \sum^{m_1}_{j=1 \atop |d_j|>1 } \sum^{m_2}_{k=1 \atop |c_k|>1 } A_jB_k \frac{\al(d_j)}{\al(c_k^{-1})}\cdot \frac{1}{1-c_kd_j} +O(\rho^{-n}) \Bigg).\nonumber
		\end{gather}			
		Here the number $\rho$ is such that
\[1<\rho < \underset{1 \leq j \leq m_1, 1 \leq k \leq m_2 \atop |d_j|>1,|c_k|>1}{\min} \{|d_j|,|c_k|\}, \qquad \underset{1 \leq j \leq m_1, 1 \leq k \leq m_2 \atop |d_j|<1,|c_k|<1}{\max} \{|d_j|,|c_k|\}<\rho^{-1}<1,
\]
 and $\phi$ is analytic in the annulus $\{z\colon \rho^{-1}<|z|<\rho\}$.
	\end{Theorem}
	
	The following theorem is about the case where the frame-symbols are rational functions with simple poles multiplied by $\phi$ or $\tilde{\phi}$. For $\mathscr{E}$ and $\mathscr{G}$ in this case, unlike what we have in Theorem~\ref{thm semi-framed rationals intro}, only the poles inside the unit disk may contribute to the leading-order asymptotics.
	\begin{Theorem}\label{thm semi-framed rationals . phi intro} Let $\phi$ be a Szeg\H{o}-type symbol, and $c$ and $d$ be complex numbers that do not lie on the unit circle. Then, the following strong Szeg\H{o} asymptotics hold for $\mathscr{H}$, $\mathscr{L}$, $\mathscr{E}$ and $\mathscr{G}$:
		\begin{gather}
			\mathscr{H}_{n+1}\Bigg[\phi; \di \sum_{j=1}^{m_1} \di \frac{A_j \phi}{z-d_j}, \sum_{k=1}^{m_2} \di \frac{B_k\phi}{z-c_k};a\Bigg] = G^n[\phi] E[\phi] (a + O(\rho^{-n})),\label{asymp H rationals lin comb1}	
\\
			\mathscr{L}_{n+1}\Bigg[\phi; \di \sum_{j=1}^{m_1} \di \frac{A_j\tilde{\phi}}{z-d_j}, \sum_{k=1}^{m_2} \di \frac{B_k\tilde{\phi}}{z-c_k};a\Bigg] = G^n[\phi] E[\phi] (a + O(\rho^{-n})),	\nonumber
\\
			\mathscr{E}_{n+1}\Bigg[\phi; \di \sum_{j=1}^{m_1} \di \frac{A_j\tilde{\phi}}{z-d_j}, \sum_{k=1}^{m_2} \di \frac{B_k\phi}{z-c_k};a\Bigg] \nonumber
\\ \qquad= G^n[\phi] E[\phi] \Bigg( a + \sum^{m_1}_{j=1 \atop |d_j|<1 } \sum^{m_2}_{k=1 \atop |c_k|<1 } A_jB_k \frac{\al(c_k)}{\al(d_j^{-1})}\cdot \frac{1}{1-c_kd_j} +O(\rho^{-n}) \Bigg),	\nonumber
\\
			\mathscr{G}_{n+1}\Bigg[\phi; \di \sum_{j=1}^{m_1} \di \frac{A_j\phi}{z-d_j}, \sum_{k=1}^{m_2} \di \frac{B_k\tilde{\phi}}{z-c_k};a\Bigg] \nonumber
\\ \qquad= G^n[\phi] E[\phi] \Bigg( a + \sum^{m_1}_{j=1 \atop |d_j|<1 } \sum^{m_2}_{k=1 \atop |c_k|<1 } A_jB_k \frac{\al(d_j)}{\al\big(c_k^{-1}\big)}\cdot \frac{1}{1-c_kd_j} +O(\rho^{-n}) \Bigg),\label{asymp G rationals lin comb1}
		\end{gather}
		where $\tilde{f}(z) = f\big(z^{-1}\big)$. Here the number $\rho$ is such that
\[
1<\rho < \underset{1 \leq j \leq m_1, 1 \leq k \leq m_2 \atop |d_j|>1,|c_k|>1}{\min} \{|d_j|,|c_k|\}, \qquad \underset{1 \leq j \leq m_1, 1 \leq k \leq m_2 \atop |d_j|<1,|c_k|<1}{\max} \{|d_j|,|c_k|\}<\rho^{-1}<1,
\]
 and $\phi$ is analytic in the annulus $\big\{z\colon \rho^{-1}<|z|<\rho\big\}$.
	\end{Theorem}
	
	In Section \ref{sec framed and multi-framed}, we eventually redirect our attention to framed and multi-framed Toeplitz determinants. Our intention in this section is not to present formal proofs of asymptotic results. Instead, our goal is to present a broad framework for approaching the asymptotic analysis of these determinants, emphasizing their recursive characteristics in relation to the Dodgson condensation identities. We will show that the semi-framed Toeplitz determinants are the building blocks for the asymptotic analysis of framed and multi-framed determinants.
	
\section{Multi-bordered Toeplitz determinants} \label{sec bordered}
	In this section, we focus on multi-bordered Toeplitz determinants
	\begin{equation}\label{btdmulti}
		D^B_n[\phi;\boldsymbol{\psi}_m] := \det \begin{bmatrix}
			\phi_0& \phi_{1} & \cdots & \phi_{n-m-1} & \psi_{1,n-1} & \cdots & \psi_{m,n-1} \\
			\phi_{-1}& \phi_0 & \cdots & \phi_{n-m-2} & \psi_{1,n-2} & \cdots & \psi_{m,n-2} \\
			\vdots & \vdots & \vdots & \vdots & \vdots & \cdots & \vdots \\
			\phi_{-n+1} & \phi_{-n+2} & \cdots & \phi_{-m} & \psi_{1,0} & \cdots & \psi_{m,0}
		\end{bmatrix},
	\end{equation}
	and their reduction to (single) bordered determinants. This reduction allows for a representation in terms of the orthogonal polynomials on the unit circle and hence a Riemann--Hilbert characterization. For $\phi \in L^1(\T)$, let us recall (see, e.g., \cite{DIK} and the references therein) the system of polynomials on the unit circle $\{Q_n(z)\}^{\infty}_{n=0}$ and $\big\{\widehat{Q}_n(z)\big\}^{\infty}_{n=0}$, $\deg Q_n = \deg \widehat{Q}_n = n$, satisfying the orthogonality relations
	\begin{equation*}
		\int_{\T} Q_n(\ze)\ze^{-m}\phi(\ze)\frac{\dd \ze}{2\pi \ii \ze}=\varkappa^{-1}_n \de_{nm}, \qquad m =0, \dots,n,
	\end{equation*}
	and
	\begin{equation*}
		\int_{\T} \widehat{Q}_n\big(\ze^{-1}\big)\ze^{m}\phi(\ze)\frac{\dd \ze}{2\pi \ii \ze}=\varkappa^{-1}_n \de_{nm}, \qquad m =0, \dots,n,
	\end{equation*}
	where $\varkappa_n \neq 0$ is the leading coefficient of both $Q_n$ and $\widehat{Q}_n$. If $D_n[\phi] \neq 0$ and $D_{n+1}[\phi] \neq 0$, the polynomials $Q_n$ and $\widehat{Q}_n$ uniquely exist and are given by
	\begin{equation}\label{Toeplitz OP 1}
		Q_n(z):= \frac{1}{\sqrt{D_n[\phi] D_{n+1}[\phi]}} \det \begin{bmatrix}
			\phi_0 & \phi_{-1} & \cdots & \phi_{-n} \\
			\phi_1 & \phi_{0} & \cdots & \phi_{-n+1} \\
			\vdots & \vdots & \ddots & \vdots \\
			\phi_{n-1} & \phi_{n-2} & \cdots & \phi_{-1} \\
			1 & z & \cdots & z^n
		\end{bmatrix},
	\end{equation}
	and
	\begin{equation}\label{Toeplitz OP 2}
		\widehat{Q}_n(z) := \frac{1}{\sqrt{D_n[\phi] D_{n+1}[\phi]}} \det \begin{bmatrix}
			\phi_0 & \phi_{-1} & \cdots & \phi_{-n+1} & 1 \\
			\phi_1 & \phi_{0} & \cdots & \phi_{-n+2} & z \\
			\vdots & \vdots & \ddots & \vdots \\
			\phi_{n} & \phi_{n-1} & \cdots & \phi_{1} & z^n
		\end{bmatrix},
	\end{equation}
	and in addition
	\begin{equation}\label{varkappa}
		\varkappa_n = \sqrt{\frac{D_{n}[\phi]}{D_{n+1}[\phi]}}, \qquad n \in \N \cup \{0\},
	\end{equation}
	where we set $D_0[\phi] \equiv 1$. The existence is clear by the construction \eqref{Toeplitz OP 1}--\eqref{Toeplitz OP 2}, while the uniqueness follows from the unique solvability for the linear system for finding the polynomial coefficients. This linear system has the Toeplitz matrix $T_{n+1}[\phi]$ as its coefficient matrix and thus $D_{n+1}[\phi]\neq 0$ implies unique solvability of the linear system. Moreover, if $D_j[\phi] \neq 0$ for all~${j=0,1,2, \ldots}$, these polynomials as constructed above satisfy the bi-orthogonality condition
	\begin{equation}\label{biorthogonality}
		\int_{\T} Q_k(\ze)\widehat{Q}_m\big(\ze^{-1}\big)\phi(\ze)\frac{\dd \ze}{2\pi \ii \ze}= \de_{km}, \qquad k,m \in \N \cup \{0\}.
	\end{equation}

	Assume that $D_k[\phi] \neq 0$ for $k=n-1, n, n+1$ so that (a) $\varkappa_n$ and $\varkappa_{n-1}$ are well defined and nonzero and (b) the polynomials $Q_n$ and $\widehat{Q}_{n-1}$ uniquely exist as described above. Now consider the matrix-valued function
	\begin{equation}\label{Toeplitz-OP-solution}
		X(z;n):=\begin{bmatrix}
			\varkappa_n^{-1} Q_n(z) & \displaystyle \varkappa^{-1}_n \int_{\T} \frac{Q_n(\ze)}{(\ze-z)} \frac{\phi(\ze)\dd \ze}{2\pi \ii \ze^n} \vspace{1mm}\\
			-\varkappa_{n-1}z^{n-1}\widehat{Q}_{n-1}\big(z^{-1}\big) & \displaystyle -\varkappa_{n-1} \int_{\T} \frac{\widehat{Q}_{n-1}\big(\ze^{-1}\big)}{(\ze-z)} \frac{\phi(\ze)\dd \ze}{2\pi \ii \ze}
		\end{bmatrix},
	\end{equation}
	constructed from the polynomials $Q_n$ and $\widehat{Q}_n$. It is due to Baik, Deift, and Johansson \cite{BDJ} that~$X$ as constructed above satisfies the Riemann--Hilbert problem {\rm RH-X1} through {\rm RH-X3}.
	
In the rest of this section, we demonstrate the utilization of the Dodgson condensation identity in reducing \eqref{btdmulti} to a number of bordered Toeplitz determinants, which, in view of the results in~\cite{BEGIL}, paves the way for an effective asymptotic analysis. As this paper aims to provide the general framework, we start with the simplest nontrivial case, which is $m=2$. We will then discuss the recursive nature of our method and how large-size asymptotic analysis for higher values of $m$ can be obtained using essentially the same ideas involved in the case $m=2$.

	Our first objective in this work is to obtain a Riemann--Hilbert representation for the Toeplitz determinants with two borders. To this end, let us assume that $\phi$ is of Szeg\H{o}-type and the border symbols $\psi_1$ and $\psi_2$ are analytic in a neighborhood of the unit circle and consider
	\begin{equation*}
		D^B_n[\phi;\boldsymbol{\psi}_2] = \det \begin{bmatrix}
			\phi_0& \phi_{1} & \cdots & \phi_{n-3} & \psi_{1,n-1} & \psi_{2,n-1} \\
			\phi_{-1}& \phi_0 & \cdots & \phi_{n-4} & \psi_{1,n-2} & \psi_{2,n-2} \\
			\vdots & \vdots & \ddots & \vdots & \vdots & \vdots \\
			\phi_{-n+3} & \phi_{-n+4} & \cdots & \phi_{0} & \psi_{1,2} & \psi_{2,2}\\
			\phi_{-n+2} & \phi_{-n+3} & \cdots & \phi_{-1} & \psi_{1,1} & \psi_{2,1} \\
			\phi_{-n+1} & \phi_{-n+2} & \cdots & \phi_{-2} & \psi_{1,0} & \psi_{2,0}
		\end{bmatrix}.
	\end{equation*}
	For simplicity of notation in this section, we denote $D^B_n[\phi;\boldsymbol{\psi}_2] \equiv \mathscr{D}$.
	Recalling \eqref{DODGSON}, let us consider
	\begin{equation}\label{DODGSON3}
		\mathscr{D} \cdot \mathscr{D}\left\lbrace \begin{matrix} 0 & n-1 \\ n-2& n-1 \end{matrix} \right\rbrace = \mathscr{D}\left\lbrace \begin{matrix} 0 \\ n-2 \end{matrix} \right\rbrace \cdot \mathscr{D}\left\lbrace \begin{matrix} n-1 \\ n-1 \end{matrix} \right\rbrace - \mathscr{D}\left\lbrace \begin{matrix} 0 \\ n-1 \end{matrix} \right\rbrace \cdot \mathscr{D}\left\lbrace \begin{matrix} n-1 \\ n-2 \end{matrix} \right\rbrace,
	\end{equation}
	where
\begin{equation*}\label{zphi}
		\mathscr{D}\left\lbrace \begin{matrix} 0 & n-1 \\ n-2& n-1 \end{matrix} \right\rbrace \equiv D_{n-2}[z \phi]
	\end{equation*} is a pure Toeplitz determinant and all determinants on the right-hand side are bordered Toeplitz determinants. Indeed,
	\begin{equation}\label{easy btds}
		\mathscr{D}\left\lbrace \begin{matrix} n-1 \\ n-1 \end{matrix} \right\rbrace \equiv D^B_{n-1}\big[\phi;z^{-1}\psi_1\big], \qquad \mathscr{D}\left\lbrace \begin{matrix} n-1 \\ n-2 \end{matrix} \right\rbrace \equiv D^B_{n-1}\big[\phi;z^{-1}\psi_2\big].
	\end{equation}
	However, the \textit{bulk symbol} for the other two bordered determinants has a nonzero winding number, more precisely we have
	\begin{equation}\label{harder btds}
		\mathscr{D}\left\lbrace \begin{matrix} 0 \\ n-2 \end{matrix} \right\rbrace \equiv D^B_{n-1}[z\phi;\psi_2], \qquad \mathscr{D}\left\lbrace \begin{matrix} 0 \\ n-1 \end{matrix} \right\rbrace \equiv D^B_{n-1}[z\phi;\psi_1].
	\end{equation}
	Using these, we can rewrite \eqref{DODGSON3} as
	\begin{equation}\label{Dodgson 2-bordered}
		D^B_n[\phi;\boldsymbol{\psi}_2] = D^B_{n-1}\big[\phi;z^{-1}\psi_1\big] \frac{D^B_{n-1}[z\phi;\psi_2]}{D_{n-2}[z \phi]} - D^B_{n-1}\big[\phi;z^{-1}\psi_2\big] \frac{D^B_{n-1}[z\phi;\psi_1]}{D_{n-2}[z \phi]}.
	\end{equation}
	
	\begin{Remark}
		Alternatively, we could consider the following Dodgson condensation identity
		\begin{equation}\label{DODGSON2}
			\mathscr{D} \cdot \mathscr{D}\left\lbrace \begin{matrix} n-2 & n-1 \\ n-2& n-1 \end{matrix} \right\rbrace = \mathscr{D}\left\lbrace \begin{matrix} n-2 \\ n-2 \end{matrix} \right\rbrace \cdot \mathscr{D}\left\lbrace \begin{matrix} n-1 \\ n-1 \end{matrix} \right\rbrace - \mathscr{D}\left\lbrace \begin{matrix} n-2 \\ n-1 \end{matrix} \right\rbrace \cdot \mathscr{D}\left\lbrace \begin{matrix} n-1 \\ n-2 \end{matrix} \right\rbrace.
		\end{equation}
		Notice that $\mathscr{D}\left\lbrace \begin{smallmatrix} n-2 & n-1 \\ n-2& n-1 \end{smallmatrix} \right\rbrace$ is the pure Toeplitz determinant $D_{n-2}[\phi]$, $ \mathscr{D}\left\lbrace \begin{smallmatrix} n-1 \\ n-1 \end{smallmatrix} \right\rbrace$ and $ \mathscr{D}\left\lbrace \begin{smallmatrix} n-1 \\ n-2 \end{smallmatrix} \right\rbrace$ are respectively the bordered Toeplitz determinants
		\[ D^B_{n-1}\big[\phi;z^{-1}\psi_1\big],\qquad D^B_{n-1}\big[\phi;z^{-1}\psi_2\big],\]
		and $\mathscr{D}\left\lbrace \begin{smallmatrix} n-2 \\ n-1 \end{smallmatrix} \right\rbrace$ and $ \mathscr{D}\left\lbrace \begin{smallmatrix} n-2 \\ n-2 \end{smallmatrix} \right\rbrace$ are semi-framed Toeplitz determinants (see Section~\ref{section framed}).\footnote{These are respectively $\mathscr{E}_{n-1}\big[\phi;z^{-2}\psi_1,z^{-2}\phi;\psi_{1,0}\big]$ and $\mathscr{E}_{n-1}\big[\phi;z^{-2}\psi_2,z^{-2}\phi;\psi_{2,0}\big]$, where $\mathscr{E}_n[\phi;\psi,\eta;a]$ is introduced in \eqref{half-framed}.}	Therefore, this DCI has the advantage that we do not need to deal with a bulk symbol with non-zero winding number, but its disadvantage is that it relates two-bordered Toeplitz determinants to semi-framed ones, which are, as we will see in Section~\ref{section framed}, more complicated objects. This is evident in the fact that the bordered Toeplitz determinants are characterized by BOPUC themselves, while the semi-framed ones are characterized by the reproducing kernel of BOPUC (see Theorem \ref{semis in terms of RepKer intro}). However, from the view point of obtaining the desired asymptotics, each of the two DCIs \eqref{DODGSON3} or \eqref{DODGSON2} can be taken as the starting point.
\end{Remark}
	
In the rest of this section, we choose to concentrate on the Dodgson condensation identity~\eqref{DODGSON3}. The asymptotics of the bordered Toeplitz determinants in \eqref{easy btds} can be obtained by rather straight-forward modifications of the findings in \cite{BEGIL}. However, the asymptotics of the bordered Toeplitz determinants in \eqref{harder btds} are more challenging as the bulk symbol $z \phi(z)$ has a~nonzero winding number. Notice that this is an instance of a~\textit{non-degenerate} Fisher Hartwig singularity at $z=1$ with the parameters $\be=1$ and $\al=0$ (see \cite{DIK} for more details). We know that the asymptotics of $D_n[z \phi]$ can be obtained from \cite[Lemma 2.4]{DIK}, which in particular states that
	\begin{equation}\label{DIK zphi}
		D_n[z \phi] = (-1)^n \frac{Q_n(0)}{\varkappa_n}D_n[\phi], \qquad n \geq N_0,
	\end{equation}
	provided that there exists a fixed $N_0 \geq 0$ such that for all $n \geq N_0$ the Toeplitz determinants~$D_n[\phi]$ are nonzero, and $Q_k(0) \neq 0$ for $k=N_0, N_0+1, \dots, n-1$.
	
	However, for the ultimate goal of finding the asymptotics of the right-hand side of \eqref{Dodgson 2-bordered}, it turns out that we do not need to use \eqref{DIK zphi} for our calculations, at least for the symbols $\psi_1$ and~$\psi_2$ of the form \eqref{general psi}--\eqref{q1 q2}. This is because for such symbols we can obtain the asymptotics~of
	\begin{equation}\label{ratios of dets}
		\frac{D^B_{n-1}[z\phi;\psi_2]}{D_{n-2}[z \phi]},\qquad \frac{D^B_{n-1}[z\phi;\psi_1]}{D_{n-2}[z \phi]}
	\end{equation}
	in terms of the solution of an asymptotically tractable Riemann--Hilbert problem. More precisely, to find the asymptotics of ratios in \eqref{ratios of dets}, we need to find the solution of the $X$-RHP when~$\phi$ is replaced by $z\phi$. We call this the $Z$-RHP and using two distinct approaches we prove in Theorems~\ref{thm Z-RHP intro} and \ref{thmXZ intro} how to construct its solution in terms of the solution to the $X$-RHP. Once we have all the above ingredients, we can find the desired asymptotics of $		D^B_n[\phi;\boldsymbol{\psi}_2]$.

\subsection{Proofs of Theorems \ref{thm Z-RHP intro} and \ref{thmXZ intro}}
	
	In this section, we will write $Z(z;n)$ to refer to the solution of the $X$-RHP when $\phi$ is replaced by $z\phi$. More precisely, $Z(z;n)$ satisfies
	\begin{itemize}\itemsep=0pt
		\item RH-Z1: 
$Z(\cdot;n)\colon \C\setminus \T \to \C^{2\times2}$ is analytic,
		\item RH-Z2: 
The limits of $Z(\ze;n)$ as $\ze$ tends to $z \in \T $ from the inside and outside of the unit circle exist, and are denoted $Z_{\pm}(z;n)$ respectively and are related by
		\[
			Z_+(z;n)=Z_-(z;n)\begin{bmatrix}
				1 & z^{-n+1}\phi(z) \\
				0 & 1
			\end{bmatrix}, \qquad z \in \T,
		\]
		
		\item RH-Z3: 
$Z(z;n)=\big( I + O\big(z^{-1}\big) \big) z^{n \sigma_3}$ as $z \to \infty$.
	\end{itemize}
	
	To fix the notation, let us consider the system of bi-orthogonal polynomials on the unit circle~$\{P_k(z)\}^{\infty}_{k=0}$ and $\big\{\widehat{P}_k(z)\big\}^{\infty}_{k=0}$, $\deg P_k = \deg \widehat{P}_k = k$, satisfying
	\begin{equation}\label{orth zphi}
		\int_{\T} P_n(\ze)\ze^{-m}\ze\phi(\ze)\frac{\dd \ze}{2\pi \ii \ze}= \frac{1}{\varkappa_n[z\phi]} \de_{nm}, \qquad m =0, \dots,n,
	\end{equation}
	and
	\[
		\int_{\T} \widehat{P}_n\big(\ze^{-1}\big)\ze^{m}\ze\phi(\ze)\frac{\dd \ze}{2\pi \ii \ze}=\frac{1}{\varkappa_n[z\phi]} \de_{nm}, \qquad m =0, \dots,n,
	\]
	where $\varkappa_n[z\phi] \neq 0$ is the leading coefficient of both $P_n$ and $\widehat{P}_n$. If $D_n[z\phi] \neq 0$ and $D_{n+1}[z\phi] \neq 0$ the polynomials $P_n$ and $\widehat{P}_n$ uniquely exist\footnote{For the same reason described earlier for the existence and uniqueness of $Q_n$ and $\widehat{Q}_n$.} and are given by
	\begin{equation}\label{Toeplitz OP zphi}
		P_n(z):= \frac{1}{\sqrt{D_n[z\phi] D_{n+1}[z\phi]}} \det \begin{bmatrix}
			(z\phi)_0 & (z\phi)_{-1} & \cdots & (z\phi)_{-n} \\
			(z\phi)_1 & (z\phi)_{0} & \cdots & (z\phi)_{-n+1} \\
			\vdots & \vdots & \ddots & \vdots \\
			(z\phi)_{n-1} & (z\phi)_{n-2} & \cdots & (z\phi)_{-1} \\
			1 & z & \cdots & z^n
		\end{bmatrix},
	\end{equation}
	and
	\begin{gather*}
		\widehat{P}_n(z) := \frac{1}{\sqrt{D_n[z\phi] D_{n+1}[z\phi]}} \det \begin{bmatrix}
			(z\phi)_0 & (z\phi)_{-1} & \cdots & (z\phi)_{-n+1} & 1 \\
			(z\phi)_1 & (z\phi)_{0} & \cdots & (z\phi)_{-n+2} & z \\
			\vdots & \vdots & \ddots & \vdots \\
			(z\phi)_{n} & (z\phi)_{n-1} & \cdots & (z\phi)_{1} & z^n
		\end{bmatrix},
	\end{gather*}
	and moreover,
	\begin{equation}\label{kappa zphi}
		\varkappa_n[z\phi] = \sqrt{\frac{D_{n}[z\phi]}{D_{n+1}[z\phi]}}, \qquad n \in \N \cup \{0\}, \qquad D_0[z\phi] \equiv 1.
	\end{equation}

	As expected and similar to the relationship of the $X$-RHP and the polynomials $Q_n$ and $\widehat{Q}_n$, if we assume that $D_k[z\phi] \neq 0$ for $k=n-1, n, n+1$, the following matrix-valued function constructed out of $P$ and $\widehat{P}$ satisfies the $Z$-RHP:
	\begin{equation}\label{Toeplitz-OP-solution zphi}
		Z(z;n)=\begin{bmatrix}
			\di	\frac{1}{\varkappa_n[z\phi]} P_n(z) & \di \frac{1}{\varkappa_n[z\phi]} \int_{\T} \frac{P_n(\ze)}{(\ze-z)} \frac{\ze \phi(\ze)\dd \ze}{2\pi \ii \ze^n} \\[8pt]
			-\varkappa_{n-1}[z\phi]z^{n-1}\widehat{P}_{n-1}\big(z^{-1}\big) & \di -\varkappa_{n-1}[z\phi] \int_{\T} \frac{\widehat{P}_{n-1}\big(\ze^{-1}\big)}{(\ze-z)} \frac{\ze \phi(\ze)\dd \ze}{2\pi \ii \ze}
		\end{bmatrix}.
	\end{equation}
	
However, one can find an explicit relation relating the solution of the $Z$-RHP to the solution of the $X$-RHP which can be directly analyzed by the Deift--Zhou nonlinear steepest descent method \cite{DZ}. This means that from the asymptotic analysis of the X-RHP we can obtain the asymptotics of the $Z$-RHP. One way of making this connection is shown in Theorem \ref{thmXZ intro}, which is based upon shifting in the index $n$. Instead, there is an alternative way\footnote{Based on the idea used in \cite{GI2}.} which will yield a simpler connection between the solution of the $Z$-RHP to the solution of the $X$-RHP. To describe this idea more generally, let us consider the Riemann--Hilbert problem
	\begin{itemize}\itemsep=0pt
		\item RH-Y1: 
$Y(\cdot;n,r)\colon\C\setminus \T \to \C^{2\times2}$ is analytic,
		\item RH-Y2: 
The limits of $Y(\ze;n,r)$ as $\ze$ tends to $z \in \T $ from the inside and outside of the unit circle exist, and are denoted $Y_{\pm}(z;n,r)$ respectively and are related by
		\[
			Y_+(z;n,r)=Y_-(z;n,r)\begin{bmatrix}
				1 & z^{-n+r}\phi(z) \\
				0 & 1
			\end{bmatrix}, \qquad z \in \T,
		\]
		
		\item RH-Y3: 
$Y(z;n,r)=\big( I + O\big(z^{-1}\big) \big) z^{n \sigma_3}$ as $z \to \infty$.
	\end{itemize}
	Using the standard Liouville's theorem arguments, we have the following uniqueness result.
	\begin{Lemma}\label{uniqueness Y}
		The solution of the Riemann--Hilbert problem {\rm RH-Y1}--{\rm RH-Y3} is unique, if it exists.
	\end{Lemma}
	
	Define the function
	\begin{equation}\label{WY}
		W(z;n,r):= Y(z;n,r) \begin{bmatrix}
			1 & 0 \\
			0 & z^{-r}
		\end{bmatrix}.
	\end{equation}
	It can be readily checked that $W(z;n,r)$ satisfies the same jump condition on the unit circle as~$X(z;n)$. Therefore, the function $\mathscr{R}(z;n,r) := W(z;n,r)X^{-1}(z;n)$
	must be a meromorphic function with singular behaviour only at $z=0$ and $\infty$. For a fixed value of $r \in \Z$, one can find the function $\mathscr{R}(z;n,r)$ explicitly in terms of the $X$-RHP data. The idea presented in the proof of the following theorem can be used to connect $Y(z;n,r)$ to $X(n,z)$ for any $r \in \Z$. However, for two reasons we only consider the case $r=1$; firstly, because it is the simplest nontrivial case (besides $r=-1$) for which the main idea can be brought forth, and secondly because it is naturally related to the problem of asymptotic analysis of two-bordered Toeplitz determinants considered in this section.
	
\subsubsection{Proof of Theorem \ref{thm Z-RHP intro}}\label{sec proof of thm 1.6}
	Notice that
	\begin{equation}\label{YZ}
		Y(z;n,1) \equiv Z(z;n).
	\end{equation}

	We can directly see that the behavior of $\mathscr{R}(z;n,1)$ as $z \to 0$ and $z \to \infty$ are respectively given~by
	\begin{align*}
		\mathscr{R}(z;n,1) & = Z(0;n) \begin{bmatrix}
			0 & 0 \\
			0 & z^{-1}
		\end{bmatrix} X^{-1}(0;n) + O(1) \qasq z \to 0, \\
		\mathscr{R}(z;n,1) & = \begin{bmatrix}
			1 & 0 \\
			0 & 0
		\end{bmatrix} + O\big(z^{-1}\big) \qasq z \to \infty.
	\end{align*}
	Therefore, by the Liouville's theorem, we have
	\begin{equation*}
		\mathscr{R}(z;n,1) = Z(0;n) \begin{bmatrix}
			0 & 0 \\
			0 & z^{-1}
		\end{bmatrix} X^{-1}(0;n) + \begin{bmatrix}
			1 & 0 \\
			0 & 0
		\end{bmatrix},
	\end{equation*}
	or
	\begin{equation*}
		W(z;n,1) = \left[ Z(0;n) \begin{bmatrix}
			0 & 0 \\
			0 & z^{-1}
		\end{bmatrix} X^{-1}(0;n) + \begin{bmatrix}
			1 & 0 \\
			0 & 0
		\end{bmatrix} \right] X(z;n).
	\end{equation*}
	In view of \eqref{YZ} and \eqref{WY}, this can be rewritten as
	\begin{equation}\label{ZX}
		Z(z;n) = \left[ Z(0;n) \begin{bmatrix}
			0 & 0 \\
			0 & z^{-1}
		\end{bmatrix} X^{-1}(0;n) + \begin{bmatrix}
			1 & 0 \\
			0 & 0
		\end{bmatrix} \right] X(z;n) \begin{bmatrix}
			1 & 0 \\
			0 & z
		\end{bmatrix}.
	\end{equation}
	Let \begin{equation*}
		Z(0;n) = \begin{bmatrix}
			A & B \\
			C & D
		\end{bmatrix}.
	\end{equation*}
	Since
	\[ Z(0;n) \begin{bmatrix}
		0 & 0 \\
		0 & z^{-1}
	\end{bmatrix} = \begin{bmatrix}
		0 & Bz^{-1} \\
		0 & Dz^{-1}
	\end{bmatrix}, \]
	the formula \eqref{ZX} implies that in order to express $Z(z;n)$ purely in terms of $X$-RHP data, we only need to find the unknowns $B$ and $D$ in terms of data from the $X$-RHP. Indeed, we can do so by requiring the right-hand side of \eqref{ZX} to behave according to {\rm RH-Z3}. In fact, using {\rm RH-X3}, and the fact that\footnote{Note that $\det X(z;n) \equiv 1$.}
	\begin{equation*}
		X^{-1}(0;n) = \begin{bmatrix}
			X_{22}(0;n) & -X_{12}(0;n) \\
			-X_{21}(0;n) & X_{11}(0;n)
		\end{bmatrix},
	\end{equation*}
	we find that the right-hand side of \eqref{ZX} behaves like
	\begin{equation*}
		Z(z;n) = \left[ \begin{bmatrix}
			1 & B X_{11}(0;n)+\overset{\infty}{X}_{1,12}(n) \\
			0 & D X_{11}(0;n)
		\end{bmatrix} + O\big(z^{-1}\big) \right] z^{n\sigma_3}.
	\end{equation*}
	Comparing this with {\rm RH-Z3} yields
	\begin{align*}
		B = - \frac{\overset{\infty}{X}_{1,12}(n)}{X_{11}(0;n)}, \qquad
		D = \frac{1}{X_{11}(0;n)}.
	\end{align*}
	Using these in \eqref{ZX} yields the desired result \eqref{ZX1}.
	
	Now, we show that $X_{11}(0;n)$ is indeed nonzero. Since the $X$-RHP and the $Z$-RHP are special cases of the $Y$-RHP, we work with the more general $Y$-RHP first and then specialize the values of $r$ to discuss $X$ and $Z$ RHPs. For a fixed value of $r \in \Z$, assume that the solution to the Riemann--Hilbert problem {\rm RH-Y1}--{\rm RH-Y3} exists. By Lemma \ref{uniqueness Y}, this solution must be unique. By {\rm RH-Y1} and {\rm RH-Y2}, we conclude that $Y_{11}(z;n,r)$ is an entire function and {\rm RH-Y3} implies that it must be a monic polynomial of degree $n$,
	\begin{equation}\label{Y11 poly}
		Y_{11}(z;n,r) = z^n + \sum_{j=0}^{n} \be_j z^j.
	\end{equation}
Now we focus on $Y_{12}(z;n,r)$. From {\rm RH-Y2} and the Plemelj--Sokhotskii formula, we find
	\begin{equation}\label{Y12}
		Y_{12}(z;n,r) = \int_{\T} \frac{Y_{11}(s;n,r)s^{-n+r}\phi(s)}{s-z} \frac{\dd s}{2 \pi \ii}.
	\end{equation}
	From {\rm RH-Y3} and the asymptotic expansion of $Y_{12}(z;n,r)$ for large $|z|$, we find that $Y_{11}(z;n,r)$ must satisfy the orthogonality conditions
	\begin{equation}\label{ort conds}
		\int_{\T} Y_{11}(s;n,r)s^{r-\ell}\phi(s) \frac{\dd s}{2 \pi \ii s} = 0, \qquad \mbox{for} \quad \ell = 0,1, \dots, n-1.
	\end{equation}
	These orthogonality conditions give a linear system for determining the coefficients $\be_j$ in \eqref{Y11 poly}. The coefficient matrix for this linear system is precisely~${T_n[z^r\cdot \phi]}$. Since we have assumed~$Y$ has a~solution, this linear system must have a solution and since~$Y$ has a unique solution by Lemma~\ref{uniqueness Y}, this linear system must have a unique solution as well. This implies that~${T_n[z^r\cdot \phi]}$ is invertible and thus $D_{n}[z^r\phi] \neq 0$.
	In particular, assuming the existence of the solution ${X(z;n) \equiv Y(z;n,0)}$ of the Riemann--Hilbert problem {\rm RH-X1} through {\rm RH-X3} and the existence of the solution~${Z(z;n) \equiv Y(z;n,1)}$ of the Riemann--Hilbert problem {\rm RH-Z1} through {\rm RH-Z3} imply that $D_n[\phi] \neq 0$ and $D_n[z\phi] \neq 0$.
In particular, assuming the existence of the solution $X(z;n) \equiv Y(z;n,0)$ of the Riemann--Hilbert problem {\rm RH-X1} through {\rm RH-X3} and the existence of the solution $Z(z;n) \equiv Y(z;n,1)$ of the Riemann--Hilbert problem {\rm RH-Z1} through {\rm RH-Z3} imply that
	\begin{equation}\label{nonzero dets}
		D_{n}[\phi] \neq 0 \qandq D_n[z\phi] \neq 0.
	\end{equation}
For $r=0$, the unique monic polynomial of degree $n$ satisfying the orthogonality conditions~\eqref{ort conds} is%
	\begin{equation}\label{X11}
		Y_{11}(z;n,0)=X_{11}(z;n) = \frac{1}{D_n[\phi]} \det \begin{bmatrix}
			\phi_0 & \phi_{-1} & \cdots & \phi_{-n} \\
			\phi_1 & \phi_{0} & \cdots & \phi_{-n+1} \\
			\vdots & \vdots & \ddots & \vdots \\
			\phi_{n-1} & \phi_{n-2} & \cdots & \phi_{-1} \\
			1 & z & \cdots & z^n
		\end{bmatrix}.
	\end{equation}
	Notice that
	\begin{equation*}
		X_{11}(0;n) = (-1)^n\frac{D_n[z\phi]}{D_n[\phi]},
	\end{equation*}
	which is nonzero and well defined due to \eqref{nonzero dets}. This finishes the proof of Theorem \ref{thm Z-RHP intro}.

\subsubsection{Proof of Theorem \ref{thmXZ intro}}\label{sec proof of thm 1.8}
	
	Recalling {\rm RH-X1}, it is obvious that $Z(z;n)$ as defined by \eqref{Z in terms of X1} satisfies {\rm RH-Z1}. From \eqref{Z in terms of X1}, it is clear that $Z(z;n)$ and $X(z;n-1)$ satisfy the same jump condition on the unit circle since~${X(z;n-1)}$ is multiplied by a holomorphic function on the left. Notice that from {\rm RH-X2}, we have
	\begin{equation*}
		\begin{split}
			X^{-1}_-(z;n-1)X_+(z;n-1) = \begin{bmatrix}
				1 & z^{-n+1}\phi(z) \\
				0 & 1
			\end{bmatrix},
		\end{split}
	\end{equation*}
	and therefore $Z(z;n)$ as defined by \eqref{Z in terms of X1} satisfies {\rm RH-Z2}. Recalling {\rm RH-X3}, as $z \to \infty$ for the right-hand side of \eqref{Z in terms of X1} we have
	\begin{align*}
			\text{r.h.s. of \eqref{Z in terms of X1}} ={}&
			\begin{bmatrix}
				z + \overset{\infty}{X}_{1,22}(n-1) - \tfrac{\overset{\infty}{X}_{2,12}(n-1)}{\overset{\infty}{X}_{1,12}(n-1)} & -\overset{\infty}{X}_{1,12}(n-1) \\
				\tfrac{1}{\overset{\infty}{X}_{1,12}(n-1)} & 0
			\end{bmatrix} \\ & \times \left(\di I+\frac{ \overset{\infty}{X}_1(n-1)}{z}+\frac{\overset{\infty}{X}_2(n-1)}{z^2} + O\big(z^{-3}\big)\right) \begin{bmatrix}
				z^{-1} & 0 \\
				0 & z
			\end{bmatrix} z^{n \sigma_3} \\ ={}& \left( I + O\big(z^{-1}\big) \right) z^{n \sigma_3}.
	\end{align*}
	Therefore, $Z(z;n)$ as defined by \eqref{Z in terms of X1} satisfies {\rm RH-Z3} as well, and hence is the unique solution of the $Z$-RHP.
	
	Finally, we show that $\overset{\infty}{X}_{1,12}(n-1)$ is indeed nonzero. From {\rm RH-X3}, \eqref{Y12} and \eqref{ort conds}, we find
	\begin{equation*}
		\overset{\infty}{X}_{1,12}(n) = - 	\int_{\T} X_{11}(s;n)s\phi(s) \frac{\dd s}{2 \pi \ii s}.
	\end{equation*}
	From this and \eqref{X11}, we have the determinantal representation
	\begin{equation*}
		\overset{\infty}{X}_{1,12}(n) = - \frac{1}{D_n[\phi]}	\int_{\T} \det \begin{bmatrix}
			\phi_0 & \phi_{-1} & \cdots & \phi_{-n} \\
			\phi_1 & \phi_{0} & \cdots & \phi_{-n+1} \\
			\vdots & \vdots & \ddots & \vdots \\
			\phi_{n-1} & \phi_{n-2} & \cdots & \phi_{-1} \\
			1 & s & \cdots & s^n
		\end{bmatrix}s\phi(s) \frac{\dd s}{2 \pi \ii s} = (-1)^{n+1} \frac{D_{n+1}[z\phi]}{D_n[\phi]},
	\end{equation*} by integrating along the last row and performing $n$ adjacent row swaps. Therefore,
	\begin{equation*}
		\overset{\infty}{X}_{1,12}(n-1) = (-1)^{n} \frac{D_{n}[z\phi]}{D_{n-1}[\phi]}.
	\end{equation*}
	By the exact same argument presented in Section~\ref{sec proof of thm 1.6}, we know that assuming the existence of the solution $X(z;n-1) \equiv Y(z;n-1,0)$ of the Riemann--Hilbert problem {\rm RH-X1} through {\rm RH-X3} (with parameter $n-1$) and the existence of the solution $Z(z;n) \equiv Y(z;n,1)$ of the Riemann--Hilbert problem {\rm RH-Z1} through {\rm RH-Z3} imply that
	\begin{equation*}
		D_{n-1}[\phi] \neq 0 \qandq D_n[z\phi] \neq 0.
	\end{equation*}
	This means that \smash{$\overset{\infty}{X}_{1,12}(n-1) $} is nonzero and well defined. We have thus concluded the proof of Theorem \ref{thmXZ intro}.
	\begin{Remark} 	It is worthwhile to highlight that we would prefer \eqref{ZX1} over \eqref{Z in terms of X1} because in~\eqref{ZX1}, we only need to use data from one subleading term, \smash{$\overset{\infty}{X}_{1}$}, while in \eqref{Z in terms of X1}, we also need to extract data from \smash{$\overset{\infty}{X}_{2}$}. The compatibility of these two solutions is expected to give rise to identities involving $Q_n$ and $\widehat{Q}$. In this case, as expected, these identities are exactly the well-known recurrence relations for the system of bi-orthogonal polynomials on the unit circle. A~new proof for these identities is presented in Lemma~\ref{lemma 3term}.
	\end{Remark}
	
	\subsection{Proof of Theorem \ref{main thm 2-bordered}}\label{Sec 2.2}
	\subsubsection[Bordered Toeplitz determinants of the type \protect{$D^B_{n}\big[\phi;z^{-1}(q_1\phi+q_2)\big]$}]{Bordered Toeplitz determinants of the type $\boldsymbol{ D^B_{n}\big[\phi;z^{-1}(q_1\phi+q_2)\big]}$}
	Let us first recall the following elementary properties of the bordered Toeplitz determinants:
	\begin{gather}
		D^B_n\Bigg[\phi; \sum^m_{j=1} a_j \psi_j\Bigg] =\sum^m_{j=1} a_j D^B_n[\phi,\psi_j],\label{linear-combination}
\\
		D^B_{n}[\phi; \phi] = D_{n}[\phi],\label{border,phi,phi}
\\
		D^B_{n}[\phi;1] = D_{n-1}[\phi].\label{border,phi,const}
	\end{gather}

	Let us denote \begin{equation}\label{q0}
		q_0(z):= \frac{1}{z-c} \qandq \psi_0(z):=q_0(z)\phi(z).\end{equation}
	As a first step, it is useful to recall the description of $D^B_{N}[\phi;q_1\phi+q_2]$ in terms of the solution of the $X$-RHP as shown in \cite{BEGIL} which allows for an effective asymptotic analysis of such bordered Toeplitz determinants.
	
	\begin{Lemma}[\cite{BEGIL}]\label{lemma 3.2}
		The bordered Toeplitz determinant $D^B_{n+1}[\phi,q_0]$, is encoded into $X$-RHP data described by
		\begin{equation}\label{psi=1/z-c 0}
			D^B_{n+1}[\phi;q_0] = \begin{cases}
				0, & |c|<1, \\
				-c^{-n-1}D_{n}[\phi]X_{11}(c;n), & |c|>1,
			\end{cases}
		\end{equation}
		where $D_{n}[\phi]$ is given by \eqref{ToeplitzDet} and $X_{11}$ is the $11$ entry of the solution to {\rm RH-X1} through {\rm RH-X3}.
	\end{Lemma}
	\begin{Corollary}[\cite{BEGIL}]\label{cor 1}
		We have
		\begin{gather*}
			D^B_{n+1}\Bigg[\phi;\di a+\frac{b_0}{z}+\sum^{m}_{j=1}\frac{b_j }{z-c_j}\Bigg] = D_n[\phi] \Bigg(a -\sum^{m}_{j=1 \atop |c_j|>1 }b_j c^{-n-1}_jX_{11}(c_j;n)\Bigg),
		\end{gather*}
		and for a Szeg\H{o}-type $\phi$
		\begin{gather*}
			D^B_{n+1}\Bigg[\phi;\di a+\frac{b_0}{z}+\sum^{m}_{j=1}\frac{b_j }{z-c_j}\Bigg] = G[\phi]^{n} E[\phi] \Bigg(a -\sum^{m}_{j=1 \atop |c_j|>1 }\frac{b_j}{c_j} \al(c_j)\Bigg)( 1 + O({\rm e}^{-\mathfrak{c}n})),
		\end{gather*}
		as $n \to \infty$, where $\al$ is given by \eqref{al}, and the constants $G[\phi]$ and $E[\phi]$ are given by \eqref{G and E} and~$\mathfrak{c}$ is some positive constant.
	\end{Corollary}
	
	\begin{Lemma}[\cite{BEGIL}]\label{psi=z}
		Let $\phi$ be of Szeg\H{o}-type. Then, as $n \to \infty$, we have
		\begin{gather*}
			D^B_{n+1}[\phi;z] = D_n[\phi] \left(-\frac{1}{2\pi \ii} \int_{\T} \ln(\phi(\tau)) \dd \tau + O({\rm e}^{-\mathfrak{c}n})\right),
		\end{gather*}
		for some positive constant $\mathfrak{c}$.
	\end{Lemma}
	
	\begin{Lemma}[\cite{BEGIL}]\label{lem2.2} Let $\psi_0$ be as defined in \eqref{q0} with $c \neq 0$. Then the bordered Toeplitz determinant $D^{B}_n[\phi;\psi_0]$ can be written in terms of the following data from the solution of the X-RHP:
		\begin{gather*}
			D^{B}_{n+1}[\phi;\psi_0] = - \frac{1}{c}D_{n+1}[\phi]+ \frac{1}{c} D_{n}[\phi] X_{12}(c;n),
		\end{gather*}
		where $D_{n}[\phi]$ is given by \eqref{ToeplitzDet} and $X_{12}$ is the $12$ entry of the solution to {\rm RH-X1} through {\rm RH-X3}.
	\end{Lemma}
	\begin{Corollary}[\cite{BEGIL}]\label{cor 2}
		We have
		\begin{gather*}
			D^B_{n+1}\Bigg[\phi;\di\Bigg( a+\sum^{m}_{j=1}\frac{b_j z}{z-c_j}\Bigg)\phi \Bigg] =aD_{n+1}[\phi]+ D_n[\phi] \sum^{m}_{j=1}b_j X_{12}(c_j;n),
		\end{gather*}
		and for a Szeg\H{o}-type $\phi$
		\begin{gather*}
			D^B_{n+1}\Bigg[\phi;\di\Bigg( a+\sum^{m}_{j=1}\frac{b_j z}{z-c_j}\Bigg)\phi \Bigg] = G[\phi]^{n+1} E[\phi] \Bigg(a+\frac{1}{G[\phi]}\sum^{m}_{j=1 \atop |c_j|<1 }b_j \al(c_j)\Bigg)( 1 + O({\rm e}^{-\mathfrak{c}n})),
		\end{gather*}
		as $n \to \infty$, where $\al$ is defined in \eqref{al}, $G[\phi]$ and $E[\phi]$ are given by \eqref{G and E}, and $\mathfrak{c}$ is some positive constant.
	\end{Corollary}
	
	We choose to follow the notations introduced in \cite{BEGIL} for a smoother navigation between the papers. Recalling \eqref{q1 q2}, we have
	\[ 	z^{-1}q_1(z) = a_1+ \frac{a_0}{z} +\frac{b_0}{z^2}+\sum^{m}_{j=1}\frac{b_j}{z-c_j}, \qquad z^{-1}q_2(z) = \hat{a}_1 + \frac{\hat{d}_0 }{z}+\frac{\hat{b}_0}{z^2}+
	\sum_{j=1}^m \frac{\hat{d}_j}{z-c_j}, \]
	where
	\[ \hat{d}_0 = \hat{a}_0 - \sum_{j=1}^{m} \hat{b}_jc^{-1}_j \qandq \hat{d}_j = \hat{b}_jc^{-1}_j. \]
	All contributions from all terms in $z^{-1}q_1$ and $z^{-1}q_2$ are expressed in lemmas above and other results in \cite{BEGIL}, except for the contribution from $b_0z^{-2}$ in $z^{-1}q_1$. Notice that the term $\hat{b}_0z^{-2}$ in~$z^{-1}q_2$ does not contribute due to \cite[Lemma 2.1]{BEGIL}.
	
The contribution from $b_0z^{-2}$ in $z^{-1}q_1$ corresponds to the bordered Toeplitz determinant $D^{B}_N\big[\phi;z^{-2}\phi\big]$. More generally, we prove the following theorem, which characterizes $X_{12}(z;n)$ as the generating function for the objects $D^B_{n+1}\big[\phi;z^{-\ell}\phi\big]/D_n[\phi]$, $\ell \in \N$.
	
	\begin{Theorem}\label{thm border z^-l phi}
		Let $\ell \in \N$, and suppose that the solution $X(z;n)$ of the Riemann--Hilbert problem {\rm RH-X1} through {\rm RH-X3} exists. The coefficient of $z^{\ell}$ in the Taylor expansion of $X_{12}(z;n)$ centered at zero, is precisely the object $D^B_{n+1}\big[\phi;z^{-\ell}\phi\big]/D_n[\phi]$. In other words,
		\begin{equation*}
			D^B_{n+1}\big[\phi;z^{-\ell}\phi\big] = \frac{D_n[\phi]}{\ell!} \frac{\dd ^{\ell}}{\dd z^{\ell}}X_{12}(z;n) \bigg|_{z=0},
		\end{equation*}
		where $X_{12}(z;n)$ is the $12$-entry in the solution of the Riemann--Hilbert problem {\rm RH-X1} through {\rm RH-X3}.
	\end{Theorem}
	
	\begin{proof}
		Notice that
		\begin{equation*}
			\frac{\dd^{\ell}}{\dd z^{\ell}} \left( \ze - z \right)^{-1} \bigg|_{z=0} = \ell ! \ze^{-\ell-1}.
		\end{equation*}
		Therefore, from \eqref{Y12} and \eqref{X11}, we have
		\begin{equation*}
			\begin{split}
				\frac{D_n[\phi]}{\ell!} \frac{\dd ^{\ell}}{\dd z^{\ell}}X_{12}(z;n) \bigg|_{z=0} & = D_n[\phi] \int_{\T} X_{11}(\ze;n) \ze^{-\ell-n} \phi(z) \frac{\dd \ze}{2 \pi \ii \ze}. \\
				& = \int_{\T} \det \begin{bmatrix}
					\phi_0 & \phi_{-1} & \cdots & \phi_{-n} \\
					\phi_1 & \phi_{0} & \cdots & \phi_{-n+1} \\
					\vdots & \vdots & \ddots & \vdots \\
					\phi_{n-1} & \phi_{n-2} & \cdots & \phi_{-1} \\
					1 & \ze & \cdots & \ze^n
				\end{bmatrix} \ze^{-n-\ell} \phi(\ze) \frac{\dd \ze}{2 \pi \ii \ze} \\ & = \det \begin{bmatrix}
					\phi_0 & \phi_{-1} & \cdots & \phi_{-n} \\
					\phi_1 & \phi_0 & \cdots & \phi_{-n+1} \\
					\vdots & \vdots & \ddots & \vdots \\
					\phi_{n-1} & \phi_{n-2} & \cdots & \phi_{-1} \\
					\int_{\T} \frac{ \ze^{-\ell} \phi(\ze)}{\ze^{n}} \frac{\dd \ze}{2 \pi \ii \ze} & \int_{\T} \frac{\ze^{-\ell} \phi(\ze)}{\ze^{n-1}} \frac{\dd \ze}{2 \pi \ii \ze} & \cdots & \int_{\T} \ze^{-\ell} \phi(\ze) \frac{\dd \ze}{2 \pi \ii \ze}
				\end{bmatrix} \\ & = \det \begin{bmatrix}
					\phi_0 & \phi_{-1} & \cdots & \phi_{-n} \\
					\phi_1 & \phi_{0} & \cdots & \phi_{-n+1} \\
					\vdots & \vdots & \ddots & \vdots \\
					\phi_{n-1} & \phi_{n-2} & \cdots & \phi_{-1} \\
					\big[z^{-\ell}\phi\big]_n & \big[z^{-\ell}\phi\big]_{n-1} & \cdots & \big[z^{-\ell}\phi\big]_0
				\end{bmatrix} = D^B_{n+1}\big[\phi;z^{-\ell}\phi\big].
			\end{split}
		\end{equation*}
		In the above we have used that $D_n[\phi] \neq 0$ which is implied by the existence of the solution to the X-RHP as elaborated in the proof of Theorem~\ref{thm Z-RHP intro}.
	\end{proof}

	The following asymptotic result simply follows from the asymptotic analysis of the X-RHP, see~\eqref{X in terms of R exact}.
	
	\begin{Corollary}\label{Cor phi z^{-l}phi}
		For a Szeg\H{o}-type symbol, we have
		\begin{equation*}
			D^B_{n+1}\big[\phi;z^{-\ell}\phi\big] = G[\phi]^{n} E[\phi] \left( \frac{\al^{(\ell)}(0)}{\ell!} + O({\rm e}^{-\mathfrak{c}n}) \right), \qasq n \to \infty,
		\end{equation*}
		where $\al$ is given by \eqref{al}, and the constants $G[\phi]$ and $E[\phi]$ are given by \eqref{G and E} and $\mathfrak{c}$ is some positive constant.
	\end{Corollary}

	\begin{Remark}
		\normalfont This is in agreement with \cite[Lemma 2.7]{BEGIL}. For $\ell=1$, Notice that $\al'(0)= [\log \phi]_1 \al(0) = [\log \phi]_1 G[\phi]$.
	\end{Remark}
	
	\begin{Lemma}\label{main lem 1}
		Let $\phi$ be of Szeg\H{o}-type and the rational functions $q_1$ and $q_2$ be given by \eqref{q1 q2}. Then, the following asymptotic behavior of $D^B_{n}\big[\phi;z^{-1}(q_1\phi+q_2)\big]$ as $n \to \infty$ takes place:
		\begin{gather*}
				D^B_{n}\big[\phi;z^{-1}(q_1\phi+q_2)\big] =
				G[\phi]^{n} E[\phi]\left(H[\phi;\psi] + O({\rm e}^{-\mathfrak{c}n})\right),
			\end{gather*}
			where $G[\phi]$ and $E[\phi]$ are given by \eqref{G and E}, $H[\phi;\psi]$ is given by
			\begin{align}
					H[\phi;\psi] ={}&	
					a_1 - \sum_{j=1}^{m} \frac{b_j}{c_j}+ a_0 [\log \phi]_{1}+ b_0 [\log \phi]_{2} + \frac{b_0}{2} [\log \phi]_{1}^2\nonumber
					\\ & +
					\frac{1}{\al(0)}\Bigg(
					\hat{a}_1-
					 \sum^{m}_{j=1 \atop |c_j|>1 }\frac{\hat{b}_j}{c^2_j} \al(c_j) + 		\sum^m_{j=1 \atop 0<|c_j|<1} \frac{b_j}{c_j} \al(c_j) \Bigg), \label{ConstantH}
			\end{align}
			and	$\al$ is given by \eqref{al}, and $\mathfrak{c}$ is some positive constant.
		\end{Lemma}
		\begin{proof}
			The proof is straight-forward, using \eqref{linear-combination}, \eqref{border,phi,phi}, \eqref{border,phi,const}, lemmas cited above from \cite{BEGIL} and Corollary \ref{Cor phi z^{-l}phi}.
		\end{proof}
		
\subsubsection[Bordered Toeplitz determinants of the type \protect{$D^B_{n}[z\phi;q_1\phi+q_2]$}]{Bordered Toeplitz determinants of the type $\boldsymbol{D^B_{n}[z\phi;q_1\phi+q_2]}$}

		Now we focus on computing the asymptotics of $ D^B_{n}[z\phi;q_1\phi+q_2] $ for the rational functions~$q_1(z)$ and~$q_2(z)$ given by \eqref{q1 q2}. In view of \eqref{linear-combination}, we need to express the following bordered Toeplitz determinants in terms of the data from the $Z$-RHP: $D^B_{n}[z\phi;\phi]$,
		$D^B_{n}\big[z\phi;\di \tfrac{1}{z}\phi\big]$,
		$D^B_{n}\big[z\phi;\di \tfrac{z\phi}{z-c}\big]$ with ${c \neq 0}$,
		$D^B_{n}[z\phi;z]$, and
		$D^B_{n}\big[z\phi;\di \tfrac{1}{z-c}\big]$ with $c \neq 0$.	Notice that we already know that $D^B_{N}\big[z\phi;z^k\big]=0$, $k \in \Z \setminus \{0,1,\dots,n\}$, since the Fourier coefficients $(z^k)_j=0$ for $0\leq j \leq n$, $k \in \Z \setminus \{0,1,\dots,n\}$.
		
		Regarding the first two bordered Toeplitz determinants in the above list, we can use the following generalization of Theorem \ref{thm border z^-l phi} which can be proven identically.
		
		\begin{Lemma}
			Let $r \in \Z$ and $\ell \in \N$. Suppose that the solution to the Riemann--Hilbert problem {\rm RH-Y1}--{\rm RH-Y3} exists.\footnote{This implies that $D_n[z^r\phi]\neq 0$ as elaborated in the proof of Theorem \ref{thm Z-RHP intro}.} The coefficient of $z^{\ell}$ in the Taylor expansion of $Y_{12}(z;n,r)$ centered at zero, is precisely the object $D^B_{n+1}\big[z^r\phi;z^{r-\ell}\phi\big]/D_n[z^r\phi]$. In other words,
			\begin{equation*}
				D^B_{n+1}\big[z^{r}\phi;z^{r-\ell}\phi\big] = \frac{D_n[z^{r}\phi]}{\ell!} \frac{\dd ^{\ell}}{\dd z^{\ell}}Y_{12}(z;n,r) \Bigg|_{z=0},
			\end{equation*}
			where $Y_{12}(z;n,r)$ is the $12$-entry in the solution of the Riemann--Hilbert problem {\rm RH-Y1}--{\rm RH-Y3}.
		\end{Lemma}
		
		\begin{Corollary}\label{Cor zphi phi and phi/z}
			Suppose that the solution $X(z;n)$ of the Riemann--Hilbert problem {\rm RH-X1} through {\rm RH-X3} and the solution $Z(z;n)$ of the Riemann--Hilbert problem {\rm RH-Z1} through {\rm RH-Z3} exist. We have the following characterizations for $D^B_{n+1}[z\phi,\phi]$ and $D^B_{n+1}\big[z\phi,\frac{1}{z}\phi\big]$ in terms of the solution of the Riemann--Hilbert problem {\rm RH-Z1}--{\rm RH-Z3}:
			\begin{gather*}
				D^B_{n+1}[z\phi,\phi] = D_n[z\phi] \frac{\dd}{\dd z}Z_{12}(z;n) \Bigg|_{z=0}, \\
				D^B_{n+1}\left[z\phi,\frac{1}{z}\phi\right] = \frac{D_n[z\phi]}{2} \frac{\dd^2}{\dd z^2}Z_{12}(z;n) \Bigg|_{z=0}.
			\end{gather*}
			Moreover, if $\phi$ is a Szeg\H{o}-type symbol and there exists $N_0 \in \N$ such that for all $n \geq N_0$, we have~${|C_n[\phi]|>0}$, with
			\begin{equation}\label{Cnphi}
				C_n[\phi] := \frac{1}{2 \pi \ii} \int_{\Ga_0} \tau^n \phi^{-1}(\tau) \al^2(\tau) \dd \tau,
			\end{equation}
			then
			\begin{align}
				&\frac{D^B_{n+1}[z\phi,\phi]}{D_n[z\phi]}	 = G[\phi] \left( 1 - [\log \phi]_1 \frac{C_n[\phi]}{C_{n-1}[\phi]} \right) \big(1 + O\big(\rho^{-2n}\big)\big), \label{asymp zphi phi} \\	
&	\frac{D^B_{n+1}\big[z\phi,\frac{1}{z}\phi\big]}{D_n[z\phi]}	 = G[\phi] \left( [\log \phi]_1 - \left( [\log \phi]_2 +\frac{[\log \phi]^2_1}{2}\right) \frac{C_n[\phi]}{C_{n-1}[\phi]} \right) \big(1 + O\big(\rho^{-2n}\big)\big), \label{asymp zphi phi/z}
			\end{align}
			where $\al$ is given by \eqref{al}, and $\Ga_0$ is a counter-clockwise circle with radius $\rho^{-1}<1$. The number~${\rho>1}$ is chosen such that $\phi$ is analytic in the annulus $\{z\colon \rho^{-1}<|z|<\rho\}$.
		\end{Corollary}
		\begin{proof}
			Using \eqref{ZX1}, we have
			\begin{equation}\label{Z12}
				Z_{12}(z;n) = (\mathscr{B}(n)+z) X_{12}(z;n) -\overset{\infty}{X}_{1,12}(n) X_{22}(z;n),
			\end{equation}
			where for the simplicity of notations we have introduced
			\begin{equation}\label{Bn}
				\mathscr{B}(n) := \frac{\overset{\infty}{X}_{1,12}(n)X_{21}(0;n)}{X_{11}(0;n)}.
			\end{equation}
			Using \eqref{X in terms of R exact}, as $n \to \infty$, uniformly for $z\in \Om_0$, we have
			\begin{align}
				X_{11}(z;n) & = - R_{1,12}(z;n) \al^{-1}(z) \left( 1 + O \left( \frac{\rho^{-2n}}{1+|z|} \right) \right), \label{X11 Om0} \\
				X_{12}(z;n) & = \al(z) \left( 1 + O \left( \frac{\rho^{-2n}}{1+|z|} \right) \right), \nonumber \\
				X_{21}(z;n) & = - \al^{-1}(z) \left( 1 + O \left( \frac{\rho^{-2n}}{1+|z|} \right) \right), \label{X21 Om0}\\
				X_{22}(z;n) & = R_{1,21}(z;n) \al(z) \left( 1 + O \left( \frac{\rho^{-2n}}{1+|z|} \right) \right), \label{X22 Om0}
			\end{align}
			where $\al$ and $R_1$ are respectively given by \eqref{45} and \eqref{R1} and $\rho^{-1}$ is the radius of the circle $\Ga_0$ shown in Figure \ref{S_contour}. First, let us consider the large-$n$ behavior of $\mathscr{B}(n)$ given by \eqref{Bn}. Let us first focus on \smash{$\overset{\infty}{X}_{1,12}(n)$}. From {\rm RH-X3}, we have
			\begin{equation*}
				\overset{\infty}{X}_{1}(n) = \lim_{z \to \infty} z( X(z;n)z^{-n\sigma_3} - I ).
			\end{equation*}
			From this, and recalling \eqref{X in terms of R exact} for $z \in \Om_{\infty}$, \eqref{R asymp} and the fact that $\al(z) \to 1$ as $z \to \infty$, we find
			\begin{align}
				\overset{\infty}{X}_{1,12}(n) &= \lim_{z \to \infty} z R_{1,12}(z) + O\big(\rho^{-3n}\big) \nonumber\\
&= \frac{1}{2 \pi \ii} \int_{\Ga_0} \tau^n \phi^{-1}(\tau) \al^2(\tau) \dd \tau \times \big(1 + O\big(\rho^{-2n}\big)\big) = O(\rho^{-n}),\label{X1,12 infty}
			\end{align}
			where we have used \eqref{R_k's are small}, \eqref{R1} and the fact that $R_{2\ell}(z;n)$ is diagonal and $R_{2\ell+1}(z;n)$ is off-diagonal, $\ell \in \N \cup \{0\}$. From \eqref{X11 Om0}, \eqref{X21 Om0} and \eqref{X1,12 infty}, we obtain
			\begin{equation}\label{BBB}
				\mathscr{B}(n) = - \frac{C_n[\phi]}{C_{n-1}[\phi]} \times \big(1 + O\big(\rho^{-2n}\big)\big),
			\end{equation}
			where $	C_n[\phi] \equiv -R_{1,12}(0;n+1) = O(\rho^{-n})$ is given by \eqref{Cnphi},
			see \eqref{R_k's are small}. From \eqref{Z12},
			\begin{gather*}
				\frac{\dd}{\dd z}Z_{12}(z;n) \Bigg|_{z=0} = X_{12}(0;n) + \mathscr{B}(n) 				\frac{\dd}{\dd z}X_{12}(z;n) \Bigg|_{z=0} + O(\rho^{-2n}) \\
\phantom{\frac{\dd}{\dd z}Z_{12}(z;n) \Bigg|_{z=0}}{} = ( \al(0) + \mathscr{B}(n) \al'(0) ) \big(1+ O\big(\rho^{-2n}\big)\big), \\
				\frac{\dd^2}{\dd z^2}Z_{12}(z;n) \Bigg|_{z=0} =2\frac{\dd}{\dd z}X_{12}(z;n) \Bigg|_{z=0} + \mathscr{B}(n) 				\frac{\dd^2}{\dd z^2}X_{12}(z;n) \Bigg|_{z=0} + O(\rho^{-2n}) \\
 \phantom{\frac{\dd^2}{\dd z^2}Z_{12}(z;n) \Bigg|_{z=0}}{}= ( 2\al'(0) + \mathscr{B}(n) \al''(0) ) \big(1+ O\big(\rho^{-2n}\big)\big),
			\end{gather*}
			recalling that \smash{$\overset{\infty}{X}_{1,12}(n) = O(\rho^{-n})$} by \eqref{X1,12 infty}, and that $X_{22}(z;n) = O(\rho^{-n}/(1+|z|))$ by \eqref{X22 Om0} and~\eqref{R_k's are small}. Finally, we arrive at \eqref{asymp zphi phi} and \eqref{asymp zphi phi/z} recalling \eqref{G and E}, \eqref{al} and observing that
			\begin{align*}
				\al(0) & = G[\phi], \qquad
				\al'(0) = G[\phi] \cdot [\log \phi]_1, \qquad
				\al''(0) = G[\phi] \cdot \big( 2 [\log \phi]_2 + [\log \phi]^2_1 \big).	\tag*{\qed}			
			\end{align*}\renewcommand{\qed}{}
		\end{proof}

\begin{Remark}
The asymptotic behavior of $C_n[\phi]$, and hence that of $\mathscr{B}_n[\phi]$, depends on the analytic properties of the symbol $\phi$, and thus detailed behavior can only be obtained for a~concrete symbol. For instance, for $\widehat{\phi}(z)$ given by \eqref{hat psi}, similar asymptotics were computed in~\cite[Proposition~2.10]{BEGIL}. Nevertheless, we will see that the knowledge of this asymptotics is not going to be needed for the leading order asymptotics of $D^B_n[\phi;\boldsymbol{\psi}_2]$. Proposition~2.10 of~\cite{BEGIL} also highlights that the collection of Szeg\H{o}-type symbols $\phi$ satisfying $|C_n[\phi]|>0$ for sufficiently large $n$, is indeed nonempty.
\end{Remark}
		
Similar to Lemma \ref{lem2.2}, we can prove the analogous result when $\phi$ is replaced by $z\phi$.
		
\begin{Lemma}\label{lemma zphi/z-c} Let \begin{equation*}
				\eta_0(z):= \frac{z\phi}{z-c} \qquad \mbox{with} \quad c \neq 0.
			\end{equation*} Then the bordered Toeplitz determinant $D^{B}_n[z\phi;\eta_0]$ can be written in terms of the following data from the solution of the Z-RHP:
			\begin{equation}\label{bdToep-RHP-eta0}
				D^{B}_{n+1}[z\phi;\eta_0] = - \frac{1}{c}D_{n+1}[z\phi]+ \frac{1}{c} D_{n}[z\phi] Z_{12}(c;n),
			\end{equation}
			where $D_{n}[\phi]$ is given by \eqref{ToeplitzDet} and $Z_{12}$ is the $12$ entry of the solution to {\rm RH-Z1} through {\rm RH-Z3}.
		\end{Lemma}
		
		\begin{Corollary}\label{Corollary zphi/z-c}
			It holds that
			\begin{equation}\label{bbb}
				\frac{D^{B}_{n+1}[z\phi;\eta_0]}{D_{n}[z\phi]} = \frac{Z_{12}(c;n) - Z_{12}(0;n)}{c},
			\end{equation}
			and moreover, if $\phi$ is a Szeg\H{o}-type symbol, we have
			\begin{align}
					\frac{D^{B}_{n+1}[z\phi;\eta_0]}{D_{n}[z\phi]} ={}& \di \frac{G[\phi]}{c} \frac{C_n[\phi]}{C_{n-1}[\phi]} \times \big(1 + O\big(\rho^{-2n}\big)\big)\nonumber \\
 & + \begin{cases}
						\al(c)\left(	1 - \di c^{-1} \frac{C_n[\phi]}{C_{n-1}[\phi]} \right) \big(1 + O\big(\rho^{-2n}\big)\big), & |c|<1, \\
						O(\rho^{-n})	, & |c|>1,
					\end{cases}\label{CCC}
			\end{align}
			where $C_n[\phi]$ and
			$\al$ are given by \eqref{Cnphi} and \eqref{al}, respectively. The number $\rho>1$ is such that~$\phi$ is analytic in the annulus $\big\{z\colon \rho^{-1}<|z|<\rho\big\}$, and in addition is chosen such that $\rho < |c|$ when~${|c|>1}$, and $\rho^{-1} > |c|$ when $|c|<1$.
		\end{Corollary}
		\begin{proof}
			Let us rewrite \eqref{bdToep-RHP-eta0} as
			\begin{equation}\label{aaa}
				\frac{D^{B}_{n+1}[z\phi;\eta_0]}{D_{n}[z\phi]} = - \frac{1}{c} \frac{D_{n+1}[z\phi]}{D_{n}[z\phi]} + \frac{1}{c} Z_{12}(c;n) = \frac{Z_{12}(c;n)-\varkappa_n^{-2}[z\phi]}{c},
			\end{equation}
			where we have used \eqref{kappa zphi}. From \eqref{orth zphi} and \eqref{Toeplitz OP zphi}, we can observe that
			\begin{equation*}
				\varkappa_n^{-2}[z\phi] = Z_{12}(0;n),
			\end{equation*}
			and if we combine this with \eqref{aaa}, we obtain \eqref{bbb}. Rewriting \eqref{bbb}
			using \eqref{Z12}, we have
			\begin{align*}
				\frac{Z_{12}(c;n) - Z_{12}(0;n)}{c} ={}& X_{12}(c;n) + \mathscr{B}(n) \left( \frac{X_{12}(c;n) - X_{12}(0;n)}{c} \right) \\ & -\overset{\infty}{X}_{1,12}(n) \left( \frac{X_{22}(c;n) - X_{22}(0;n)}{c} \right).
			\end{align*}
			We now obtain \eqref{CCC}, using \eqref{X in terms of R exact}, \eqref{X1,12 infty}, and \eqref{BBB}.
		\end{proof}
		
Let us recall $q_0$ as defined in \eqref{q0}. The Fourier coefficients of $q_0$ are given by
		\begin{gather*}
			q_{0,j} = \begin{cases}
				0, & |c|<1, \\
				-(c)^{-j-1}, & |c|>1,
			\end{cases}\qquad 0\leq j \leq n.
		\end{gather*}
		The following results can be proven identically as Lemma \ref{lemma 3.2} and Corollary \ref{cor 1} reminded above from \cite{BEGIL}, which establishes how $D^B_{N}[z\phi;q]$ is encoded into $Z$-RHP data.
		\begin{Lemma}\label{lem 2.15}
			The bordered Toeplitz determinant $D^B_{n+1}\big[z\phi,\di \tfrac{1}{z-c}\big]$, is encoded into the $Z$-RHP data described by
			\begin{equation}\label{psi=1/z-c}
				D^B_{n+1}\left[z\phi;\frac{1}{z-c}\right] = \begin{cases}
					0, & |c|<1, \\
					-c^{-n-1}D_{n}[z\phi]Z_{11}(c;n), & |c|>1,
				\end{cases}
			\end{equation}
			where $D_{n}[\phi]$ is given by \eqref{ToeplitzDet} and $Z_{11}$ is the $11$ entry of the solution to {\rm RH-Z1} through {\rm RH-Z3}. Moreover, if $\phi$ is a Szeg\H{o}-type symbol, we have
			\begin{equation}\label{psi=1/z-c 1}
				\frac{D^B_{n+1}\left[z\phi;\di \frac{1}{z-c}\right]}{D_{n}[z\phi]}	 = \begin{cases}
					0, & |c|<1, \\
					\di	- \al(c) \left( c^{-1} - c^{-2} \frac{C_n[\phi]}{C_{n-1}[\phi]} \right) \big(1 + O\big(\rho^{-2n}\big)\big), & |c|>1,
				\end{cases}
			\end{equation}
			where $C_n[\phi]$ and
			$\al$ are given by \eqref{Cnphi} and \eqref{al}, respectively. For the case $|c|>1$, the number~${1<\rho < |c|}$ is such that $\phi$ is analytic in the annulus $\big\{z\colon \rho^{-1}<|z|<\rho\big\}$.
		\end{Lemma}
		\begin{proof}
			The proof of \eqref{psi=1/z-c} is identical to that of \eqref{psi=1/z-c 0}. From \eqref{ZX1}, we have
			\begin{equation}\label{Z11 X11 X21}
				Z_{11}(z;n) = \big( 1 + \mathscr{B}(n) z^{-1} \big) X_{11}(z;n) -\overset{\infty}{X}_{1,12}(n) z^{-1} X_{21}(z;n),
			\end{equation}
			where $\mathscr{B}(n)$ is defined in \eqref{Bn}.\footnote{Since $Z_{11}$ is a polynomial (see \eqref{Toeplitz-OP-solution zphi}), the coefficient of $z^{-1}$ in \eqref{Z11 X11 X21} must vanish: $\mathscr{B}(n) X_{11}(0;n) -\smash{\overset{\infty}{X}}_{1,12}(n) X_{21}(0;n) = 0$. This is, as expected, in agreement with \eqref{Bn}.} Now, \eqref{psi=1/z-c 1} follows from \eqref{X in terms of R exact}, \eqref{R_k's are small}, \eqref{X1,12 infty}, and~\eqref{BBB}.
		\end{proof}
		
		Finally, let us find the asymptotics of $D^B_{n+1}[z\phi;z]$.
		\begin{Lemma}\label{lem 2.16} It holds that
			\begin{equation}\label{zphi z}
				D^B_{n+1}[z\phi;z] = D^B_{n}[z\phi] \lim_{z \to \infty} \left[ \frac{Z_{11}(z;n)-z^n}{z^{n-1}} \right].
			\end{equation}
			Moreover, if $\phi$ is of Szeg\H{o}-type, as $n \to \infty$, we have
			\begin{equation}\label{zphi z asymp}
				\frac{D^B_{n+1}\left[z\phi;z\right]}{D_n[z\phi]}		 = \left( - [\log \phi]_{-1} - \frac{C_n[\phi]}{C_{n-1}[\phi]} \right) \big(1 + O\big(\rho^{-2n}\big)\big),
			\end{equation}
			where $C_n[\phi]$ is given by \eqref{Cnphi} and	the number $\rho>1$ is chosen such that $\phi$ is analytic in the annulus $\big\{z\colon \rho^{-1}<|z|<\rho\big\}$.
		\end{Lemma}
		\begin{proof}
			One can prove \eqref{zphi z} in the exact same manner as \cite[equation (2.16)]{BEGIL}. Let us recall~\eqref{Z11 X11 X21}. Observe that
 \[
 \frac{\overset{\infty}{X}_{1,12}(n) z^{-1} X_{21}(z;n)}{z^{n-1}} = O\big(z^{-1}\big),
 \]
 as $X_{21}(z;n)$ is a polynomial of degree $n-1$, and thus the above term in \eqref{Z11 X11 X21} does not contribute to the limit in \eqref{zphi z}. So we just focus on the first term in \eqref{Z11 X11 X21}. Expanding $\al(z)$, given by~\eqref{al}, as $z \to \infty $, we get
			\begin{gather*}
				\al(z) = 1 - \frac{1}{2\pi \ii z} \int_{\T} \ln (\phi(\tau)) \dd \tau + O\big(z^{-2}\big).
			\end{gather*}
			Using this in the expression for $X_{11}(z;n) = \al(z) z^n \big(1 + O \big(\frac{\rho^{-2n}}{1+|z|}\big)\big)$ in $\Om_{\infty}$ (see \eqref{X in terms of R exact} and Figure~\ref{S_contour}) and combining with \eqref{Z11 X11 X21}, \eqref{zphi z} and \eqref{BBB}, we obtain \eqref{zphi z asymp}.	
		\end{proof}

		The following result now follows in a straightforward way from Lemmas \ref{lem 2.15}, \ref{lem 2.16}, Corollaries~\ref{Cor zphi phi and phi/z} and~\ref{Corollary zphi/z-c}, and equations \eqref{linear-combination}, \eqref{border,phi,phi} and \eqref{border,phi,const}.
		
\begin{Corollary}\label{main cor 1}
Let $\psi$ be given by \eqref{general psi} and \eqref{q1 q2}, and $\phi$ be of Szeg\H{o}-type. Then, the following asymptotic behavior as $n \to \infty$ takes place:
			\begin{gather*}
				\frac{D^{B}_{n+1}[z\phi;\psi]}{D_{n}[z\phi]} = G[\phi] \left( F [\phi,\psi] - H[\phi,\psi] \frac{C_n[\phi]}{C_{n-1}[\phi]} + O(\rho^{-n})\right),
			\end{gather*}
			where $F [\phi,\psi]$ is given by \eqref{ConstantF2}, and $H[\phi,\psi]$ is given by \eqref{ConstantH}.
			In the above formulae, $C_n[\phi]$ and
			$\al$ are given by \eqref{Cnphi} and \eqref{al}, respectively, and the number $\rho$ is such that
\[
1<\rho < \underset{1 \leq j \leq m \atop |c_j|>1}{\min} \{|c_j|\}, \qquad \underset{1 \leq j \leq m \atop 0<|c_j|<1}{\max} \{|c_j|\}<\rho^{-1}<1,
\]
 and $\phi$ is analytic in the annulus $\big\{z\colon \rho^{-1}<|z|<\rho\big\}$.
		\end{Corollary}
		This result is the last needed asymptotics to prove Theorem \ref{main thm 2-bordered}. In fact, from \eqref{Dodgson 2-bordered},
		\begin{gather*}
			D^B_n[\phi;\boldsymbol{\psi}_2] = D^B_{n-1}\big[\phi;z^{-1}\psi_1\big] \frac{D^B_{n-1}[z\phi;\psi_2]}{D_{n-2}[z \phi]} - D^B_{n-1}\big[\phi;z^{-1}\psi_2\big] \frac{D^B_{n-1}[z\phi;\psi_1]}{D_{n-2}[z \phi]},
		\end{gather*}
		Corollary \ref{main cor 1} and Lemma \ref{main lem 1}, we obtain \eqref{135} and \eqref{136}. We have thus finished the proof of Theorem \ref{main thm 2-bordered}.
		
\begin{Remark}\label{remark multi-bordered} \normalfont
It is worthwhile to highlight that the framework presented in this section can be recursively used to find the asymptotics of a $k$-bordered Toeplitz determinant when each~$\psi_j$ is of the form \eqref{general psi}--\eqref{q1 q2}, for any finite $k$. For instance, let us consider the three-bordered Toeplitz determinant
$
\mathcal{D}:= D^B_n[\phi;\boldsymbol{\psi}_3] \equiv D^B_n[\phi;\psi_1,\psi_2,\psi_3]$.
			Like the two-bordered case, we use the Dodgson condensation identity \eqref{DODGSON3}, this time for $\mathcal{D}$:
			\begin{equation}\label{DODGSON33}
				\mathcal{D} \cdot \mathcal{D}\left\lbrace \begin{matrix} 0 & n-1 \\ n-2& n-1 \end{matrix} \right\rbrace = \mathcal{D}\left\lbrace \begin{matrix} 0 \\ n-2 \end{matrix} \right\rbrace \cdot \mathcal{D}\left\lbrace \begin{matrix} n-1 \\ n-1 \end{matrix} \right\rbrace - \mathcal{D}\left\lbrace \begin{matrix} 0 \\ n-1 \end{matrix} \right\rbrace \cdot \mathcal{D}\left\lbrace \begin{matrix} n-1 \\ n-2 \end{matrix} \right\rbrace.
			\end{equation}
			We observe that
			\begin{equation*}
				\mathcal{D}\left\lbrace \begin{matrix} 0 & n-1 \\ n-2& n-1 \end{matrix} \right\rbrace = D^B_{n-2}\big[z\phi;z^{-1}\psi_1\big]
			\end{equation*}
			is a (single) bordered Toeplitz determinant, while all four determinants on the right-hand side of \eqref{DODGSON33} are two-bordered Toeplitz determinants:
			\begin{align}
				&\mathcal{D}\left\lbrace \begin{matrix} 0 \\ n-2 \end{matrix} \right\rbrace = D^B_{n-1}[z\phi;\psi_1,\psi_3], \label{E11} \\
				&\mathcal{D}\left\lbrace \begin{matrix} n-1 \\ n-1 \end{matrix} \right\rbrace = D^B_{n-1}\big[\phi;z^{-1}\psi_1,z^{-1}\psi_2\big], \\
			&	\mathcal{D}\left\lbrace \begin{matrix} 0 \\ n-1 \end{matrix} \right\rbrace = D^B_{n-1}[z\phi;\psi_1,\psi_2], \\
			&	\mathcal{D}\left\lbrace \begin{matrix} n-1 \\ n-2 \end{matrix} \right\rbrace = D^B_{n-1}\big[\phi;z^{-1}\psi_1,z^{-1}\psi_3\big]. \label{E44}
			\end{align}
			These two-bordered Toeplitz determinants can be asymptotically analyzed using the results and methods described earlier in this section, and thus pave the way for the asymptotic analysis of~${D^B_n[\phi;\boldsymbol{\psi}_3]}$ via \eqref{DODGSON33}.
		\end{Remark}
		
		\subsection{A new proof of the three term recurrence relations for BOPUC}
		Finally, we would like to discuss the compatibility of \eqref{ZX1} and \eqref{Z in terms of X1} in view of the uniqueness of the solution of the $Z$ Riemann--Hilbert problem, see Lemma \ref{uniqueness Y}. This compatibility provides a new proof for the recurrence relations of the system of bi-orthogonal polynomials on the unit circle in the next lemma.
		
\begin{Lemma}[{\cite[Lemma 2.2]{DIK}, \cite{SzegoOP}}]\label{lemma 3term} Suppose that for each $n \!\in \N \cup \{0\}$ the solution $X(z;n)$ of the Riemann--Hilbert problem {\rm RH-X1} through {\rm RH-X3} and the solution $Z(z;n)$ of the Rie\-mann--Hilbert problem {\rm RH-Z1} through {\rm RH-Z3} exist. Then the system of bi-orthogonal poly\-no\-mi\-als~${\big\{Q_j(z),\widehat{Q}_j(z)\big\}^{\infty}_{j=0}}$ exist and satisfy the following recurrence relations for $n \in \N \cup \{0\}$:
			\begin{align}
				&	\varkappa_n z Q_n(z) = \varkappa_{n+1} Q_{n+1}(z) - Q_{n+1}(0) z^{n+1} \widehat{Q}_{n+1}\big(z^{-1}\big), \label{2.3 DIK} \\
				&	\varkappa_n z^{-1} \widehat{Q}_{n}\big(z^{-1}\big) = \varkappa_{n+1} \widehat{Q}_{n+1}\big(z^{-1}\big) - \widehat{Q}_{n+1}(0) z^{-n-1}Q_{n+1}(z), \label{2.4 DIK}\\
				&	\varkappa_{n+1} z^{-1} \widehat{Q}_{n}\big(z^{-1}\big) = \varkappa_{n} \widehat{Q}_{n+1}\big(z^{-1}\big) - \widehat{Q}_{n+1}(0) z^{-n}Q_{n}(z), \label{2.5 DIK}
			\end{align}
			and
			\begin{equation}
				\varkappa^2_{n+1} - \varkappa^2_{n} = Q_{n+1}(0) \widehat{Q}_{n+1}(0). \label{2.6 DIK}
			\end{equation}
		\end{Lemma}
		
		\begin{proof}
			As described in the proofs of Theorems \ref{thm Z-RHP intro} and \ref{thmXZ intro} in Sections \ref{sec proof of thm 1.6} and \ref{sec proof of thm 1.8}, under these assumptions for each $n \in \N \cup \{0\}$, we have
			\begin{itemize}\itemsep=0pt
				\item $D_n[\phi] \neq 0$ and $\varkappa_{n}$ exist and is nonzero (see \eqref{varkappa}), and
				\item$X_{11}(0;n) \equiv \varkappa^{-1}_n Q_n(0) \neq 0$ and $ \overset{\infty}{X}_{1,12}(n)\neq 0$.
			\end{itemize}
			The compatibility of the $11$ and $21$ entries of \eqref{ZX1} and \eqref{Z in terms of X1} can be written as
			\begin{gather}
						\begin{bmatrix}
						- z^{-1} \overset{\infty}{X}_{1,12}(n) & \overset{\infty}{X}_{1,12}(n-1) \\
						z^{-1} & 0
					\end{bmatrix} \begin{bmatrix}
						X_{21}(z;n) \\ X_{21}(z;n-1)
					\end{bmatrix}\nonumber \\
\qquad= 	\begin{bmatrix}
						- \di \tfrac{\overset{\infty}{X}_{1,12}(n)X_{21}(0;n)}{X_{11}(0;n)}z^{-1} - 1 & z + \overset{\infty}{X}_{1,22}(n-1) - \di \tfrac{\overset{\infty}{X}_{2,12}(n-1)}{\overset{\infty}{X}_{1,12}(n-1)} \\
						\di	\tfrac{X_{21}(0;n)}{X_{11}(0;n)} z^{-1} & \di
						\tfrac{1}{\overset{\infty}{X}_{1,12}(n-1)}
					\end{bmatrix} \begin{bmatrix}
						X_{11}(z;n) \\ X_{11}(z;n-1)
					\end{bmatrix}.\label{Z compatibility}
			\end{gather}
			Solving this linear system by inverting the coefficient matrix on the left-hand side, in particular, yields
			\begin{equation*}
				X_{21}(z;n) = \frac{X_{21}(0;n)}{X_{11}(0;n)} X_{11}(z;n) + 		\frac{1}{\overset{\infty}{X}_{1,12}(n-1)} z X_{11}(z;n-1).
			\end{equation*}
			Using \eqref{Toeplitz-OP-solution}, shifting the index $n \mapsto n+1$, and straight-forward simplifications yield
			\begin{equation}\label{zQ}
				z Q_n(z) - \frac{\varkappa^3_n \overset{\infty}{X}_{1,12}(n)}{Q_{n+1}(0)}Q_{n+1}(z) = - \overset{\infty}{X}_{1,12}(n) \varkappa^2_n z^n \widehat{Q}_n\big(z^{-1}\big).
			\end{equation}
			Matching the coefficients of $z^{n+1}$ yields the identity
			\begin{equation*}
				\overset{\infty}{X}_{1,12}(n) = \frac{Q_{n+1}(0)}{\varkappa^2_{n} \varkappa_{n+1}},
			\end{equation*}
			using which we can write \eqref{zQ} as
			\begin{gather}\label{zQ1}
				z Q_n(z) = \frac{\varkappa_n}{\varkappa_{n+1}}Q_{n+1}(z) - \frac{Q_{n+1}(0)}{\varkappa_{n+1}} z^n \widehat{Q}_n\big(z^{-1}\big).
			\end{gather}
			Now, by inverting the coefficient matrix on the right-hand side of \eqref{Z compatibility}, in particular, we obtain
			\begin{gather}\label{other rec}
				\mathcal{D} X_{11}(z;n-1) = X_{21}(z;n-1)-X_{21}(z;n),
			\end{gather}
			for some constant $\mathcal{D}$. Recalling \eqref{Toeplitz-OP-solution} and matching the coefficients of $z^{n-1}$ gives \begin{gather*}
				\mathcal{D} = \varkappa_{n-1}\widehat{Q}_{n-1}(0).
			\end{gather*}
			Using this along with \eqref{Toeplitz-OP-solution}, shifting the index $n \mapsto n+2$	and straightforward rearrangement of terms in \eqref{other rec} yield \eqref{2.4 DIK}.
			
			Now, we combine \eqref{2.4 DIK} and \eqref{zQ1} to obtain
			\begin{equation}\label{zQ2}
				\varkappa_n z Q_n(z) = \left[ \frac{\varkappa^2_{n} + Q_{n+1}(0) \widehat{Q}_{n+1}(0) }{\varkappa_{n+1}} \right] Q_{n+1}(z) - Q_{n+1}(0) z^{n+1} \widehat{Q}_{n+1}\big(z^{-1}\big).
			\end{equation}
			Evaluating this equation at $z=0$ gives \eqref{2.6 DIK}. Combining \eqref{2.6 DIK} and \eqref{zQ2} gives \eqref{2.3 DIK}. Finally, eliminating $Q_{n+1}(z)$ from \eqref{2.3 DIK} and \eqref{2.4 DIK}, and using \eqref{2.6 DIK}, yields \eqref{2.5 DIK}.
		\end{proof}

\section[Semi-framed, framed and multi-framed Toeplitz determinants]{Semi-framed, framed \\
and multi-framed Toeplitz determinants}\label{section framed}
		As will become clear later in the sequel, the semi-framed Toeplitz determinants form the building blocks to study the asymptotics of framed and multi-framed Toeplitz determinants. To get started in this section, it is useful to revisit the definitions of the semi-framed Toeplitz determinants which were introduced in the introduction here again. For $\phi,\psi,\eta \in L^1(\T)$ and a~parameter $a \in \C$ define the $n\times n$ semi-framed Toeplitz determinants $\mathscr{E}_n[\phi;\psi,\eta;a]$, $\mathscr{G}_n[\phi;\psi,\eta;a]$, $\mathscr{H}_n[\phi;\psi,\eta;a]$ and $\mathscr{L}_n[\phi;\psi,\eta;a]$ as
		\begin{gather}
			\mathscr{E}_n[\phi;\psi,\eta;a] := \det \begin{bmatrix}
				\phi_0& \phi_{-1} & \cdots & \phi_{-n+2} & \psi_{n-2} \\
				\phi_{1}& \phi_0 & \cdots & \phi_{-n+3} & \psi_{n-3} \\
				\vdots & \vdots & \ddots & \vdots & \vdots \\
				\phi_{n-2} & \phi_{n-3} & \cdots & \phi_{0} & \psi_{0} \\
				\eta_{n-2} & \eta_{n-3} & \cdots & \eta_{0} & a
			\end{bmatrix},\label{half-framed}
\\
			\mathscr{G}_n[\phi;\psi,\eta;a] := \det \begin{bmatrix}
				\phi_0& \phi_{-1} & \cdots & \phi_{-n+2} & \psi_{0} \\
				\phi_{1}& \phi_0 & \cdots & \phi_{-n+3} & \psi_{1} \\
				\vdots & \vdots & \ddots & \vdots & \vdots \\
				\phi_{n-2} & \phi_{n-3} & \cdots & \phi_{0} & \psi_{n-2} \\
				\eta_{0} & \eta_{1} & \cdots & \eta_{n-2} & a
			\end{bmatrix},
\\
			\mathscr{H}_n[\phi;\psi,\eta;a] := \det \begin{bmatrix}
				\phi_0& \phi_{-1} & \cdots & \phi_{-n+2} & \psi_{0} \\
				\phi_{1}& \phi_0 & \cdots & \phi_{-n+3} & \psi_{1} \\
				\vdots & \vdots & \ddots & \vdots & \vdots \\
				\phi_{n-2} & \phi_{n-3} & \cdots & \phi_{0} & \psi_{n-2} \\
				\eta_{n-2} & \eta_{n-3} & \cdots & \eta_{0} & a
			\end{bmatrix},\label{half-framed 2}
		\end{gather}
		and
		\begin{equation}\label{half-framed 3}
			\mathscr{L}_n[\phi;\psi,\eta;a] := \det \begin{bmatrix}
				\phi_0& \phi_{-1} & \cdots & \phi_{-n+2} & \psi_{n-2} \\
				\phi_{1}& \phi_0 & \cdots & \phi_{-n+3} & \psi_{n-3} \\
				\vdots & \vdots & \ddots & \vdots & \vdots \\
				\phi_{n-2} & \phi_{n-3} & \cdots & \phi_{0} & \psi_{0} \\
				\eta_{0} & \eta_{1} & \cdots & \eta_{n-2} & a
			\end{bmatrix},
		\end{equation}
		where $f_j$'s are the Fourier coefficients of $f \in \{ \phi, \psi, \eta \}$. To distinguish these framed Toeplitz matrices, it is helpful to think of them visually as $ \leftarrow \uparrow$, $\rightarrow \downarrow$, $\leftarrow \downarrow $, and $\rightarrow \uparrow$, respectively. For example, $\leftarrow \uparrow$ is associated with $\mathscr{E}_n$ because the index of the Fourier coefficients in the last row of $\boldsymbol{\mathscr{E}}_n$ increase from right to left ($\leftarrow$) and the index of the Fourier coefficients in the last column of $\boldsymbol{\mathscr{E}}_N$ increase from bottom to top ($\uparrow$). We should mention that each of these determinants can be written in terms of any other one by a simple observation (see Lemma \ref{EGL in terms of H}).
		
		It can be easily checked that
		\begin{gather}
			F_N\Bigg[\phi; \sum_{j=1}^{m_1} A_j \psi_j,\sum_{k=1}^{m_2} B_k \eta_k;a\Bigg] = \sum_{j=1}^{m_1} \sum_{k=1}^{m_2} A_j B_k F_N\big[\phi; \psi_j, \eta_k; \widehat{a}_{j,k} \big],\nonumber \\
 F \in \{\mathscr{E},\mathscr{G}, \mathscr{H}, \mathscr{L}\},\label{framed lin comb}
		\end{gather}
		where
		$\widehat{a}_{j,k}$ are complex numbers satisfying
		\begin{equation}\label{framed lin comb1}
			\sum_{j=1}^{m_1} \sum_{k=1}^{m_2} A_j B_k \widehat{a}_{j,k}= a.
		\end{equation}
		If $A_j$ and $B_k$ are nonzero, one such set of numbers is obviously
$\widehat{a}_{j,k} = \frac{a}{m_1m_2A_jB_k}$.

		\begin{Lemma}\label{EGL in terms of H}
			The semi-framed determinants $\mathscr{E}_n[\phi;\psi,\eta;a]$, $\mathscr{G}_n[\phi;\psi,\eta;a]$, and $\mathscr{L}_n[\phi;\psi,\eta;a]$ have the following representations in terms of $\mathscr{H}_n[\phi;f,g;a]$:
			\begin{align*}
				\mathscr{E}_n[\phi;\psi,\eta;a] & = \mathscr{H}_n\big[\phi;z^{n-2}\Tilde{\psi},\eta;a\big], \qquad
				\mathscr{G}_n[\phi;\psi,\eta;a] = \mathscr{H}_n\big[\phi;\psi,z^{n-2}\Tilde{\eta};a\big], \\
				\mathscr{L}_n[\phi;\psi,\eta;a] & = \mathscr{H}_n\big[\phi;z^{n-2}\Tilde{\psi},z^{n-2}\Tilde{\eta};a\big],
			\end{align*}
			where $\Tilde{f}$ denotes the function $z \mapsto f\big(z^{-1}\big)$, $f \in \{\psi,\eta\}$.
		\end{Lemma}
		\begin{proof}
			It is enough to observe that $\big( z^{n-2}\Tilde{f} \big)_j = f_{n-2-j}$.
		\end{proof}
		
		Notice that in general the semi-framed Toeplitz determinants can not be reduced to simpler objects like pure-Toeplitz determinants or bordered Toeplitz determinants via Dodgson condensation identities. Let us discuss here why no such identity exists. Let $\mathscr{M}$ be an $N \times N$ semi-framed Toeplitz determinant. If one hopes for a Dodgson condensation identity
		\begin{gather*}
			\mathscr{M} \cdot \mathscr{M}\left\lbrace \begin{matrix} j_1 & j_2 \\ k_1& k_2 \end{matrix} \right\rbrace = \mathscr{M}\left\lbrace \begin{matrix} j_1 \\ k_1 \end{matrix} \right\rbrace \cdot \mathscr{M}\left\lbrace \begin{matrix} j_2 \\ k_2 \end{matrix} \right\rbrace - \mathscr{M}\left\lbrace \begin{matrix} j_1 \\ k_2 \end{matrix} \right\rbrace \cdot \mathscr{M}\left\lbrace \begin{matrix} j_2 \\ k_1 \end{matrix} \right\rbrace,
		\end{gather*}
		with a simpler right-hand side (free of semi-framed determinants), then they must choose $j_2=k_2=N-1$.\footnote{Recall, say from \eqref{ToeplitzDet}, that we index the rows and columns of an $N\times N$ matrix by $0 \leq j \leq N-1$ and $0 \leq k \leq N-1$, respectively.} Then, it is easy to see that any other choice for $j_1$ and $k_1$ can not lead to a~situation where the right-hand side of the corresponding Dodgson condensation identity is free of semi-framed determinants. For example, with the most natural choice $j_1=k_1=N-2$, we have\looseness=-1
		\begin{align*}
 \underbrace{\mathscr{M}}_{\text{semi-framed}}
\cdot \underbrace{ \mathscr{M}\left\lbrace \begin{matrix} N-2 & N-1 \\ N-2 & N-1 \end{matrix} \right\rbrace }_{\text{pure Toeplitz}} ={}& \underbrace{ \mathscr{M}\left\lbrace \begin{matrix} N-2 \\ N-2 \end{matrix} \right\rbrace }_{\text{semi-framed}} \cdot \underbrace{ \mathscr{M}\left\lbrace \begin{matrix} N-1 \\ N-1 \end{matrix} \right\rbrace}_{\text{pure Toeplitz}} \\ &- \underbrace{ \mathscr{M}\left\lbrace \begin{matrix} N-2 \\ N-1 \end{matrix} \right\rbrace}_{\text{bordered Toeplitz}} \cdot \underbrace{ \mathscr{M}\left\lbrace \begin{matrix} N-1 \\ N-2 \end{matrix} \right\rbrace}_{\text{bordered Toeplitz}}.
		\end{align*}
		
		This suggests that the semi-framed Toeplitz determinants (corresponding to generic symbols) are structured determinants which must be studied independently without the hope for their reduction to the pure Toeplitz or bordered Toeplitz determinants. In fact, to that end, it turns out that the characterizing objects for the semi-framed Toeplitz determinants $\mathscr{E}_n[\phi;\psi,\eta;a]$, $\mathscr{G}_n[\phi;\psi,\eta;a]$, $\mathscr{H}_n[\phi;\psi,\eta;a]$, and $\mathscr{L}_n[\phi;\psi,\eta;a]$ are the reproducing kernel of the system of orthogonal polynomials associated with the symbol $\phi$, while the characterizing objects for the bordered Toeplitz determinants $D^B_n[\phi;\psi]$ are the orthogonal polynomials themselves (see Section~\ref{sec bordered} and~\cite{BEGIL}).
		
		Even though in general, as we saw above, there does not exist a Dodgson condensation identity which can relate a single semi-framed Toeplitz determinant to a number of pure-Toeplitz and bordered-Toeplitz ones, there are particular examples where such reductions are possible. In some sense such cases are the analogues of the identity \eqref{border,phi,const}, where for a particularly simple symbol (in \eqref{border,phi,const}: $\psi \equiv 1$) a more complex structured determinant (in \eqref{border,phi,const}: the bordered Toeplitz determinant) can be reduced to a less complex structured determinant (in \eqref{border,phi,const}: the pure-Toeplitz determinant). In the case of semi-framed Toeplitz determinants $\mathscr{H}_{n+1}[\phi;\psi,\eta;a]$, we get such reductions when either $\psi \equiv c$, $\eta \equiv c$, $\psi \equiv cz^{n-1}$, $\eta \equiv cz^{n-1}$, $c \in \C$, as described in the following lemma.
		
		\begin{Lemma}
			It holds that
			\begin{align*}
				& \mathscr{H}_{n+1}[\phi;1,\eta;a] = a D_{n}[\phi] + (-1)^{n} D^B_{n}\big[z^{-1}\phi;\eta\big], \\
				& \mathscr{H}_{n+1}[\phi;\psi,1;a] = a D_{n}[\phi] - D^B_{n}\big[\tilde{\phi};z^{n-1} \tilde{\psi}\big], \qquad \mathscr{H}_{n+1}\big[\phi;z^{n-1},\eta;a\big] = a D_{n}[\phi] - D^B_{n}[\phi;\eta], \\
				& \mathscr{H}_{n+1}\big[\phi;\psi,z^{n-1};a\big] = a D_{n}[\phi] + (-1)^n D^B_{n}\big[z^{-1}\tilde{\phi};z^{n-1} \tilde{\psi}\big],
			\end{align*}
			where $\Tilde{f}$ denotes the function $z \mapsto f\big(z^{-1}\big)$, $f \in \{\psi,\phi\}$.
		\end{Lemma}
		\begin{proof}
			These are immediate consequences of the definitions \eqref{btd0} and \eqref{half-framed 2} and observing that $\big( z^{n-1}\Tilde{f} \big)_j = f_{n-1-j}$.
		\end{proof}
		
		\begin{Remark}
		In view of Lemma \ref{EGL in terms of H}, it is indeed sufficient to prove the above lemma for the $\mathscr{H}$-semi-framed Toeplitz determinants.
		\end{Remark}
		
		\subsection[The Riemann--Hilbert characterization for semi-framed Toeplitz determinants:\\ Proof of Theorem 1.11]{The Riemann--Hilbert characterization for semi-framed \\ Toeplitz determinants: Proof of Theorem \ref{semis in terms of RepKer intro}}

		Similar to what is shown about bordered Toeplitz determinant $D^B_n[\phi,\psi]$ in \cite[Section 2]{BEGIL}, in this section we show that the semi-framed Toeplitz determinants can also be expressed in terms of the solution of the Riemann--Hilbert problem for \textit{pure} Toeplitz determinants.
		
		Let $\psi=q_1 \phi + q_2$ and $\eta=q_3 \phi + q_4$ where $\phi$ is the generating function of the Toeplitz part and $q_j$'s are rational functions with simple poles, $j=1,2,3,4$. Below we show that unlike the bordered Toeplitz determinants which are related to the orthogonal polynomials and/or their Cauchy-type transforms (see Section~\ref{sec bordered}), the semi-framed Toeplitz determinants are related to the reproducing kernel of the same system of orthogonal polynomials. In order to see this connection, we need to first find a determinantal representation for the reproducing kernel which would play the same role for semi-framed Toeplitz determinants, as the determinantal representation \eqref{Toeplitz OP 1} plays for the bordered Toeplitz determinants.
		
		We follow \cite{GW} to find this determinantal representation for the reproducing kernel. To that end, we need to recall the LU decomposition of the Toeplitz matrix $T_n[\phi]$. Write the polynomials~$Q_n(z)$ and $\widehat{Q}_n(z)$ as
		\begin{equation*}
			Q_n(z) = \sum_{j=0}^{n} q_{n,j} z^j, \qquad \widehat{Q}_n(z) = \sum_{j=0}^{n} \hat{q}_{n,j} z^j,
		\end{equation*}
		and let us also denote
		\begin{equation*}
			\boldsymbol{Z}_n(z) := \begin{bmatrix}
				1 \\ z \\ \vdots \\ z^{n}
			\end{bmatrix} \qandq \boldsymbol{F}_n(z) := \begin{bmatrix}
				F_0(z) \\ F_1(z) \\ \vdots \\ F_n(z)
			\end{bmatrix}, \qquad \boldsymbol{F} \in \big\{\boldsymbol{Q},\boldsymbol{\widehat{Q}}\big\}.
		\end{equation*}
		We thus have $\boldsymbol{Q}_n(z) = \boldsymbol{A}_{n} \boldsymbol{Z}_n(z)$, $\boldsymbol{\widehat{Q}}_n(z) = \boldsymbol{B}_{n} \boldsymbol{Z}_n(z)$,
		where $\boldsymbol{A}_{n}$ and $ \boldsymbol{B}_{n}$ are the following $(n+1)\times(n+1)$ lower triangular matrices
		\begin{equation}\label{An Bn}
			\boldsymbol{A}_{n} := \begin{bmatrix}
				q_{0,0} & 0 & \cdots & 0 \\
				q_{1,0} & q_{1,1} & \cdots & 0 \\
				\vdots & \vdots & \ddots & \vdots \\
				q_{n,0} & q_{n,1} & \cdots & q_{n,n}
			\end{bmatrix}, \qquad 	\boldsymbol{B}_{n} := \begin{bmatrix}
				\hat{q}_{0,0} & 0 & \cdots & 0 \\
				\hat{q}_{1,0} & \hat{q}_{1,1} & \cdots & 0 \\
				\vdots & \vdots & \ddots & \vdots \\
				\hat{q}_{n,0} & \hat{q}_{n,1} & \cdots & \hat{q}_{n,n}
			\end{bmatrix}.
		\end{equation}
		For the rest of this section, we assume that for fixed $n$, $D_j[\phi] \neq 0$, for $j \in\{0,1,\dots,n+1\}$,
		so that the bi-orthogonality conditions \eqref{biorthogonality} hold at least for the indices $k,m \in \{0,1,\dots,n+1\}$.
		\begin{Theorem}[\cite{GW}]
			The LU decomposition of $T_{n+1}[\phi]$ is given by
			\begin{equation}\label{LDU D}
				T_{n+1}[\phi] = [\boldsymbol{B}_{n}]^{-1} 	 \big[\boldsymbol{A}_{n}^{\mathsf T}\big]^{-1}.
			\end{equation}
		\end{Theorem}
		\begin{proof}
			We have
			\begin{align*}
					\delta_{\nu \mu} & = \int_{\T} Q_{\nu}(\ze) \widehat{Q}_{\mu}\big(\ze^{-1}\big) \phi(\ze) \frac{\dd \ze}{2\pi \ii \ze} = \sum_{m=0}^{\nu} \sum_{\ell=0}^{\mu} q_{\nu,m}\hat{q}_{\mu,\ell} \int_{\T} \ze^{m-\ell} \phi(\ze) \frac{\dd \ze}{2\pi \ii \ze} \\
&
					= \sum_{m=0}^{\nu} \sum_{\ell=0}^{\mu} q_{\nu,m}\hat{q}_{\mu,\ell} \phi_{\ell-m} = \sum_{m=0}^{\nu} \sum_{\ell=0}^{\mu} ( \boldsymbol{A}_{n} )_{\nu,m} ( \boldsymbol{B}_{n} )_{\mu,\ell} ( T_{n+1}[\phi] )_{\ell,m} \\
& = \sum_{m=0}^{\nu} \sum_{\ell=0}^{\mu} ( \boldsymbol{B}_{n} )_{\mu,\ell} ( T_{n+1}[\phi] )_{\ell,m} \big( \boldsymbol{A}_{n}^{\mathsf T} \big)_{m,\nu} = \big( \boldsymbol{B}_{n} T_{n+1}[\phi] \boldsymbol{A}_{n}^{\mathsf T} \big)_{\mu,\nu},
			\end{align*}
			which is equivalent to \eqref{LDU D}.
		\end{proof}

Let us now consider the reproducing kernel
		\begin{gather*}
			K_{n}(z,\mathcal{z}) := \sum_{j=0}^{n} Q_{j}(\mathcal{z})\widehat{Q}_{j}(z),
		\end{gather*}
		and for a complex parameter $a$ define
		\begin{equation}\label{ReproducingKer}
			\widehat{K}_{n}(z,\mathcal{z};a) := \frac{1}{D_{n+1}[\phi]}\det
			\begin{bmatrix}
				\phi_{0} & \phi_{-1} & \cdots & \phi_{-n} & 1\\
				\phi_{1} & \phi_{0} & \cdots & \phi_{1-n} & z\\
				\vdots & \vdots & \ddots & \vdots & \vdots\\
				\phi_{n} & \phi_{n-1} & \cdots & \phi_{0} & z^{n} \\
				1 & \mathcal{z} & \cdots & \mathcal{z}^{n} & a
			\end{bmatrix}.
		\end{equation}
		\begin{Theorem}[\cite{GW}]\label{Rep Ker Semi Framed}
			The reproducing kernel $K_{n}(z,\mathcal{z})$ has the following semi-framed Toeplitz determinant representation 		\begin{equation}\label{repKer to semi-framed}
				K_{n}(z,\mathcal{z}) = a - \widehat{K}_{n}(z,\mathcal{z};a).
			\end{equation}
			
		\end{Theorem}
		\begin{proof}
			Let
			\begin{equation*}
				\boldsymbol{\widehat{K}}_{n}(z,\mathcal{z};a) := \begin{bmatrix}
					\phi_{0} & \phi_{-1} & \cdots & \phi_{-n} & 1\\
					\phi_{1} & \phi_{0} & \cdots & \phi_{1-n} & z\\
					\vdots & \vdots & \ddots & \vdots & \vdots\\
					\phi_{n} & \phi_{n-1} & \cdots & \phi_{0} & z^{n} \\
					1 & \mathcal{z} & \cdots & \mathcal{z}^{n} & a
				\end{bmatrix},
			\end{equation*}
			and consider the following $(n+2)\times(n+2)$ extensions of $\boldsymbol{A}_{n}$ and $ \boldsymbol{B}_{n}$ introduced in \eqref{An Bn}:
			\begin{equation*}
				\boldsymbol{\widehat{A}}_{n}:=\begin{bmatrix}
					\boldsymbol{A}_{n} & \boldsymbol{0}_{n+1} \\
					\boldsymbol{0}^{\mathsf T}_{n+1} & 1
				\end{bmatrix}, \qquad 	\boldsymbol{\widehat{B}}_{n}:=\begin{bmatrix}
					\boldsymbol{B}_{n} & \boldsymbol{0}_{n+1} \\
					\boldsymbol{0}^{\mathsf T}_{n+1} & 1
				\end{bmatrix},
			\end{equation*}
			where $\boldsymbol{0}^{\mathsf T}_{n}$ is the $1\times n$ vector of zeros. We now have
			\begin{align}
					\boldsymbol{\widehat{B}}_{n}
					\boldsymbol{\widehat{K}_{n}}(z,\ze;a)
					\boldsymbol{\widehat{A}}_{n}^{\mathsf T} & = \begin{bmatrix}
						\boldsymbol{B}_{n} & \boldsymbol{0}_{n+1} \\
						\boldsymbol{0}^{\mathsf T}_{n+1} & 1
					\end{bmatrix} \begin{bmatrix}
						T_{n+1}[\phi] & \boldsymbol{Z}_{n}(z) \\
						\boldsymbol{Z}^{\mathsf T}_{n}(\mathcal{z}) & a
					\end{bmatrix} \begin{bmatrix}
						\boldsymbol{A}^{\mathsf T}_{n} & \boldsymbol{0}_{n+1} \\
						\boldsymbol{0}^{\mathsf T}_{n+1} & 1
					\end{bmatrix} \nonumber\\
& = \begin{bmatrix}
						\boldsymbol{B}_{n} T_{n+1}[\phi] \boldsymbol{A}^{\mathsf T}_{n} & \boldsymbol{B}_{n} \boldsymbol{Z}_{n}(z) \\
						\boldsymbol{Z}^{\mathsf T}_{n}(\mathcal{z}) \boldsymbol{A}^{\mathsf T}_{n} & a
					\end{bmatrix} = \begin{bmatrix}
						\boldsymbol{I}_{n}	 & \boldsymbol{\widehat{Q}}_n(z) \\
						\boldsymbol{Q}^{\mathsf T}_n(\mathcal{z}) & a
					\end{bmatrix}.\label{312}
			\end{align}
			Taking the determinant of both sides of \eqref{312} yields $\widehat{K}_{n}(z,\mathcal{z};a) = a - K_{n}(z,\mathcal{z})$, where we
			have used \begin{gather*}
				\det \boldsymbol{\widehat{A}}_{n} = \det 	\boldsymbol{\widehat{B}}_{n} = \prod_{j=0}^{n} \varkappa_j = \frac{1}{\sqrt{D_{n+1}[\phi]}}.\tag*{\qed}
			\end{gather*}		\renewcommand{\qed}{}
		\end{proof}

\subsubsection{Proof of Theorem \ref{semis in terms of RepKer intro} }
		This theorem bridges the semi-framed Toeplitz determinants \eqref{half-framed}--\eqref{half-framed 3} to the reproducing kernel $K_n(z,\mathcal{z})$.	We only prove \eqref{H} as the remaining identities can be proven identically. Recalling \eqref{ReproducingKer} notice that
		\[
				\int_{\T} \widehat{K}_{n}\big(z^{-1}_1,z_2;\hat{a}\big) z^{-n}_2 \eta(z_2) \frac{\dd z_2}{2 \pi \ii z_2} = \frac{1}{D_{n+1}[\phi]}\det
				\begin{bmatrix}
					\phi_{0} & \phi_{-1} & \cdots & \phi_{-n} & 1\\
					\phi_{1} & \phi_{0} & \cdots & \phi_{1-n} & z^{-1}_1\\
					\vdots & \vdots & \ddots & \vdots & \vdots\\
					\phi_{n} & \phi_{n-1} & \cdots & \phi_{0} & z_1^{-n} \\
					\eta_n & \eta_{n-1} & \cdots & \eta_0 & \hat{a} \eta_n
				\end{bmatrix}.
		\]
		Therefore,
		\begin{gather*}
				\int_{\T} \left(\int_{\T} \widehat{K}_{n}\big(z^{-1}_1,z_2;\hat{a}\big) z^{-n}_2 \eta(z_2) \frac{\dd z_2}{2 \pi \ii z_2} \right) \psi(z_1) \frac{\dd z_1}{2 \pi \ii z_1} \\
\qquad = \frac{1}{D_{n+1}[\phi]}\det
				\begin{bmatrix}
					\phi_{0} & \phi_{-1} & \cdots & \phi_{-n} & \psi_{0}\\
					\phi_{1} & \phi_{0} & \cdots & \phi_{1-n} & \psi_{1}\\
					\vdots & \vdots & \ddots & \vdots & \vdots\\
					\phi_{n} & \phi_{n-1} & \cdots & \phi_{0} & \psi_{n} \\
					\eta_n & \eta_{n-1} & \cdots & \eta_0 & \hat{a} \eta_n \psi_{0}
				\end{bmatrix},
		\end{gather*}
		or
		\begin{equation*} \frac{\mathscr{H}_{n+2}\left[\phi; \psi, \eta ;\hat{a} \eta_n \psi_{0}\right]}{D_{n+1}[\phi]} = \int_{\T} \left(\int_{\T} \widehat{K}_{n}\big(z^{-1}_1,z_2;\hat{a}\big) z^{-n}_2 \eta(z_2) \frac{\dd z_2}{2 \pi \ii z_2} \right) \psi(z_1) \frac{\dd z_1}{2 \pi \ii z_1}.
		\end{equation*}
		Now, employing \eqref{repKer to semi-framed} we find
		\begin{equation*}
			\begin{split}
				\frac{\mathscr{H}_{n+2}\left[\phi; \psi, \eta ;\hat{a} \eta_n \psi_{0}\right]}{D_{n+1}[\phi]} & = \int_{\T} \left(\int_{\T} \left(\hat{a}-K\big(z^{-1}_1,z_2\big) \right) z^{-n}_2 \eta(z_2) \frac{\dd z_2}{2 \pi \ii z_2} \right) \psi(z_1) \frac{\dd z_1}{2 \pi \ii z_1} \\
 & = \hat{a} \eta_n \psi_{0} - \int_{\T} \left(\int_{\T} K\big(z^{-1}_1,z_2\big) z^{-n}_2 \eta(z_2) \frac{\dd z_2}{2 \pi \ii z_2} \right) \psi(z_1) \frac{\dd z_1}{2 \pi \ii z_1}.
			\end{split}
		\end{equation*}
		This is the desired result \eqref{H} if we denote $a \equiv \hat{a} \eta_n \psi_{0}$.
		\subsubsection{Proof of Corollary \ref{EGL and RHP intro}}			
		Moving on, let us recall the Christoffel--Darboux identity for the bi-orthogonal polynomials on the unit circle \cite{DIK,SzegoOP}. For any $z \neq 0$ and $n \in \N \cup \{0\}$, we have
		\begin{align}
				K_{n}(z^{-1},z) ={}& \sum_{j=0}^{n} Q_{j}(z)\widehat{Q}_{j}\big(z^{-1}\big) = -(n+1) Q_{n+1}(z)\widehat{Q}_{n+1}\big(z^{-1}\big) \nonumber\\
 & + z \left( \widehat{Q}_{n+1}\big(z^{-1}\big) \frac{\dd}{\dd z} Q_{n+1}(z) - Q_{n+1}(z) \frac{\dd}{\dd z} \widehat{Q}_{n+1}\big(z^{-1}\big) \right),\label{425}
		\end{align}		
and for any $z_2$, $z_1 \neq 0$ and $n \in \N \cup \{0\}$, we have
		\begin{equation*}
			\big(1-z^{-1}_1z_2\big) \sum_{j=0}^{n} Q_{j}(z_2)\widehat{Q}_{j}\big(z^{-1}_1\big) = z_1^{-n-1}Q_{n+1}(z_1)z_2^{n+1}\widehat{Q}_{n+1}\big(z_2^{-1}\big) - \widehat{Q}_{n+1}\big(z^{-1}_1\big) Q_{n+1}(z_2).
		\end{equation*}
		Therefore, if $z_1 \neq z_2$
		\begin{equation}\label{316}
			K_{n}\big(z^{-1}_1,z_2\big) = \frac{z_1^{-n-1}Q_{n+1}(z_1)z_2^{n+1}\widehat{Q}_{n+1}\big(z_2^{-1}\big) - \widehat{Q}_{n+1}(z_1^{-1}) Q_{n+1}(z_2)}{1-z^{-1}_1z_2}.
		\end{equation}
		Now observe that the equations \eqref{425} and \eqref{316} are going to be particularly useful when we attempt to use \eqref{H} for the $\leftarrow \downarrow $- or $\mathscr{H}$- framed determinants. In fact, we readily have \eqref{H and X-RHP}. To see this, let us recall the integral on the right-hand side of \eqref{H}
		\begin{equation*} \frac{1}{4\pi^2} \int^{2\pi}_{0} \int^{2\pi}_{0} K_n\big({\rm e}^{-\ii \theta_1},{\rm e}^{\ii \theta_2}\big) \eta\big({\rm e}^{\ii \theta_2}\big) \psi\big({\rm e}^{\ii \theta_1}\big) {\rm e}^{-\ii n \theta_2} \dd \theta_1 \dd \theta_2.
		\end{equation*}
		Notice that the integration over the diagonal set $\{(\theta,\theta), 0 \leq \theta \leq 2 \pi\}$ of measure zero makes no contribution to this integral and hence we do not need to employ the identity \eqref{425}. Recalling~\eqref{Toeplitz-OP-solution} the equation \eqref{316} can be written in terms of the $X$-RHP data as follows:
		\begin{gather*}
			K_{n}\big(z^{-1}_1,z_2\big) = \frac{z^{-n}_1}{z_1-z_2} \det \begin{bmatrix}
				X_{11}(z_2;n+1) & X_{21}(z_2;n+2) \\
				X_{11}(z_1;n+1) & X_{21}(z_1;n+2)
			\end{bmatrix}.
		\end{gather*}
		Plugging this into \eqref{H} gives the desired result \eqref{H and X-RHP}. 	The remaining $X$-RHP characterizations for $\mathscr{E}$, $\mathscr{G}$, and $\mathscr{L}$ can be immediately obtained using Lemma \ref{EGL in terms of H}.
		
		\begin{Remark} \normalfont
			Notice that the integrand in \eqref{H and X-RHP} is well defined on the unit circle because the first column of the solution of the $X$-RHP is entire (see {\rm RH-X1} and {\rm RH-X2}).
		\end{Remark}
		
\subsection[Semi-framed Toeplitz determinants involving rational frame symbols:\\ Proofs of Theorems \ref{thm semi-framed rationals intro} and \ref{thm semi-framed rationals . phi intro}]{Semi-framed Toeplitz determinants involving\\ rational frame symbols: Proofs of Theorems \ref{thm semi-framed rationals intro} and \ref{thm semi-framed rationals . phi intro}}\label{sec semi-framed}
		
In the following two subsections, we examine frame symbols which are either rational or are products of the bulk symbol with a rational function. It is important to note that we make these choices because they represent simple and yet nontrivial examples for illustrating the asymptotic analysis. Indeed, the Riemann--Hilbert characterizations provided in Corollary~\ref{EGL and RHP intro} can be further explored with more complex symbol choices, such as Fisher--Hartwig $\phi$ and other classes of frame symbols.
		
\subsubsection{Proof of Theorem \ref{thm semi-framed rationals intro}}
		
In the case of rational border symbols, we have a simpler representation of semi-framed Toeplitz determinants. Recall the function $q_0$ introduced in \eqref{q0}.
		The Fourier coefficients of $q_0$ are
		\begin{equation}\label{q0 Fourier Coefficients}
			q_{0,j} = \begin{cases}
				0, & |c|<1, \\
				-(c)^{-j-1}, & |c|>1,
			\end{cases} \qquad 0 \leq j \leq n.
		\end{equation}
		This immediately leads to the following elementary property of semi-framed Toeplitz determinants:
		
		\begin{Lemma}
			Let $c_j$ be complex numbers with $|c_j|<1$, $j=1,\dots,m$. It holds that
			\begin{equation*}
				\mathscr{H}_{n+2}\Bigg[\phi;\psi,\sum_{j=1}^{m}\frac{b_j}{z-c_j};a\Bigg] = a D_{n+1}[\phi],
			\end{equation*}
			and
			\begin{equation*}
				\mathscr{H}_{n+2}\Bigg[\phi;\sum_{j=1}^{m}\frac{b_j}{z-c_j},\eta;a\Bigg] = a D_{n+1}[\phi].
			\end{equation*}
		\end{Lemma}
		\begin{Lemma}\label{thm semi-framed rationals}
			Let $c$ and $d$ be complex numbers that do not lie on the unit circle. The semi-framed determinant $\mathscr{H}_{n+2}\big[\phi;\di \tfrac{1}{z-d},\di \tfrac{1}{z-c};a\big]$ is encoded into the $X$-RHP data described by
			\begin{gather}
						\frac{\mathscr{H}_{n+2}\big[\phi; \tfrac{1}{z-d},\tfrac{1}{z-c};a\big]}{D_{n+1}[\phi]}
= a - \frac{1}{ (dc)^{n+1}}\nonumber \\
\phantom{\qquad = }{}\times \!\!\begin{cases}
						0, & \text{either $|c|<1$ or $|d|<1$,} \\
						 \tfrac{ X_{11}(c;n+1)X_{21}(d;n+2) - X_{21}(c;n+2)X_{11}(d;n+1) }{d-c}, & \text{$|c|>1$ and $|d|>1$ and $c \neq d$,} \\
						X_{11}(d;n+1)X'_{21}(d;n+2)&\\
 \qquad- X_{21}(d;n+2)X'_{11}(d;n+1), & \text{$|c|>1$ and $|d|>1$ and $c = d$,}
					\end{cases}\label{Hn+2 rationals X-RHP}\\
						\frac{\mathscr{L}_{n+2}\big[\phi; \tfrac{1}{z-d},\tfrac{1}{z-c};a\big]}{D_{n+1}[\phi]} \nonumber\\
\qquad= a -\! \begin{cases}
						0, & \text{either $|c|<1$ or $|d|<1$,} \\
						 \tfrac{ X_{11}\big(\tfrac{1}{c};n+1\big)X_{21}\big(\tfrac{1}{d};n+2\big) - X_{21}\big(\tfrac{1}{c};n+2\big)X_{11}\big(\tfrac{1}{d};n+1\big) }{d-c}, & \text{$|c|>1$ and $|d|>1$ and $c \neq d$}, \\
						\tfrac{1}{d^2} \big( X_{11}\big(\tfrac{1}{d};n+1\big)X'_{21}\big(\tfrac{1}{d};n+2\big) &\\
\qquad- X_{21}\big(\tfrac{1}{d};n+2\big)X'_{11}\big(\tfrac{1}{d};n+1\big) \big), & \text{$|c|>1$ and $|d|>1$ and $c = d$,}\nonumber
					\end{cases}
\\
					\frac{\mathscr{E}_{n+2}\big[\phi; \tfrac{1}{z-d},\tfrac{1}{z-c};a\big]}{D_{n+1}[\phi]}\nonumber\\
 \qquad= a - \frac{1}{c^{n+1}} \begin{cases}
						0, & \text{either $|c|<1$ or $|d|<1$,} \\
						 \tfrac{ X_{11}(c;n+1)X_{21}(d^{-1};n+2) - X_{21}(c;n+2)X_{11}(d^{-1};n+1) }{1-dc}, & \text{$|c|>1$ and $|d|>1$,}
					\end{cases}\nonumber
\\
					 	\frac{\mathscr{G}_{n+2}\big[\phi; \tfrac{1}{z-d},\tfrac{1}{z-c};a\big]}{D_{n+1}[\phi]}\nonumber \\
\qquad= a - \frac{1}{d^{n+1}} \begin{cases}
						0, & \text{either $|c|<1$ or $|d|<1$}, \\
						 \tfrac{ X_{11}(c^{-1};n+1)X_{21}(d;n+2) - X_{21}(c^{-1};n+2)X_{11}(d;n+1) }{dc-1}, & \text{$|c|>1$ and $|d|>1$}.
					\end{cases}\nonumber
			\end{gather}
	Moreover, if $\phi$ is of Szeg\H{o}-type, we have
			\begin{gather}
				\mathscr{H}_{n+1}\left[\phi; \frac{1}{z-d},\frac{1}{z-c};a\right] = G^n[\phi] E[\phi] (a + O(\rho^{-n})),\label{asymp H rationals}\\
				\mathscr{L}_{n+1}\left[\phi; \frac{1}{z-d},\frac{1}{z-c};a\right] = G^n[\phi] E[\phi] (a + O(\rho^{-n})),\nonumber\\
				\mathscr{E}_{n+1}\left[\phi; \frac{1}{z-d},\frac{1}{z-c};a\right]\nonumber\\
\qquad = G^n[\phi] E[\phi] \begin{cases}
					(a + O(\rho^{-n})), & \text{either $|c|<1$ or $|d|<1$}, \\
				 	\big( a + \tfrac{\al(c)}{\al(d^{-1})}\cdot \tfrac{1}{1-cd} + O(\rho^{-n}) \big), & \text{$|c|>1$ and $|d|>1$},
				\end{cases}	\nonumber
\\
				\mathscr{G}_{n+1}\left[\phi; \frac{1}{z-d},\frac{1}{z-c};a\right]\nonumber\\
\qquad = G^n[\phi] E[\phi] \begin{cases}
					(a + O(\rho^{-n})), & \text{either $|c|<1$ or $|d|<1$}, \\
					\big( a + \tfrac{\al(d)}{\al(c^{-1})}\cdot \tfrac{1}{1-cd} + O(\rho^{-n}) \big), & \text{$|c|>1$ and $|d|>1$}.
				\end{cases} 		\nonumber			
			\end{gather}
			Here the number $\rho$ is such that for $\la \in \{c,d\}$: a$)$ $1<\rho < |\la|$, if $|\la|>1$, b$)$ $|\la|<\rho^{-1}<1$, if~${|\la|<1}$, and c$)$ $\phi$ is analytic in the annulus $\big\{z\colon \rho^{-1}<|z|<\rho\big\}$.
		\end{Lemma}
		\begin{proof}
			The statements about the cases in which $|c|<1$ or $|d|<1$ are obvious in view of \eqref{q0 Fourier Coefficients}. We only prove the statements involving $\mathscr{H}$, since the proof of the statements about $\mathscr{L}$, $\mathscr{G}$ and $\mathscr{E}$ can be obtained in a similar way. Consider the case $|c|>1$ and $|d|>1$. Notice that
			\begin{gather*}
				\mathscr{H}_{n+2}\left[\phi;\frac{1}{z-d},\frac{1}{z-c};a\right] = \det \begin{bmatrix}
					\phi_0& \phi_{-1} & \cdots & \phi_{-n} & -d^{-1} \\
					\phi_{1}& \phi_0 & \cdots & \phi_{-n+1} & -d^{-2} \\
					\vdots & \vdots & \ddots & \vdots & \vdots \\
					\phi_{n} & \phi_{n-1} & \cdots & \phi_{0} & -d^{-n-1} \\
					-c^{-n-1} & -c^{-n} & \cdots & -c^{-1} & a
				\end{bmatrix}.
			\end{gather*}
			Recalling \eqref{ReproducingKer}, we observe that
			\begin{align*}
				\mathscr{H}_{n+2}\left[\phi;\frac{1}{z-d},\frac{1}{z-c};a\right] & = \frac{1}{d \cdot c^{n+1}} D_{n+1}[\phi] \widehat{K}_{n}\big(d^{-1},c;a d c^{n+1}\big) \\ & = \frac{1}{d \cdot c^{n+1}} D_{n+1}[\phi] \big( adc^{n+1} - K_{n}\big(d^{-1},c\big) \big),	
			\end{align*}
			where in the last step we have used Theorem~\ref{Rep Ker Semi Framed}. Therefore, using \eqref{425} and \eqref{316}, we find
			\begin{gather}
					\mathscr{H}_{n+2}\left[\phi;\frac{1}{z-d},\frac{1}{z-c};a\right]
= aD_{n+1}[\phi] - \frac{D_{n+1}[\phi]}{d \cdot c^{n+1}}\nonumber \\
\phantom{\qquad}{}\times \begin{cases}
						-(n+1) Q_{n+1}(d)\widehat{Q}_{n+1}\big(d^{-1}\big)+ d \big( \widehat{Q}_{n+1}\big(d^{-1}\big) \tfrac{\dd}{\dd z} Q_{n+1}(z) \big|_{z=d} & \\
\qquad{} - Q_{n+1}(d) \tfrac{\dd}{\dd z} \widehat{Q}_{n+1}\big(z^{-1}\big) \big|_{z=d} \big), & c = d, \\
						 \tfrac{d^{-n-1}Q_{n+1}(d)c^{n+1}\widehat{Q}_{n+1}\big(c^{-1}\big) - \widehat{Q}_{n+1}\big(d^{-1}\big) Q_{n+1}(c)}{1-d^{-1}c}, & c \neq d.
					\end{cases}\label{Hn+2 rationals}
\end{gather}
Now, recall from \eqref{Toeplitz-OP-solution} that $Q_{n+1}(z)=\varkappa_{n+1} X_{11}(z;n+1)$ and $\widehat{Q}_{n+1}\big(z^{-1}\big) = - \varkappa^{-1}_{n+1} z^{-n-1} X_{21}(z;\allowbreak n+2)$.
 Plugging these into \eqref{Hn+2 rationals} after straightforward simplifications we get \eqref{Hn+2 rationals X-RHP}. The asymptotics~\eqref{asymp H rationals} now follows from \eqref{Hn+2 rationals X-RHP}, \eqref{X in terms of R exact}, and the strong Szeg\H{o} theorem for pure Toeplitz determinants when $\phi$ is of Szeg\H{o}-type.
		\end{proof}

We have now arrived at the proof of Theorem \ref{thm semi-framed rationals intro}, as for instance, \eqref{asymp H rationals lin comb} follows from \eqref{framed lin comb}, \eqref{framed lin comb1}, and \eqref{asymp H rationals}.
		
\begin{Remark}
In view of Lemma \ref{thm semi-framed rationals}, we observe that the way one positions the Fourier coefficients in the last row and the last column of a semi-framed Toeplitz determinant does in fact affect the leading order term of the large-size asymptotics. This observation motivated us to revisit the bordered Toeplitz determinants with the reverse order of positioning the Fourier coefficients. Indeed, let us denote\footnote{Compare with \eqref{btd0}.}
			\begin{gather*}
				\mathcal{D}^{B}_{n}[\phi; \psi] := \det \begin{bmatrix}
					\phi_0& \phi_1 & \cdots & \phi_{n-2} & \psi_{0} \\
					\phi_{-1}& \phi_0 & \cdots & \phi_{n-3}&\psi_{1} \\
					\vdots & \vdots & \ddots & \vdots & \vdots\\
					\phi_{1-n} & \phi_{2-n} & \cdots & \phi_{-1}& \psi_{n-1}
				\end{bmatrix}, \qquad n \geq 2.
			\end{gather*}
			Then, for $\psi=q_2$ (given by \eqref{q1 q2}) by the same techniques used in \cite{BEGIL}, we obtain\footnote{Notice that $\mathcal{D}^{B}_{n+1}[\phi; 1]/D_n[\phi]=(-1)^{n}D_{n}[z\phi]/D_n[\phi] = \varkappa^{-1}_{n} Q_n(0) = X_{11}(0;n)$, where we have used~\eqref{DIK zphi} and \eqref{Toeplitz-OP-solution}. Also, note that $\mathcal{D}^{B}_{n+1}[\phi; z]/D_n[\phi]$ is the coefficient of $z$ in the polynomial $X_{11}(z;n)$, so $\mathcal{D}^{B}_{n+1}[\phi; z]/D_n[\phi] = X'_{11}(0;n)$.}
			\begin{gather}
				\frac{\mathcal{D}^{B}_{n+1}\left[\phi; \hat{a}_0+\hat{a}_1z+\tfrac{\hat{b}_0}{z}+	\sum_{j=1}^m \tfrac{\hat{b}_j}{z-c_j}\right]}{D_n[\phi]}\nonumber\\
\qquad = \hat{a}_0 X_{11}(0;n) + \hat{a}_0 X'_{11}(0;n) - \sum_{j=1 \atop |c_j|>1}^{m} \hat{b}_j c^{-1}_j X_{11}\big(c^{-1}_j;n\big).\label{DnB phi q2 reversed order X-RHP}
			\end{gather}
			Notice that the right-hand side of \eqref{DnB phi q2 reversed order X-RHP} is exponentially small as $n \to \infty$ in view of \eqref{X in terms of R exact} and~\eqref{R_k's are small}, as opposed to
			\begin{gather*}
				\frac{D^{B}_{n+1}\big[\phi; \hat{a}_0+\hat{a}_1z+\tfrac{\hat{b}_0}{z}+\sum_{j=1}^m \tfrac{\hat{b}_j}{z-c_j}\big]}{D_n[\phi]} = 			
				\hat{a}_0-\hat{a}_1[\log \phi]_{-1} -
				 \sum^{m}_{j=1 \atop |c_j|>1 } \frac{\hat{b}_j}{c_j} \al(c_j) + O({\rm e}^{-cn}),
			\end{gather*}
			for some $c>0$, which we have taken from Theorem \ref{main thm} (recall that $\al(0)=G[\phi]$). 	So we see that different ways of positioning the Fourier coefficients in the last column of a bordered Toeplitz determinant also changes the leading order term of the large-size asymptotics.
		\end{Remark}
		
\subsubsection{Proof of Theorem \ref{thm semi-framed rationals . phi intro}}
		
		Throughout this section, we assume that $\phi$ is a Szeg\H{o}-type symbol, so that we can refer to the Riemann--Hilbert analysis reminded in Appendix~\ref{Appendices}.
		
		\begin{Lemma} Let $\phi$ be a Szeg\H{o}-type symbol, and suppose that $|c|, |d| \neq 1$. It holds that as $n \to \infty$
			\begin{gather}
				\mathscr{H}_{n+1}\left[\phi;\di \frac{\phi}{z-d}, \frac{\phi}{z-c} ;a\right] = G[\phi]^n E[\phi] (a + O(\rho^{-n})),\label{H phi . rational}
\\
				\mathscr{L}_{n+1}\left[\phi;\di \frac{\tilde{\phi}}{z-d}, \frac{\tilde{\phi}}{z-c} ;a\right] = G[\phi]^n E[\phi] (a + O(\rho^{-n})),\label{L phi . rational}
\\
				\mathscr{E}_{n+1}\left[\phi;\di \frac{\tilde{\phi}}{z-d}, \frac{\phi}{z-c} ;a\right] \nonumber\\
\qquad= G[\phi]^n E[\phi] \begin{cases}
					a+ \tfrac{\al(c)}{\al(d^{-1})}\cdot \tfrac{1}{1-cd}+O(\rho^{-n}), & \text{if $|c|<1$ and $|d|<1$}, \\
					a+O(\rho^{-n}), & \text{either $|c|>1$ or $|d|>1$},
				\end{cases}\label{E phi . rational}
			\end{gather}
			and
			\begin{gather}
				\mathscr{G}_{n+1}\left[\phi; \frac{\phi}{z-d}, \frac{\tilde{\phi}}{z-c} ;a\right]\nonumber\\
\qquad = G[\phi]^n E[\phi] \begin{cases}
					a+ \tfrac{\al(d)}{\al(c^{-1})}\cdot \tfrac{1}{1-cd}+O(\rho^{-n}), & \text{if $|c|<1$ and $|d|<1$}, \\
					a+O(\rho^{-n}), & \text{either $|c|>1$ or $|d|>1$}.
				\end{cases}\label{G phi . rational}
			\end{gather}
			Here the number $\rho$ is such that for $\la \in \{c,d\}$: a$)$ $1<\rho < |\la|$, if $|\la|>1$, b$)$ $|\la|<\rho^{-1}<1$, if~${|\la|<1}$, and c$)$ $\phi$ is analytic in the annulus $\big\{z\colon \rho^{-1}<|z|<\rho\big\}$.
		\end{Lemma}
		\begin{proof}
			We only prove \eqref{E phi . rational} as \eqref{H phi . rational}, \eqref{L phi . rational}, and \eqref{G phi . rational} can be obtained similarly. 			From \eqref{X in terms of R exact} in $\Om_1$ (see Figure \ref{S_contour}), we have
			\begin{align*}
				X_{11}(z;n) & = z^n \al(z) \phi^{-1}(z) ( 1+ O(\rho^{-n})),\qquad n \to \infty, \\
				X_{21}(z;n) & = - \al^{-1}(z) ( 1+ O(\rho^{-n})),\qquad n \to \infty.
			\end{align*}
			Therefore, recalling \eqref{E and X-RHP}, we have
			\begin{gather*}
					\frac{\mathscr{E}_{n+2}[\phi;\psi, \eta ;a]}{D_{n+1}[\phi]} - a \\
\qquad = - \int_{\T} \int_{\T} \frac{z^{-n}_2\eta(z_2) \Tilde{\psi}(z_1)}{z_1-z_2} \det \begin{bmatrix}
						X_{11}(z_2;n+1) & X_{21}(z_2;n+2) \\
						X_{11}(z_1;n+1) & X_{21}(z_1;n+2)
					\end{bmatrix} \frac{\dd z_2}{2 \pi \ii z_2} \frac{\dd z_1}{2 \pi \ii z_1} \\
\qquad= - \int_{\T} \int_{\T} \frac{z^{-n}_2\eta(z_2) \Tilde{\psi}(z_1)}{z_1-z_2} \bigl( - z^{n+1}_2 \al_+(z_2) \phi^{-1}(z_2) \al^{-1}_+(z_1) \bigr)\frac{\dd z_2}{2 \pi \ii z_2} \frac{\dd z_1}{2 \pi \ii z_1} \\
\phantom{\qquad=}{} - \int_{\T} \int_{\T} \frac{z^{-n}_2\eta(z_2) \Tilde{\psi}(z_1)}{z_1-z_2} \big( z^{n+1}_1 \al_+(z_1) \phi^{-1}(z_1) \al^{-1}_+(z_2) \big)\frac{\dd z_2}{2 \pi \ii z_2} \frac{\dd z_1}{2 \pi \ii z_1} + O(\rho^{-n}).
			\end{gather*}
			In view of \eqref{47}, we can rewrite this as
			\begin{align*}
					\frac{\mathscr{E}_{n+2}[\phi;\psi, \eta ;a]}{D_{n+1}[\phi]} - a ={} &\int_{\T} \int_{\T} \frac{\eta(z_2) \Tilde{\psi}(z_1)}{z_1-z_2} \al_+(z_2) \phi^{-1}(z_2) \phi^{-1}(z_1) \al^{-1}_-(z_1) \frac{\dd z_2}{2 \pi \ii } \frac{\dd z_1}{2 \pi \ii z_1} \\
 & - \int_{\T} \int_{\T} \frac{\eta(z_2) \Tilde{\psi}(z_1)}{z_1-z_2} \left( \frac{z_1}{z_2} \right)^n \al_+(z_1) \phi^{-1}(z_1) \phi^{-1}(z_2) \al^{-1}_-(z_2) \frac{\dd z_2}{2 \pi \ii z_2} \frac{\dd z_1}{2 \pi \ii }\\
 & + O(\rho^{-n}).
			\end{align*}
			For $\mathfrak{r}>1$, define $	\T^{\mathfrak{r}}_{\mp} := \big\{ z\colon |z|= \mathfrak{r}^{\pm1} \big\}$ and $	\D^{\mathfrak{r}}_{\mp} := \big\{ z\colon |z|< \mathfrak{r}^{\pm1} \big\}$. We choose $\mathfrak{r}$ so that $\psi$, $\eta$, and $\phi$ are analytic in $\D^{\mathfrak{r}}_- \setminus \overline{\D^{\mathfrak{r}}_+}$. With this choice of $\mathfrak{r}$ we deform the contours of integration to rewrite the previous equation as
			\begin{gather}
					\frac{\mathscr{E}_{n+2}[\phi;\psi, \eta ;a]}{D_{n+1}[\phi]} - a\nonumber\\
\qquad = \int_{\T^{\mathfrak{r}}_{-}} \int_{\T^{\mathfrak{r}}_{+}} \frac{\eta(z_2) \Tilde{\psi}(z_1)}{z_1-z_2} \al(z_2) \phi^{-1}(z_2) \phi^{-1}(z_1) \al^{-1}(z_1) \frac{\dd z_2}{2 \pi \ii } \frac{\dd z_1}{2 \pi \ii z_1}\nonumber \\
\phantom{\qquad =}{} - \int_{\T^{\mathfrak{r}}_{+}} \int_{\T^{\mathfrak{r}}_{-}} \frac{\eta(z_2) \Tilde{\psi}(z_1)}{z_1-z_2} \left( \frac{z_1}{z_2} \right)^n \al(z_1) \phi^{-1}(z_1) \phi^{-1}(z_2) \al^{-1}(z_2) \frac{\dd z_2}{2 \pi \ii z_2} \frac{\dd z_1}{2 \pi \ii } + O(\rho^{-n}) \nonumber\\
\qquad = \int_{\T^{\mathfrak{r}}_{-}} \int_{\T^{\mathfrak{r}}_{+}} z^{-1}_1\eta(z_2) \Tilde{\psi}(z_1) \al(z_2) \phi^{-1}(z_2) \phi^{-1}(z_1) \al^{-1}(z_1) \left( \sum_{k=0}^{\infty} \frac{z^k_2}{z^k_1} \right) \frac{\dd z_2}{2 \pi \ii } \frac{\dd z_1}{2 \pi \ii z_1} \nonumber\\
\phantom{\qquad =}{} + \int_{\T^{\mathfrak{r}}_{+}} \int_{\T^{\mathfrak{r}}_{-}} z^{-1}_2\eta(z_2) \Tilde{\psi}(z_1) \al(z_1) \phi^{-1}(z_1) \phi^{-1}(z_2) \al^{-1}(z_2) \left( \sum_{k=0}^{\infty} \frac{z^{k+n}_1}{z^{k+n}_2} \right) \frac{\dd z_2}{2 \pi \ii z_2} \frac{\dd z_1}{2 \pi \ii }\nonumber \\
\phantom{\qquad =}{} + O(\rho^{-n}) \nonumber\\
\qquad = \sum_{k=0}^{\infty} \left[ \int_{\T^{\mathfrak{r}}_{+}} \frac{\eta(z)}{\phi(z)} \al(z) z^k \frac{\dd z}{2 \pi \ii} \right] \left[ \int_{\T^{\mathfrak{r}}_{-}} \frac{\Tilde{\psi}(z)}{\phi(z)} \al^{-1}(z) z^{-k-1} \frac{\dd z}{2 \pi \ii z} \right] \nonumber\\
\phantom{\qquad =}{} + \sum_{k=0}^{\infty} \left[ \int_{\T^{\mathfrak{r}}_{+}} \frac{\Tilde{\psi}(z)}{\phi(z)} \al(z) z^{k+n} \frac{\dd z}{2 \pi \ii} \right] \left[ \int_{\T^{\mathfrak{r}}_{-}} \frac{\eta(z)}{\phi(z)} \al^{-1}(z) z^{-n-k-1} \frac{\dd z}{2 \pi \ii z} \right] + O(\rho^{-n}).\label{3.61 a}
			\end{gather}
			
			Since $|c| \neq 1$ and $|d| \neq 1$, we can choose $\varepsilon>0$ small enough so that $c$ and $d$ do not belong to~$\D^{\ep}_- \setminus \overline{\D^{\ep}_+}$. Replacing $\psi$ and $\eta$ respectively by \smash{$\frac{\tilde{\phi}}{z-d}$} and $\frac{\phi}{z-c}$ in the last member of \eqref{3.61 a} gives
			\begin{gather}
					\frac{\mathscr{E}_{n+2}\big[\phi;\di \tfrac{\tilde{\phi}}{z-d}, \tfrac{\phi}{z-c} ;a\big]}{D_{n+1}[\phi]} - a\nonumber\\
\qquad\simeq \sum_{k=0}^{\infty} \left[ \int_{\T^{\mathfrak{r}}_{+}} \frac{\al(z) z^k}{z-c} \frac{\dd z}{2 \pi \ii} \right] \left[ \int_{\T^{\mathfrak{r}}_{-}} \frac{\al^{-1}(z) z^{-k-1}}{z^{-1}-d} \frac{\dd z}{2 \pi \ii z} \right] \nonumber\\
\phantom{\qquad\simeq}{}+ \sum_{k=0}^{\infty} \left[ \int_{\T^{\mathfrak{r}}_{+}} \frac{\al(z) z^{k+n}}{z^{-1}-d} \frac{\dd z}{2 \pi \ii} \right] \left[ \int_{\T^{\mathfrak{r}}_{-}} \frac{\al^{-1}(z) z^{-n-k-1}}{z-c} \frac{\dd z}{2 \pi \ii z} \right] \nonumber\\
\qquad= \sum_{k=0}^{\infty} \left[ \int_{\T^{\mathfrak{r}}_{+}} \frac{\al(z) z^k}{z-c} \frac{\dd z}{2 \pi \ii} \right] \left[ \int_{\T^{\mathfrak{r}}_{+}} \frac{\al^{-1}\big(z^{-1}\big) z^{k}}{z-d} \frac{\dd z}{2 \pi \ii} \right]\nonumber\\
\phantom{\qquad\simeq}{}+ \frac{1}{c\cdot d} \sum_{k=0}^{\infty} \left[
					\int_{\T^{\mathfrak{r}}_{+}} \frac{\al(z) z^{n+k+1}}{z-d^{-1}} \frac{\dd z}{2 \pi \ii} \right] \left[ \int_{\T^{\mathfrak{r}}_{+}} \frac{\al^{-1}\big(z^{-1}\big) z^{n+k+1}}{z-c^{-1}} \frac{\dd z}{2 \pi \ii} \right].\label{3.63 a}
			\end{gather}
			Since $\al$ is analytic in $\C \setminus \T$, we immediately have
			\begin{equation*}
				\int_{\T^{\mathfrak{r}}_{+}} \frac{\al(z) z^k}{z-c} \frac{\dd z}{2 \pi \ii} = \begin{cases}
					\al(c)c^k, & |c|<1, \\
					0, & |c|>1,
				\end{cases}
			\end{equation*} and
			\begin{equation*}
				\int_{\T^{\mathfrak{r}}_{+}} \frac{\al(z) z^{n+k+1}}{z-d^{-1}} \frac{\dd z}{2 \pi \ii} = \begin{cases}
					0, & |d|<1, \\
					\al(d^{-1}) d^{-n-k-1}, & |d|>1.
				\end{cases}
			\end{equation*}
			Recall that \[			\al(z)=\exp \left[ \frac{1}{2 \pi {\rm i} } \int_{\T} \frac{\ln(\phi(\tau))}{\tau-z}{\rm d}\tau \right].\] Since $\al(z) \neq 0$ where it is defined, $\al^{-1}\big(z^{-1}\big)$ is also analytic in $\C \setminus \T$. So
			\begin{equation*}
				\int_{\T^{\mathfrak{r}}_{+}} \frac{\al^{-1}\big(z^{-1}\big) z^{k}}{z-d} \frac{\dd z}{2 \pi \ii} = \begin{cases}
					\al^{-1}\big(d^{-1}\big) d^{k}, & |d|<1, \\
					0, & |d|>1,
				\end{cases}
			\end{equation*}
			and
			\begin{equation*}
				\int_{\T^{\mathfrak{r}}_{+}} \frac{\al^{-1}\big(z^{-1}\big) z^{n+k+1}}{z-c^{-1}} \frac{\dd z}{2 \pi \ii} = \begin{cases}
					0, & |c|<1, \\
					\al^{-1}(c) c^{-n-k-1}, & |c|>1.
				\end{cases}
			\end{equation*}
			Using these in last member of \eqref{3.63 a} and summing up the resulting geometric series, we obtain\footnote{We choose $\rho>1$ so that there are no poles in the annulus with radii $\rho$ and $\rho^{-1}$, so $c^{-n} = O(\rho^{-n})$ when~${|c|>1}$.}
			\begin{align*}
					\frac{\mathscr{E}_{n+2}\big[\phi; \tfrac{\tilde{\phi}}{z-d}, \tfrac{\phi}{z-c} ;a\big]}{D_{n+1}[\phi]} - a = & \begin{cases}
						 \dfrac{\al(c)}{\al\big(d^{-1}\big)}\cdot \dfrac{1}{1-cd} + O(\rho^{-n}), 		 & \text{if $|c|<1$ and $|d|<1$}, \\
						O(\rho^{-n}), & \text{either $|c|>1$ or $|d|>1$}.
					\end{cases}
			\end{align*}
			Changing $n \mapsto n-1$ and recalling the strong Szeg\H{o} theorem, we obtain \eqref{E phi . rational}.		
		\end{proof}

We have now arrived at the proof of Theorem \ref{thm semi-framed rationals . phi intro}, as the asymptotic formulas \eqref{asymp H rationals lin comb1} through~\eqref{asymp G rationals lin comb1} simply follow from the above lemma and equations \eqref{framed lin comb} and \eqref{framed lin comb1}.
		
		\subsection[Beyond the semi-framed case: framed and multi-framed Toeplitz determinants]{Beyond the semi-framed case: \\
framed and multi-framed Toeplitz determinants}\label{sec framed and multi-framed}

		Finally, we turn our attention to the framed and multi-framed Toeplitz determinants, which are determinants of matrices of Toeplitz structure in addition to one or multiple frames around them (recall \eqref{framed intro} and \eqref{MM}). As mentioned before, in this section we do not intend to prove any asymptotic results, instead we would like to highlight the general framework for how one may approach such asymptotic analysis.
		
		For a framed Toeplitz determinant, there are many ways one can place the Fourier coefficients of the symbols along the four borderes. For example, consider an $(n+3) \times (n+3)$ framed Toeplitz determinant with border symbols $\psi$, $\eta$, $\xi$ and $\ga$, respectively for the right, bottom, top and left borders. Then, if we want to use the zeroth up to the $n$-th Fourier coefficients of these symbols along the borders, we have sixteen ways to construct such framed Toeplitz determinants.\footnote{Recall the semi-framed Toeplitz determinants where we had four forms \eqref{half-framed}--\eqref{half-framed 3}.} Nevertheless, for any of these choices, a framed Toeplitz determinant can be related to four semi-framed Toeplitz determinants by the following Dodgson condensation identity:
		\begin{gather}\label{DODGSON framed B}
\underbrace{\mathscr{M}}_{\text{framed}}
			\cdot \underbrace{ \mathscr{M}\left\lbrace \begin{matrix} 0 & n+2 \\ 0 & n+2 \end{matrix} \right\rbrace }_{\text{pure Toeplitz}} = \underbrace{ \mathscr{M}\left\lbrace \begin{matrix} 0 \\ 0 \end{matrix} \right\rbrace }_{\text{semi-framed}} \cdot \underbrace{ \mathscr{M}\left\lbrace \begin{matrix} n+2 \\ n+2 \end{matrix} \right\rbrace}_{\text{semi-framed}} - \underbrace{ \mathscr{M}\left\lbrace \begin{matrix} 0 \\ n+2 \end{matrix} \right\rbrace}_{\text{semi-framed}} \cdot \underbrace{ \mathscr{M}\left\lbrace \begin{matrix} n+2 \\ 0 \end{matrix} \right\rbrace}_{\text{semi-framed}}.
		\end{gather}
		
For example, among the aforementioned sixteen choices, suppose we want to study the asymptotics of
		\begin{gather*}
			\mathscr{M}_{n+3}[\phi; \xi, \psi, \eta, \ga; \boldsymbol{a}_4] := \det \begin{bmatrix}
				a_1 & \xi_{n} & \xi_{n-1} & \cdots & \xi_{0} & a_2 \\
				\ga_{n} &	\phi_0& \phi_{-1} & \cdots & \phi_{-n} & \psi_{0} \\
				\ga_{n-1} &	\phi_{1}& \phi_0 & \cdots & \phi_{-n+1} & \psi_{1} \\
				\vdots &	\vdots & \vdots & \ddots & \vdots & \vdots \\
				\ga_0 &	\phi_{n} & \phi_{n-1} & \cdots & \phi_{0} & \psi_{n} \\
				a_4 &	\eta_{n} & \eta_{n-1} & \cdots & \eta_{0} & a_3
			\end{bmatrix},
		\end{gather*}
		or
		\begin{gather*}
			\mathscr{N}_{n+3}[\phi; \xi, \psi, \eta, \ga; \boldsymbol{a}_4] := \det \begin{bmatrix}
				a_1 & \xi_{0} & \xi_{1} & \cdots & \xi_{n} & a_2 \\
				\ga_{0} &	\phi_0& \phi_{-1} & \cdots & \phi_{-n} & \psi_{n} \\
				\ga_{1} &	\phi_{1}& \phi_0 & \cdots & \phi_{-n+1} & \psi_{n-1} \\
				\vdots &	\vdots & \vdots & \ddots & \vdots & \vdots \\
				\ga_n &	\phi_{n} & \phi_{n-1} & \cdots & \phi_{0} & \psi_{0} \\
				a_4 &	\eta_{n} & \eta_{n-1} & \cdots & \eta_{0} & a_3
			\end{bmatrix},
		\end{gather*}
		where $\boldsymbol{a}_4$ denotes the ordered set $\{a_1, a_2, a_3, a_4\}$, and $a_k$'s are arbitrary complex numbers, ${k=1,\dots,4}$. Employing the Dodgson condensation identity \eqref{DODGSON framed B} for $\mathscr{M}$ and $\mathscr{N}$, we respectively obtain
		\begin{align}
			\mathscr{M}_{n+3}[\phi; \xi, \psi, \eta, \ga; \boldsymbol{a}_4] \cdot D_{n+1}[\phi] ={}& \mathscr{H}_{n+2}[\phi;\psi,\eta;a_3] \cdot \mathscr{E}_{n+2}[\phi;\ga,\xi;a_1]\nonumber \\
 &- \mathscr{E}_{n+2}[\phi;\ga,\eta;a_4] \cdot \mathscr{H}_{n+2}[\phi;\psi,\xi;a_2], 	\label{M in terms of semi-framed stuff}
		\end{align}
		and
		\begin{align*}
			\mathscr{N}_{n+3}[\phi; \xi, \psi, \eta, \ga; \boldsymbol{a}_4] \cdot D_{n+1}[\phi]={}& \mathscr{E}_{n+2}[\phi;\psi,\eta;a_3] \cdot \mathscr{G}_{n+2}[\phi;\ga,\xi;a_1] \\
&- \mathscr{H}_{n+2}[\phi;\ga,\eta;a_4] \cdot \mathscr{L}_{n+2}[\phi;\psi,\xi;a_2].
		\end{align*}
		These identities and their analogues for the other fourteen framed Toeplitz determinants, provide a pathway to the asymptotics at least for the class of border symbols considered in Section~\ref{sec semi-framed}, since we already know how to compute the asymptotics of the semi-framed Toeplitz determinants appearing on the right-hand side (see Section~\ref{sec semi-framed}).
		
		When dealing with a multi-framed Toeplitz determinant, repeated application of appropriate Dodgson condensation identities can simplify the analysis, ultimately reducing it to the asymptotic analysis of semi-framed Toeplitz determinants once again. For example, consider the following $(n+5) \times (n+5)$ two-framed Toeplitz determinant:
		\begin{gather}
			\mathscr{K}_{n+5}[\phi; \boldsymbol{\xi}, \boldsymbol{\psi}, \boldsymbol{\eta}, \boldsymbol{\ga}; \boldsymbol{a}_8]\nonumber\\
\qquad := \de\begin{bmatrix}
				a_5 & \xi_{2,n+2} & \xi_{2,n+1} & \xi_{2,n} & \cdots & \xi_{2,1} & \xi_{2,0} & a_6 \\
				\ga_{2,n+2} &	a_1 & \xi_{1,n} & \xi_{1,n-1} & \cdots & \xi_{1,0} & a_2 & \psi_{2,0} \\
				\ga_{2,n+1} &	\ga_{1,n} &	\phi_0& \phi_{-1} & \cdots & \phi_{-n} & \psi_{1,0} & \psi_{2,1} \\
				\ga_{2,n} &	\ga_{1,n-1} &	\phi_{1}& \phi_0 & \cdots & \phi_{-n+1} & \psi_{1,1} & \psi_{2,2} \\
				\vdots &	\vdots & \vdots & \vdots & \ddots & \vdots & \vdots & \vdots \\
				\ga_{2,1} &	\ga_{1,0} &	\phi_{n} & \phi_{n-1} & \cdots & \phi_{0} & \psi_{1,n} & \psi_{2,n+1} \\
				\ga_{2,0} &	a_4 &	\eta_{1,n} & \eta_{1,n-1} & \cdots & \eta_{1,0} & a_3 & \psi_{2,n+2} \\
				a_8 &	\eta_{2,n+2} &	\eta_{2,n+1} & \eta_{2,n} & \cdots & \eta_{2,1} & \eta_{2,0} & a_7
			\end{bmatrix}.\label{two framed}
		\end{gather}
		To this two-framed Toeplitz determinant we apply
		\begin{equation}\label{DODGSON framed C} \underbrace{\mathscr{K}}_{\text{two-framed}}
			\cdot \underbrace{ \mathscr{K}\left\lbrace \begin{matrix} 0 & n+4 \\ 0 & n+4 \end{matrix} \right\rbrace }_{\text{framed}} = \mathscr{K}\left\lbrace \begin{matrix} 0 \\ 0 \end{matrix} \right\rbrace \cdot \mathscr{K}\left\lbrace \begin{matrix} n+4 \\ n+4 \end{matrix} \right\rbrace - \mathscr{K}\left\lbrace \begin{matrix} 0 \\ n+4 \end{matrix} \right\rbrace \cdot \mathscr{K}\left\lbrace \begin{matrix} n+4 \\ 0 \end{matrix} \right\rbrace,
		\end{equation}
		where each of the determinants on the right-hand side are that of a framed Toeplitz matrix with an extra \textit{semi-frame} around it. Indeed,
		\begin{gather*}
			\mathscr{K}\left\lbrace \begin{matrix} 0 \\ 0 \end{matrix} \right\rbrace = \det \begin{bmatrix}
				a_1 & \xi_{1,n} & \xi_{1,n-1} & \cdots & \xi_{1,0} & a_2 & \psi_{2,0} \\
				\ga_{1,n} &	\phi_0& \phi_{-1} & \cdots & \phi_{-n} & \psi_{1,0} & \psi_{2,1} \\
				\ga_{1,n-1} &	\phi_{1}& \phi_0 & \cdots & \phi_{-n+1} & \psi_{1,1} & \psi_{2,2} \\
				\vdots & \vdots & \vdots & \ddots & \vdots & \vdots & \vdots \\
				\ga_{1,0} &	\phi_{n} & \phi_{n-1} & \cdots & \phi_{0} & \psi_{1,n} & \psi_{2,n+1} \\
				a_4 &	\eta_{1,n} & \eta_{1,n-1} & \cdots & \eta_{1,0} & a_3 & \psi_{2,n+2} \\
				\eta_{2,n+2} &	\eta_{2,n+1} & \eta_{2,n} & \cdots & \eta_{2,1} & \eta_{2,0} & a_7
			\end{bmatrix},
\\
			\mathscr{K}\left\lbrace \begin{matrix} n+4 \\ n+4 \end{matrix} \right\rbrace	= \det	\begin{bmatrix}
				a_5 & \xi_{2,n+2} & \xi_{2,n+1} & \xi_{2,n} & \cdots & \xi_{2,1} & \xi_{2,0} \\
				\ga_{2,n+2} &	a_1 & \xi_{1,n} & \xi_{1,n-1} & \cdots & \xi_{1,0} & a_2 \\
				\ga_{2,n+1} &	\ga_{1,n} &	\phi_0& \phi_{-1} & \cdots & \phi_{-n} & \psi_{1,0} \\
				\ga_{2,n} &	\ga_{1,n-1} &	\phi_{1}& \phi_0 & \cdots & \phi_{-n+1} & \psi_{1,1} \\
				\vdots &	\vdots & \vdots & \vdots & \ddots & \vdots & \vdots \\
				\ga_{2,1} &	\ga_{1,0} &	\phi_{n} & \phi_{n-1} & \cdots & \phi_{0} & \psi_{1,n} \\
				\ga_{2,0} &	a_4 &	\eta_{1,n} & \eta_{1,n-1} & \cdots & \eta_{1,0} & a_3 &
			\end{bmatrix},
\\
			\mathscr{K}\left\lbrace \begin{matrix} 0 \\ n+4 \end{matrix} \right\rbrace	= \det \begin{bmatrix}
				\ga_{2,n+2} &	a_1 & \xi_{1,n} & \xi_{1,n-1} & \cdots & \xi_{1,0} & a_2 \\
				\ga_{2,n+1} &	\ga_{1,n} &	\phi_0& \phi_{-1} & \cdots & \phi_{-n} & \psi_{1,0} \\
				\ga_{2,n} &	\ga_{1,n-1} &	\phi_{1}& \phi_0 & \cdots & \phi_{-n+1} & \psi_{1,1} \\
				\vdots &	\vdots & \vdots & \vdots & \ddots & \vdots & \vdots \\
				\ga_{2,1} &	\ga_{1,0} &	\phi_{n} & \phi_{n-1} & \cdots & \phi_{0} & \psi_{1,n} \\
				\ga_{2,0} &	a_4 &	\eta_{1,n} & \eta_{1,n-1} & \cdots & \eta_{1,0} & a_3 \\
				a_8 &	\eta_{2,n+2} &	\eta_{2,n+1} & \eta_{2,n} & \cdots & \eta_{2,1} & \eta_{2,0}
			\end{bmatrix},
		\end{gather*}
		and
		\begin{equation*}
			\mathscr{K}\left\lbrace \begin{matrix} n+4 \\ 0 \end{matrix} \right\rbrace	= \det \begin{bmatrix}
				\xi_{2,n+2} & \xi_{2,n+1} & \xi_{2,n} & \cdots & \xi_{2,1} & \xi_{2,0} & a_6 \\
				a_1 & \xi_{1,n} & \xi_{1,n-1} & \cdots & \xi_{1,0} & a_2 & \psi_{2,0} \\
				\ga_{1,n} &	\phi_0& \phi_{-1} & \cdots & \phi_{-n} & \psi_{1,0} & \psi_{2,1} \\
				\ga_{1,n-1} &	\phi_{1}& \phi_0 & \cdots & \phi_{-n+1} & \psi_{1,1} & \psi_{2,2} \\
				\vdots & \vdots & \vdots & \ddots & \vdots & \vdots & \vdots \\
				\ga_{1,0} &	\phi_{n} & \phi_{n-1} & \cdots & \phi_{0} & \psi_{1,n} & \psi_{2,n+1} \\
				a_4 &	\eta_{1,n} & \eta_{1,n-1} & \cdots & \eta_{1,0} & a_3 & \psi_{2,n+2} 			\end{bmatrix}.
		\end{equation*}
		Now consider the following auxiliary DCIs:
		\begin{gather}
\mathscr{K}\left\lbrace \begin{matrix} 0 \\ 0 \end{matrix} \right\rbrace
			\cdot \underbrace{ \mathscr{K}\left\lbrace \begin{matrix} 0 & n+3 & n+4 \\ 0 & n+3 & n+4 \end{matrix} \right\rbrace }_{\text{semi-framed}}\nonumber\\
\qquad = \underbrace{\mathscr{K}\left\lbrace \begin{matrix} 0 & n+3 \\ 0 & n+3 \end{matrix} \right\rbrace}_{\text{framed}} \cdot \underbrace{\mathscr{K}\left\lbrace \begin{matrix} 0 & n+4 \\ 0 & n+4 \end{matrix} \right\rbrace}_{\text{framed}} - \underbrace{\mathscr{K}\left\lbrace \begin{matrix} 0 & n+3 \\ 0 & n+4 \end{matrix} \right\rbrace}_{\text{framed}} \cdot \underbrace{\mathscr{K}\left\lbrace \begin{matrix} 0 & n+4 \\ 0 & n+3 \end{matrix} \right\rbrace}_{\text{framed}},\label{DODGSON framed CC}
\\
\mathscr{K}\left\lbrace \begin{matrix} n+4 \\ n+4 \end{matrix} \right\rbrace
			\cdot \underbrace{ \mathscr{K}\left\lbrace \begin{matrix} 0 & 1 & n+4 \\ 0 & 1 & n+4 \end{matrix} \right\rbrace }_{\text{semi-framed}} \nonumber\\
\qquad= \underbrace{\mathscr{K}\left\lbrace \begin{matrix} 0 & n+4 \\ 0 & n+4 \end{matrix} \right\rbrace}_{\text{framed}} \cdot \underbrace{\mathscr{K}\left\lbrace \begin{matrix} 1 & n+4 \\ 1 & n+4 \end{matrix} \right\rbrace}_{\text{framed}} - \underbrace{\mathscr{K}\left\lbrace \begin{matrix} 0 & n+4 \\ 1 & n+4 \end{matrix} \right\rbrace}_{\text{framed}} \cdot \underbrace{\mathscr{K}\left\lbrace \begin{matrix} 1 & n+4 \\ 0 & n+4 \end{matrix} \right\rbrace}_{\text{framed}},\nonumber
\\
 \mathscr{K}\left\lbrace \begin{matrix} 0 \\ n+4 \end{matrix} \right\rbrace
			\cdot \underbrace{ \mathscr{K}\left\lbrace \begin{matrix} 0 & n+3 & n+4 \\ 0 & 1 & n+4 \end{matrix} \right\rbrace }_{\text{semi-framed}} \nonumber\\
\qquad= \underbrace{\mathscr{K}\left\lbrace \begin{matrix} 0 & n+3 \\ 0 & n+4 \end{matrix} \right\rbrace}_{\text{framed}} \cdot \underbrace{\mathscr{K}\left\lbrace \begin{matrix} 0 & n+4 \\ 1 & n+4 \end{matrix} \right\rbrace}_{\text{framed}} - \underbrace{\mathscr{K}\left\lbrace \begin{matrix} 0 & n+3 \\ 1 & n+4 \end{matrix} \right\rbrace}_{\text{framed}} \cdot \underbrace{\mathscr{K}\left\lbrace \begin{matrix} 0 & n+4 \\ 0 & n+4 \end{matrix} \right\rbrace}_{\text{framed}},\nonumber
		\end{gather}
		and
		\begin{gather}
\mathscr{K}\left\lbrace \begin{matrix} n+4 \\ 0 \end{matrix} \right\rbrace
			\cdot \underbrace{ \mathscr{K}\left\lbrace \begin{matrix} 0 & 1 & n+4 \\ 0 & n+3 & n+4 \end{matrix} \right\rbrace }_{\text{semi-framed}} \nonumber\\
\qquad= \underbrace{\mathscr{K}\left\lbrace \begin{matrix} 0 & n+4 \\ 0 & n+3 \end{matrix} \right\rbrace}_{\text{framed}} \cdot \underbrace{\mathscr{K}\left\lbrace \begin{matrix} 1 & n+4 \\ 0 & n+4 \end{matrix} \right\rbrace}_{\text{framed}} - \underbrace{\mathscr{K}\left\lbrace \begin{matrix} 0 & n+4 \\ 0 & n+4 \end{matrix} \right\rbrace}_{\text{framed}} \cdot \underbrace{\mathscr{K}\left\lbrace \begin{matrix} 1 & n+4 \\ 0 & n+3 \end{matrix} \right\rbrace}_{\text{framed}}.\label{DODGSON framed F}
		\end{gather}
		Using these, we can express the objects on the right-hand side of \eqref{DODGSON framed C} in terms of the framed Toeplitz determinant $\mathscr{M}$ introduced in \eqref{MM}. Indeed, from the equations \eqref{DODGSON framed CC} through \eqref{DODGSON framed F} we respectively obtain
		\begin{gather*}
				\mathscr{K}\left\lbrace \begin{matrix} 0 \\ 0 \end{matrix} \right\rbrace = \frac{1}{\mathscr{E}_{n+2}[\phi;\ga_1,\xi_1;a_1]} \left( \mathscr{M}_{n+3}\left[\phi; \xi_1, \frac{\psi_2}{z}, \frac{\eta_2}{z}, \ga_1; \boldsymbol{a}^{(1)}_4\right]\cdot \mathscr{M}_{n+3}\big[\phi; \xi_1, \psi_1, \eta_1, \ga_1; \boldsymbol{a}^{(2)}_4\big] \right. \\
\phantom{\mathscr{K}\left\lbrace \begin{matrix} 0 \\ 0 \end{matrix} \right\rbrace =}{} - \left. \mathscr{M}_{n+3}\left[\phi; \xi_1, \psi_1, \frac{\eta_2}{z}, \ga_1; \boldsymbol{a}^{(3)}_4\right]\cdot \mathscr{M}_{n+3}\left[\phi; \xi_1, \frac{\psi_2}{z}, \eta_1, \ga_1; \boldsymbol{a}^{(4)}_4\right] \right),
\\
				\mathscr{K}\left\lbrace \begin{matrix} n+4 \\ n+4 \end{matrix} \right\rbrace = \frac{1}{\mathscr{H}_{n+2}[\phi;\psi_1,\eta_1;a_3]}\\
\phantom{\mathscr{K}\left\lbrace \begin{matrix} n+4 \\ n+4 \end{matrix} \right\rbrace= }{} \times \left( \mathscr{M}_{n+3}\big[\phi; \xi_1, \psi_1, \eta_1, \ga_1; \boldsymbol{b}^{(1)}_4\big]\cdot \mathscr{M}_{n+3}\left[\phi; \frac{\xi_2}{z}, \psi_1, \eta_1, \frac{\ga_2}{z}; \boldsymbol{b}^{(2)}_4\right] \right. \\
\phantom{\mathscr{K}\left\lbrace \begin{matrix} n+4 \\ n+4 \end{matrix} \right\rbrace =\times}{} - \left. \mathscr{M}_{n+3}\left[\phi; \xi_1, \psi_1, \eta_1, \frac{\ga_2}{z}; \boldsymbol{b}^{(3)}_4\right]\cdot \mathscr{M}_{n+3}\left[\phi; \frac{\xi_2}{z}, \psi_1, \eta_1, \ga_1; \boldsymbol{b}^{(4)}_4\right] \right),
\\
				\mathscr{K}\left\lbrace \begin{matrix} 0 \\ n+4 \end{matrix} \right\rbrace = \frac{(-1)^{n+1}}{\mathscr{H}_{n+2}[\phi;\psi_1,\xi_1;a_2]} \\
\phantom{\mathscr{K}\left\lbrace \begin{matrix} 0 \\ n+4 \end{matrix} \right\rbrace =}{}\times \left( \mathscr{M}_{n+3}\left[\phi; \xi_1, \psi_1, \frac{\eta_2}{z}, \ga_1; \boldsymbol{c}^{(1)}_4\right]\cdot \mathscr{M}_{n+3}\left[\phi; \xi_1, \psi_1, \eta_1, \frac{\ga_2}{z}; \boldsymbol{c}^{(2)}_4\right] \right. \\
 \phantom{\mathscr{K}\left\lbrace \begin{matrix} 0 \\ n+4 \end{matrix} \right\rbrace =\times}{} - \left. \mathscr{M}_{n+3}\left[\phi; \xi_1, \psi_1, \frac{\eta_2}{z},\frac{\ga_2}{z}; \boldsymbol{c}^{(3)}_4\right]\cdot \mathscr{M}_{n+3}\big[\phi; \xi_1, \psi_1, \eta_1, \ga_1; \boldsymbol{c}^{(4)}_4\big] \right),
\\
				\mathscr{K}\left\lbrace \begin{matrix} n+4 \\ 0 \end{matrix} \right\rbrace = \frac{(-1)^{n+1}}{\mathscr{E}_{n+2}[\phi;\ga_1,\eta_1;a_4]} \\
\phantom{\mathscr{K}\left\lbrace \begin{matrix} n+4 \\ 0 \end{matrix} \right\rbrace =}{}\times \left( \mathscr{M}_{n+3}\left[\phi; \xi_1,\frac{\psi_2}{z}, \eta_1, \ga_1; \boldsymbol{d}^{(1)}_4\right]\cdot \mathscr{M}_{n+3}\left[\phi; \frac{\xi_2}{z}, \psi_1, \eta_1, \ga_1; \boldsymbol{d}^{(2)}_4\right] \right. \\
\phantom{\mathscr{K}\left\lbrace \begin{matrix} n+4 \\ 0 \end{matrix} \right\rbrace =\times }{}- \left. \mathscr{M}_{n+3}\big[\phi; \xi_1, \psi_1, \eta_1, \ga_1; \boldsymbol{d}^{(3)}_4\big]\cdot \mathscr{M}_{n+3}\left[\phi; \frac{\xi_2}{z}, \frac{\psi_2}{z}, \eta_1, \ga_1; \boldsymbol{d}^{(4)}_4\right] \right),
		\end{gather*}
		with
		\begin{gather*}
			\boldsymbol{a}^{(1)}_4 = \{a_1, \psi_{2,0}, a_7, \eta_{2,n+2} \},\qquad	 \boldsymbol{b}^{(2)}_4 = \{a_5, \xi_{2,0}, a_3, \ga_{2,0}\}, \qquad		
			\boldsymbol{c}^{(3)}_4 = \{\ga_{2,n+2}, a_2, \eta_{2,0}, a_8\},		\\
	\boldsymbol{d}^{(4)}_4 = \{\xi_{2,n+2},a_6, \psi_{2,n+2}, a_4\},	\qquad
			\boldsymbol{a}^{(3)}_4 = \boldsymbol{c}^{(1)}_4 = \{a_1, a_2, \eta_{2,0}, \eta_{2,n+2}\},\\
 \boldsymbol{a}^{(4)}_4 = \boldsymbol{d}^{(1)}_4 = \{a_1, \psi_{2,0}, \psi_{2,n+2}, a_4\}, \qquad \boldsymbol{b}^{(3)}_4 = \boldsymbol{c}^{(2)}_4 = \{\ga_{2,n+2}, a_2, a_3,\ga_{2,0}\}, \\
 \boldsymbol{b}^{(4)}_4 = 	\boldsymbol{d}^{(2)}_4 = \{\xi_{2,n+2}, \xi_{2,0}, a_3, a_4\}, \qquad
			\boldsymbol{a}^{(2)}_4 = \boldsymbol{b}^{(1)}_4 = \boldsymbol{c}^{(4)}_4 = \boldsymbol{d}^{(3)}_4 = \{a_1, a_2, a_3, a_4\}.
		\end{gather*}
		The above equations together with \eqref{M in terms of semi-framed stuff} and the results of Section \ref{sec semi-framed}, provides the needed pathway to compute the asymptotics of the two-framed Toeplitz determinant \eqref{two framed}.
		
\appendix

		\section[Solution of the Riemann--Hilbert problem for BOPUC\\ with Szeg\H{o}-type symbols]{Solution of the Riemann--Hilbert problem \\
for BOPUC with Szeg\H{o}-type symbols}\label{Appendices}

		The following Riemann--Hilbert problem for BOPUC is due to J.~Baik, P.~Deift and K.~Johansson~\cite{BDJ}. Find $X(z;n)$ satisfying
		\begin{itemize}\itemsep=0pt
			\item {\rm RH-X1:} $X(\cdot;n)\colon\C\setminus \T \to \C^{2\times2}$ is analytic,
			\item {\rm RH-X2:} The limits of $X(\ze;n)$ as $\ze$ tends to $z \in \T $ from the inside and outside of the unit circle exist, and are denoted $X_{\pm}(z;n)$ respectively and are related by
			\[
				X_+(z;n)=X_-(z;n)\begin{bmatrix}
					1 & z^{-n}\phi(z) \\
					0 & 1
				\end{bmatrix}, \qquad z \in \T,
			\]
			
			\item {\rm RH-X3:} $X(z;n)=\big( I + O\big(z^{-1}\big) \big) z^{n \sigma_3}$ as $z \to \infty$,
		\end{itemize}
		(see \cite{D, DIK}). Below we show the standard steepest descent analysis to asymptotically solve this problem for a Szeg\H{o}-type symbol. We first normalize the behavior at $\infty$ by defining \begin{gather*}
			T(z;n) := \begin{cases}
				X(z;n)z^{-n\sigma_3}, & |z|>1, \\
				X(z;n), & |z|<1.
			\end{cases}
		\end{gather*}
		The function $T$ defined above satisfies the following RH problem:
				\begin{itemize}\itemsep=0pt
			\item {\rm RH-T1:} $T(\cdot;n) \colon\C\setminus \T \to \C^{2\times2}$ is analytic,
			\item {\rm RH-T2:} $T_{+}(z;n)=T_{-}(z;n)\begin{bmatrix}
				z^n & \phi(z) \\
				0 & z^{-n}
			\end{bmatrix}$, $z \in \T$,
			\item {\rm RH-T3:} $T(z;n)=I+O(1/z)$, $ z \to \infty$,
		\end{itemize}
		So $T$ has a highly-oscillatory jump matrix as $n \to \infty$. The next transformation yields a Riemann--Hilbert problem, normalized at infinity, having an exponentially decaying jump matrix on the \textit{lenses}. Note that we have the following factorization of the jump matrix of the $T$-RHP:
\begin{align*}
			\begin{bmatrix}
				z^n & \phi(z) \\
				0 & z^{-n}
			\end{bmatrix} &= \begin{bmatrix}
				1 & 0 \\
				z^{-n}\phi(z)^{-1} & 1
			\end{bmatrix}\begin{bmatrix}
				0 & \phi(z) \\
				-\phi(z)^{-1} & 0
			\end{bmatrix}\begin{bmatrix}
				1 & 0 \\
				z^{n}\phi(z)^{-1} & 1
			\end{bmatrix}\\
& \equiv J_{0}(z;n)J^{(\infty)}(z)J_{1}(z;n).
		\end{align*}
		Now, we define the following function: \begin{gather*}
			S(z;n):=\begin{cases}
				T(z;n)J^{-1}_{1}(z;n), & z \in \Om_1, \\
				T(z;n)J_{0}(z;n), & z \in \Om_2, \\
				T(z;n), & z\in \Om_0\cup \Om_{\infty}.
			\end{cases}
		\end{gather*}
		Also introduce the following function on $\Ga_S := \Ga_0 \cup \Ga_1 \cup \T$:
 \begin{gather*}
			J_{ S}(z;n)=\begin{cases}
				J_{1}(z;n), & z \in \Ga_0, \\
				J^{(\infty)}(z), & z \in \T, \\
				J_{0}(z;n), & z \in \Ga_1. \\
			\end{cases}
		\end{gather*} 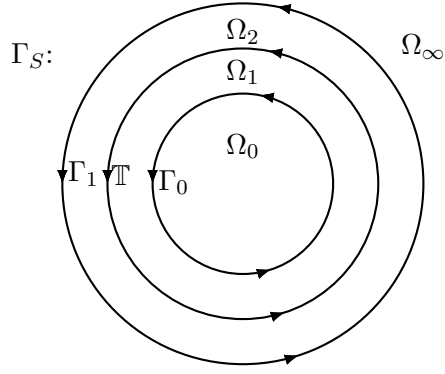
\begin{figure}
			\centering
			
			\begin{tikzpicture}[scale=0.4]

				\draw[ ->-=0.22,->-=0.5,->-=0.8,thick] (-3,0) circle (4.5cm);
				
				\draw[ ->-=0.22,->-=0.5,->-=0.8,thick] (-3,0) circle (3cm);
				
				\draw[ ->-=0.22,->-=0.5,->-=0.8,thick] (-3,0) circle (6cm);
				
				
				
				\node at (-3,4.45) [below] {$\Om_1$};
				
				\node at (-3,5.8) [below] {$\Om_2$};
				
				\node at (-8.3,-0.4) [above] {$\Ga_1$};
				
				
				
				\node at (-6.4,0.3) [left] {$\T$};
				
				\node at (-5.3,0.8) [below] {$\Ga_0$};

				\node at (-3,2) [below] {$\Om_0$};
				
				
				
				\node at (3,5.3) [below] {$\Om_{\infty}$};
				
				\node at (-10,5) [below] {$\Ga_S$:};
			\end{tikzpicture}

			\caption{Opening of lenses: the jump contour for the $S$-RHP.}
			\label{S_contour}
		\end{figure}
		Therefore, we have the following Riemann--Hilbert problem for $S(z;n)$:
		\begin{itemize}\itemsep=0pt
			\item {\rm RH-S1:} $ S(\cdot;n)\colon\C\setminus \Ga_S \to \C^{2\times2}$ is analytic,
			\item {\rm RH-S2:} $S_{+}(z;n)=S_{-}(z;n) J_{S}(z;n)$, $z \in \Ga_S$,
			\item {\rm RH-S3:} $ S(z;n)=I+O(1/z)$, as $z \to \infty$.
		\end{itemize}
		Note that the matrices $J_0(z;n)$ and $J_1(z;n)$ tend to the identity matrix uniformly on their respective contours, exponentially fast as $n \to \infty$.	We are looking for a piecewise analytic function $P^{(\infty)}(z)\colon \C \setminus \T \to \C^{2\times2}$ such that
\begin{itemize}\itemsep=0pt
			\item {\rm RH-Global1:} $P^{(\infty)}$ is holomorphic in $\C \setminus \T$,
			\item {\rm RH-Global2:} for $z\in \T$, we have \begin{equation}\label{44}
				P_+^{(\infty)}(z)=P_-^{(\infty)}(z) \begin{bmatrix}
					0 & \phi(z) \\
					- \phi^{-1}(z) & 0
				\end{bmatrix},
			\end{equation}
			\item {\rm RH-Global3:} $P^{(\infty)}(z)=I+O(1/z)$, as $z \to \infty$.
		\end{itemize}
		We can find a piecewise analytic function $\al$ which solves the following scalar multiplicative Riemann--Hilbert problem \begin{equation}\label{47}
			\al_+(z)=\al_-(z)\phi(z), \qquad z \in \T.
		\end{equation}
		By Plemelj--Sokhotski formula, we have \begin{equation}\label{45}
			\al(z)=\exp \left[ \frac{1}{2 \pi {\rm i} } \int_{\T} \frac{\ln(\phi(\tau))}{\tau-z}{\rm d}\tau \right],
		\end{equation}
		Now, using (\ref{47}), we have the following factorization: \begin{gather*}
			\begin{bmatrix}
				0 & \phi(z) \\
				- \phi^{-1}(z) & 0
			\end{bmatrix}= \begin{bmatrix}
				\al_-^{-1}(z) & 0 \\
				0 & \al_-(z)
			\end{bmatrix} \begin{bmatrix}
				0 & 1 \\
				- 1 & 0
			\end{bmatrix} \begin{bmatrix}
				\al^{-1}_+(z) & 0 \\
				0 & \al_+(z)
			\end{bmatrix}.
		\end{gather*}
		So, the function \begin{equation}\label{48}
			P^{(\infty)}(z) :=\begin{cases}
				\begin{bmatrix}
					0 & \al(z) \\
					-\al^{-1}(z) & 0
				\end{bmatrix}, & |z|<1, \\[14pt]
				\begin{bmatrix}
					\al(z) & 0 \\
					0 & \al^{-1}(z)
				\end{bmatrix}, & |z|>1,
			\end{cases}
		\end{equation}
		satisfies (\ref{44}). Also, by the properties of the Cauchy integral, $P^{(\infty)}(z)$ is holomorphic in $\C \setminus \T$. Moreover, $\al(z)=1+O\big(z^{-1}\big)$, as $z\to \infty$, and hence \begin{gather*}
			P^{(\infty)}(z)=I+O(1/z), \qquad z \to \infty.
		\end{gather*}
		Therefore, $P^{(\infty)}$ given by (\ref{48}) is the unique solution of the global parametrix Riemann--Hilbert problem. Let us now consider the ratio
		\begin{equation*}
			R(z;n):= S(z;n) \big[ P^{(\infty)}(z) \big]^{-1}.
		\end{equation*}
		We have the following Riemann--Hilbert problem for $R(z;n)$:
		\begin{itemize}\itemsep=0pt
			\item {\rm RH-R1:} $ R$ is holomorphic in $\C\setminus (\Ga_0 \cup \Ga_1)$,
			\item {\rm RH-R2:} $R_{+}(z;n)=R_{-}(z;n) J_{R}(z;n)$, $ z \in \Ga_0 \cup \Ga_1 =: \Sigma_R$,
			\item {\rm RH-R3:} $ R(z;n)=I+O(1/z) $ as $z \to \infty$.
		\end{itemize}
		This Riemann--Hilbert problem is solvable for large $n$ (see \cite{Deiftetal2, Deiftetal}), and $R(z;n)$ can be written~as
		\begin{gather*}
			R(z;n) = I + R_1(z;n) + R_2(z;n) + R_3(z;n) + \cdots, \qquad n \geq n_0,
		\end{gather*}
		where $R_k$ can be found recursively. Indeed,
		\begin{gather*}
			R_k(z;n) = \frac{1}{2\pi \ii}\int_{\Sigma_R} \frac{\left[ R_{k-1}(\mu;n)\right]_- \left( J_R(\mu;n)-I\right) }{\mu-z}\dd \mu, \qquad z \in \C \setminus \Sigma_R, \qquad k\geq1.
		\end{gather*}	
		It is easy to check that $R_{2\ell}(z;n)$ is diagonal and $R_{2\ell+1}(z;n)$ is off-diagonal; $\ell \in \N \cup \{0\}$, and that
		\begin{gather}\label{R_k's are small}
			R_{k,ij}(z;n) = O \left( \frac{\rho^{-kn}}{1+|z|} \right), \qquad n \to \infty, \qquad k \geq 1,
		\end{gather}
		uniformly in $z\in \C \setminus \Sigma_R$, where $\rho$ \big(resp.\ $\rho^{-1}$\big) is the radius of $\Ga_1$ (resp.\ $\Ga_0$). Let us compute~${R_1(z;n)}$; we have
		\begin{gather*}\label{JR-I}
			J_R(z)-I = \begin{cases}
	P^{(\infty)}(z) \begin{bmatrix}
					0 & 0 \\
					z^n \phi^{-1}(z) & 0 \end{bmatrix} \left[ P^{(\infty)}(z) \right]^{-1}, & z \in \Ga_0, \\[14pt]
			 	P^{(\infty)}(z) \begin{bmatrix}
					0 & 0 \\
					z^{-n} \phi^{-1}(z) & 0 \end{bmatrix} \left[ P^{(\infty)}(z) \right]^{-1}, & z \in \Ga_1,
			\end{cases} \\
\phantom{J_R(z)-I }{} = \begin{cases}
				\begin{bmatrix}
					0 & -z^{n} \phi^{-1}(z)\al^{2}(z) \\
					0 & 0 \end{bmatrix}, & z \in \Ga_0, \\[14pt]
 				\begin{bmatrix}
					0 & 0 \\
					z^{-n} \phi^{-1}(z)\al^{-2}(z) & 0 \end{bmatrix}, & z \in \Ga_1.
			\end{cases}
		\end{gather*}
		Therefore,
		\begin{gather}\label{R1}
			R_1(z;n)= \begin{bmatrix}
				0 & -\tfrac{1}{2\pi {\rm i}}\int_{\Ga_0} \tfrac{\tau^{n} \phi^{-1}(\tau)\al^{2}(\tau)}{\tau-z}{\rm d}\tau \\
 \frac{1}{2\pi {\rm i}}\int_{\Ga_1} \tfrac{\tau^{-n} \phi^{-1}(\tau)\al^{-2}(\tau)}{\tau-z}{\rm d}\tau
				& 0 \end{bmatrix}.
		\end{gather}
		If we trace back the Riemann--Hilbert problems $R \mapsto S \mapsto T \mapsto X $, we will obtain
		\begin{gather}\label{X in terms of R exact}
			X(z;n) = R(z;n)\begin{cases} \begin{bmatrix} \al(z) & 0 \\ 0 & \al^{-1}(z) \end{bmatrix}z^{n\sigma_3}, & z \in \Om_{\infty}, \\[14pt]
				\begin{bmatrix} \al(z) & 0 \\ -z^{-n}\al^{-1}(z) \phi^{-1}(z) & \al^{-1}(z) \end{bmatrix}z^{n\sigma_3}, & z \in \Om_{2}, \\[14pt]
				\begin{bmatrix} z^{n}\al(z) \phi^{-1}(z) & \al(z) \\ -\al^{-1}(z) & 0 \end{bmatrix}, & z \in \Om_{1}, \\[14pt] \begin{bmatrix} 0 & \al(z) \\ -\al^{-1}(z) & 0 \end{bmatrix}, & z \in \Om_{0},
			\end{cases}
		\end{gather}
		where for $z \in \C \setminus \Sigma_{R}$, as $n \to\infty$, we have
		\begin{gather}\label{R asymp}
			R(z;n)=\di \begin{bmatrix}
				1 + O \left(\frac{ \rho^{-2n}}{ 1+|z| }\right) & R_{1,12}(z;n)+ O \left(\frac{ \rho^{-3n}}{ 1+|z| }\right) \\
				R_{1,21}(z;n)+ O \left(\frac{ \rho^{-3n}}{ 1+|z| }\right) & 1 +O \left(\frac{ \rho^{-2n}}{ 1+|z| }\right)
			\end{bmatrix}.
		\end{gather}

		\subsection*{Acknowledgements}
		The author would like to thank Harini Desiraju, Alexander Its, Karl Liechty, and Nicholas Witte for their interest in this project and for helpful conversations. The author would like to also thank the anonymous referees for their helpful remarks and suggestions. This material is based upon work supported by the National Science Foundation under Grant No. DMS-1928930. The author gratefully acknowledges the Mathematical Sciences Research Institute, Berkeley California and the organizers of the semester-long program \textit{Universality and Integrability in Random Matrix Theory and Interacting Particle Systems} for their support in the Fall of 2021, during which part of this project
		was completed.


\LastPageEnding

\end{document}